\documentclass[11pt,letterpaper]{amsart}
\usepackage{amssymb}
\usepackage{mathrsfs, tikz}
\usepackage[all,cmtip]{xy}
\usepackage{pstricks}

\usepackage{ulem}
\usepackage{comment}

\usepackage{graphpap,color}

\definecolor{cof}{RGB}{219,144,71}
\definecolor{pur}{RGB}{186,146,162}
\definecolor{greeo}{RGB}{91,173,69}
\definecolor{greet}{RGB}{52,111,72}

\makeatletter
\newcommand*\bigcdot{\mathpalette\bigcdot@{.7}}
\newcommand*\bigcdot@[2]{\mathbin{\vcenter{\hbox{\scalebox{#2}{$\m@th#1\bullet$}}}}}
\makeatother

\def\bcd{{\bigcdot}}

\numberwithin{equation}{section}

\newtheorem{thm}{Theorem}[section]
\newtheorem{thm-defn}[thm]{Theorem-Definition}
\newtheorem{defn}[thm]{Definition}
\newtheorem{prop}[thm]{Proposition}
\newtheorem{prop-defn}[thm]{Proposition-Definition}
\newtheorem{res}[thm]{Result}
\newtheorem{conj}[thm]{Conjecture}

\newtheorem{lemma}[thm]{Lemma}
\newtheorem{cor}[thm]{Corollary}

\newtheorem{question}[thm]{Question}
\newtheorem{rem}[thm]{Remark}
\newtheorem{exam}[thm]{Example}

\newcommand{\Bl}{{\rm Bl}}

\newcommand{\id}{{\rm id}}
\newcommand{\Id}{{\rm Id}}
\newcommand{\Gr}{{\rm Gr}}

\newcommand{\CD}{\xymatrix@R=1pc@C=1pc}
\newcommand{\CDR}{\xymatrix@R=1pc}
\newcommand{\CDC}{\xymatrix@C=1pc}

\def\cB{{\mathcal B}}

\def\cL{{\mathcal L}}

\def\cP{{\mathcal P}}

\def\cG{{\mathcal G}}

\def\sF{{\mathscr F}}

\def\sV{\mathscr{V}}

\def\frb{{\fr\fb}}

\def\fB{\mathfrak{B}}

\def\fG{\mathfrak{G}}

\def\fL{\mathfrak{L}}

\def\2M{M}

\def\tn{\mathfrak{N}}
\def\fP{\mathfrak{P}}

\def\fT{\mathfrak{T}}
\def\fV{\mathfrak{V}}

\def\fb{\mathfrak{b}}

\def\fy{\mathfrak{y}}

\def\ff{\mathfrak{f}}

\def\tn{{\tilde{\bf n}}}

\def\fq{\mathfrak{q}}

\def\fr{\mathfrak{r}}
\def\fs{\mathfrak{s}}

\def\MM{{\mathbb M}}
\def\NN{{\mathbb N}}
\def\RR{{\mathbb R}}
\def\PP{{\mathbb P}}

\def\GG{{\mathbb G}}
\def\GGm{{{\mathbb G}_{\rm m}}}

\def\ZZ{{\mathbb Z}}
\def\FF{{\mathbb F}}
\def\HH{{\mathbb H}}

\def\TT{{\mathbb T}}
\def\TTb{{{\mathbb T}_{\bigcdot}}}

\def\E-bcd{{E_\bcd}}

\def\LL{{\mathbb L}}
\def\UU{{\mathbb U}}

\def\II{{\mathbb I}}
\def\JJ{{\mathbb J}}
\def\JJsrt{{\JJ^{(2),{\rm srt}}_\bcd}}

\def\AA{{\mathbb A}}

\def\whX{\widehat{X}}
\def\whGr{\widehat{\Gr}}

\def\p{{\rm p}}
\def\sgn{{\rm sgn}}

\def\wde{{\wedge^d E}}
\def\wdep{{(\wde)}}

\def\wdeb{{\wedge^d E_\bcd}}
\def\wdebp{{(\wdeb)}}

\def\hil{{\rm hil}}

\def\Bl{{\rm Bl}}

\def\lra{\longrightarrow}

\def\kk{{\bf k}}
\def\ba{{\bf a}}
\def\bb{{\bf b}}
\def\bc{{\bf c}}
\def\be{{\bf e}}

\def\bp{{\bf p}}

\def\bt{{\bf t}}

\def\bv{{\bf v}}
\def\bm{{\bf m}}
\def\bn{{\bf n}}

\def\bx{{\bf x}}
\def\by{{\bf y}}
\def\bz{{\bf z}}

\def\ua{{\underbar {\it a}}}
\def\ub{{\underbar {\it b}}}

\def\ud{{\underbar {\rm d}}}

\def\ui{{\underbar {\it i}}}

\def\uh{{\underbar {\it h}}}
\def\uk{{\underbar {\it k}}}

\def\um{{\underbar {{\mu}}}}
\def\um{{\underbar {\it m}}}
\def\uu{{\underbar {\it u}}}
\def\uv{{\underbar {\it v}}}
\def\uw{{\underbar {\it w}}}

\def\uwo{{\underbar {\it w}_{\rm o}}}
\def\uwe{{\underbar {\it w}_{\rm e}}}

\def\uwo{{\underbar {\it w}_{\rm o}}}
\def\uwe{{\underbar {\it w}_{\rm e}}}
\def\buw{{\underbar {\bf w}}}
\def\uwoe{{(\uw_{\rm o}, \uw_{\rm e})}}

\def\bko{{ {\bf k}_{\rm o}}}
\def\bke{{ {\bf k}_{\rm e}}}
\def\bkoe{{(\bko, \bke)}}

\def\pl{{\hbox{Pl\"ucker}}}
\def\vrpl{{\hbox{$\vr$-Pl\"ucker}}}

\def\hpl{{\hbox{$\vp\vr$-Pl\"ucker}}}
 \def\2{{\rm I\!I}}

\def\-{{\setminus}}

\def\vp{{\varpi}}
\def\vr{{\varrho}}

\def\sign{{\rm sgn}}
\def\rk{{\rm rank}}

\def\part{{\rm part}}

\def\sort{{\rm sort}}

\def\ze{\zeta}

\def\la{\lambda}

\def\La{\Lambda}

\def\Ladn{{\Lambda}_{d,[n]}}

\def\vi{\varphi}
\def\mh{{\rm mh}}

\def\Ladnsort{{\Lambda}^{\rm sort}_{d,[n]}}

\def\rladnsort{{\mathring\Lambda}^{\rm sort}_{d,[n]}}

\def\sgn{{\rm sgn}}
\def\vs{{\varsigma}}

\def\di{\diamond}

\def\whwp{{\widehat\wp}}

\def\wtbuw{{\widetilde\buw}}
\def\supp{{\rm supp}}

\def\barf{{\bar f}}

\def\gr{{\rm gr}}
\def\pgr{{\rm pgr}}

\def\PPMdn{{\PP(M_{dn})}}

\def\GL{{\rm GL}}
\def\PGL{{\rm PGL}}

\def\rS{{\rm S}}

\def\Aut{{\rm Aut}}
\def\fj{{\frak j}}
\def\fri{{\frak i}}

\makeatletter
\def\@tocline#1#2#3#4#5#6#7{\relax
  \ifnum #1>\c@tocdepth 
  \else
    \par \addpenalty\@secpenalty\addvspace{#2}%
    \begingroup \hyphenpenalty\@M
    \@ifempty{#4}{%
      \@tempdima\csname r@tocindent\number#1\endcsname\relax
    }{%
      \@tempdima#4\relax
    }%
    \parindent\z@ \leftskip#3\relax \advance\leftskip\@tempdima\relax
    \rightskip\@pnumwidth plus4em \parfillskip-\@pnumwidth
    #5\leavevmode\hskip-\@tempdima
      \ifcase #1
       \or\or \hskip 1em \or \hskip 2em \else \hskip 3em \fi%
      #6\nobreak\relax
    \hfill\hbox to\@pnumwidth{\@tocpagenum{#7}}\par
    \nobreak
    \endgroup
  \fi}
\makeatother

\begin{document}

\title[$\fG$-Quotients and Equations]{
$\fG$-Quotients  of  Grassmannians and Equations}

\author{Yi Hu}


\maketitle

\begin{abstract} 
Laurent Lafforgue's presentation of a Grassmannian $\Gr^{d, E}$ naturally
comes equipped with the induced action of a subtorus $\TTb$ of  $\PGL(E)$.
By investigating the defining  ideals of  $\TTb$-orbit closures  through general points of $\Gr^{d,E}$
 and studying their degenerations,
we obtain a morphsim $\fq: \FF^{d, E_\bcd} \to \HH^{d, E_{\bigcdot}}$
such that $\HH^{d, E_\bcd}$, termed the $\fG$-quotient of $\Gr^{d,E}$ by $\TTb$, is  birational to $[\Gr^{d, E}/\TTb]$, and $\fq$, 
termed  $\fG$-family of $\Gr^{d,E}$ by $\TTb$, is a family of
general $\TTb$-orbit closures  and  their degenerations. 
We  obtain a series of new results on $\HH^{d, E_{\bigcdot}}$ and
$\FF^{d, E_\bcd}$.  The main results  include:

\noindent
$\di$ The $\fG$-quotient $\HH^{d, E_\bcd}$ 
by construction is embedded in the product 
$$\prod_{\ba \in \JJ_\bcd} \PP_\ba \times \prod_{\buw \in  \JJsrt} \PP_\buw$$
of projective spaces $\PP_\ba$ and
$\PP_\buw$ with explicit  finited sets $\JJ_\bcd$ and $\JJsrt$.

\noindent
$\di$ The  $\fG$-family $\FF^{d, E_\bcd}$ by construction is 
 embedded in the product 
 $$\PP(\wedge^d E) \times  \prod_{\ba \in \JJ_\bcd} \PP_\ba 
\times \prod_{\buw \in  \JJsrt} \PP_\buw$$
 and its main component $\FF^{d, E_\bcd}_*$
is isomorphic to the blowup of
 $\Gr^{d,E}$ along an explicitly described ideal $J^{d,E_\bcd}$

\noindent
$\di$   Over  characteristic zero, 
the  defining equations of $\FF^{d, E}_*$ 
by the maximal torus
 as a closed subscheme of $\PP(\wedge^d E) 
\times \prod_{\buw \in  \JJsrt} \PP_\buw$
  consist of explicit $\wp$-binomials
and linearized $\pl$ relations together with 
neatly structured $\frb$-binomials and $\hpl$ relations. 

\noindent
$\di$  Over characteristic zero, 
the defining equations of the $\fG$-quotient $\HH^{d, E}$
of $\Gr^{d,E}$ by the maximal torus
as a closed subscheme of $\prod_{\buw \in  \JJsrt} \PP_\buw$
consist of explicitly characterized $\vr$-binomials, linearized $\pl$ relations,
 and $\vrpl$ relations.

All the relations aforementioned are 
in remarkably neat and well-structured forms.
 
   
 
   \end{abstract}

\tableofcontents

\section{Introduction}

In inertial thinking, one needs  arbitrary relations when  treating arbitrary singularities.
Our point  instead is that we can commence with 
a set of considerably simpler  standardized equations for singularities. 
To furnish such standard equations for arbitrary singularities, Mn\"ev's universality offers a platform to begin with.
By the Gelfand-MacPherson correspondence,  Lafforgue's version of Mn\"ev's universality [Theorem I.14, \cite{La03}] (cf. \cite{LV12})
 asserts that the quotient  by the maximal torus  $\GG_m^{n-1}$ of 
a thin Schubert cell of 
the Grassmannian $\Gr^{3,n}$ with varying $n$  can possess singularity of any given type defined over $\ZZ$. 

In this paper, we construct a morphism between projective varieties
 $$\fq: \FF^{3,n} \lra \HH^{3,n}$$
such that $\HH^{3,n}$, termed $\fG$-quotient of $\Gr^{3,n}$, is a birational model of the quotient 
$[\Gr^{3,n}/\GG_m^{n-1}]$  
 and $\fq$ is a family of  
general $\GG_m^{n-1}$-orbit closures in $\Gr^{3,n}$ and their degenerations.
The scheme $\FF^{3,n}$, termed $\fG$-family of $\Gr^{3,n}$, by construction, is  embedded
in a product of projective spaces  coming equipped with
explicit defining multi-homogeneous equations. 
The embedding is described in Theorem \ref{ZasBlowup-intro}.
The equations for the main component $\FF^{3,n}_*$ of $\FF^{3,n}$
are described in Theorem \ref{eqs-F-intro}.
The relations  in Theorem \ref{eqs-F-intro}  
are all in neat forms. These relations naturally induce explicit defining equations for all boundary strata,
which, according to Mn\"ev's universality, posses arbitrary singularity types. That being said, 
 we hope that these relations will become standard equations for singularities.
Here, the open subset 
$\UU^{3,n}$ of $\Gr^{3,n}$  consisting of points none of
whose $\pl$ coordinates vanish is naturally embedded in  $\FF^{3,n}$ as the open stratum
 and its complement $\FF^{3,n} \- \UU^{3,n}$ can be partitioned into a union of
boundary strata.

The $\fG$-quotient $\HH^{3,n}$, by  construction, is also  embedded
in a product of projective spaces. This is described in Theorem \ref{thm1}. 
From Theorem \ref{eqs-F-intro}, by elimination, we obtain the defining multi-homogeneous equations
for $\HH^{3,n}$ as well.
This is described in Theorem \ref{eqs-H-intro}. 


Our method in this paper works for any Grassmannian. Hence, in the main text, we work 
for all $\Gr(d,n)$.
In addition,  as in \cite{La03}, we also consider the actions on $\Gr^{d,E}$ by
the slightly more general poly-diagonal subtori $\TTb$. Our method
extends to these cases as well.
 Following Lafforgue's presentation, we will use the symbol $\Gr^{d,E}$ 
in place of $\Gr(d,n)$, where $E$ denotes a vector space of dimension $n$.

The idea of this paper extends to more general linear actions,
see Remark \ref{more to come}.

{\it Related works.}
The Chow/Hilbert quotients of Grassmannians by the maximal tori
which are also birational models of $[\Gr(d,n)/\GG_m^{n-1}]$
were introduced by M.  Kapranov  in his influential paper \cite{Kap93}.  Ours are  related.
From the analytic viewpoint, Hilbert quotients
were also studied by A. Bialynicki-Birula and A. Sommese \cite{B-BS}. 
The geometry of these quotients were well studied in   \cite{Alex15} \cite{HKT}  \cite{KT06}  
by Alexeev, Keel-Tevelev and Hacking-Keel-Tevelev.
In [Question 1.12, \cite{KT06}], Keel and Tevelev asked what the log canonical model
$\Gr_\varnothing^{3,n} \subset \overline{X}_{lc}^{3,n}$ 
 is and wondered in 1.19  whether it would give a canonical way of (partially)
resolving a boundary whose strata include all possible singularities.
Certain equations of  the Chow quotient of $\Gr(2,n)$ have been obtained  in \cite{GM10}  \cite{KT09}
 by  Gibney-Maclagan and Keel-Tevelev.
With motivations in Langlands programs,
 Faltings in special cases \cite{Faltings01} and Lafforgue in general \cite{La99, La03}  
also  introduced and studied various 
quotients of Grassmannians by poly-diagonal subtori.


\section{Statements of main results}

For any positive integer $n$, let $[n]=\{1,\cdots, n\}$. Following Lafforgue  \cite{La03}, 
suppose we have a set of vector spaces over a field $\kk$ (or, free modules), 
$E_1, \cdots, E_n$, such that $E_i$  is of dimension $r_i$ (or, of rank $r_i$)
 for some positive integer $r_i$, for all $i \in [n]$.
We let  $$E=E_1 \oplus \ldots \oplus E_n.$$ 
 The Grassmannian $\Gr^{d, E}$, defined by
$$\Gr^{d, E}=\{ \hbox{linear subspaces} \;K \hookrightarrow E \mid \dim K=d \}, $$
is a projective variety defined over $\ZZ$, for  any fixed  integer $d \in [\tn]$
where $\tn=r_1 +\cdots +r_n$. 

The group $ \Aut (E_{\bigcdot})= \Aut(E_1) \times \cdots \times \Aut (E_n)$
naturally acts on $\Gr^{d, E}$, inducing the action of its center $(\GGm)^n$,
where $E_\bcd$  denotes 
the decomposition $E=\oplus_i E_i$. 
The  action of  $(\GGm)^n$ factors to give the action of the quotient subtorus 
$\TTb=(\GGm)^n/\GGm$. 
We call $\TTb$ a poly-diagonal subtorus (of a maximal torus of $\Aut(E)/\GG_m$.

\subsection{The $\fG$-quotient $\HH^{d, E_\bcd}$ by the poly-diagonal torus} 
We let $\UU^{d,E}$ be the open subset of $\Gr^{d,E}$ consisting of points none of whose
$\pl$ coordinates are zero.
The main results of this article 
 begin with the following.

\begin{thm-defn}\label{thm1} 
There exist  two explicit finite sets $\JJ_\bcd$ and $\JJsrt$ 
and  a projective space $\PP_\ba$ for every element $\ba \in \JJ_\bcd$ and
a projective space $\PP_\buw$  for every element $\buw \in \JJsrt$
 such that the geometric quotient $\UU^{d,E}/\TTb$
 is naturally embedded in
$$\prod_{\ba \in \JJ_\bcd} \PP_\ba \times \prod_{\buw \in \JJsrt} \PP_\buw . $$
We let $\HH^{d, E_\bcd}$ be the closure of the image of $\UU^{d,E}/\TTb$ in
$\prod_{\ba \in \JJ_\bcd} \PP_\ba \times \prod_{\buw \in \JJsrt} \PP_\buw $
and call it the \emph{$\fG$-quotient} of $\Gr^{d,E}$ by $\TTb$.
\end{thm-defn}


We now  explain the notation in this theorem.
Consider the canonical decomposition
$$\wde=\bigoplus_{\ui  \in \rS_\bcd} \wedge^{\ui} E_{\bigcdot}, \;\; \hbox{where}$$
 $$\rS_\bcd=\{ \ui =(i_1,\cdots, i_n) \in \NN^n \mid i_1 + \cdots + i_n= d, \;
0\le i_\alpha \le r_\alpha, \alpha \in [n]\}, \;\; \hbox{and}$$
 $$\wedge^{\ui} E_{\bigcdot}=\wedge^{i_1}E_1\otimes \cdots \otimes \wedge^{i_n} E_{i_n}.$$
It gives rise to the  embedding 
\begin{equation}\label{pl-embed}
\Gr^{d, E} \hookrightarrow \PP(\wedge^d E)=\{(\bp_\ui)_{\ui \in \rS_\bcd} \in \GGm 
\backslash (\prod_{\ui} \wedge^{\ui} E_{\bigcdot} \- \{0\} )\},
\end{equation}
$$K \lra [\wedge^d K].$$

If we let $E_\alpha=F_{\tn_{\alpha-1} +1} \oplus \cdots \oplus F_{\tn_{\alpha -1}+r_\alpha}$
be a decomposition of $E_\alpha$ into the direct sum of vector spaces 
of dimension  1 (or, free $\ZZ$-modules of rank 1) for every $\alpha \in [n]$, 
where $\tn_0=0$ and $\tn_\alpha= r_1 +\cdots + r_\alpha$,
then $$E= F_1 \oplus \cdots \oplus F_\tn$$
and we have the canonical decomposition
\begin{equation}\label{ususal-pl-decom}
\wde=\bigoplus_{\uu =(u_1,\cdots, u_d) \in \II_{d,[\tn]}} F_{u_1}\wedge \cdots \wedge F_{u_d}
\end{equation}
where  $\II_{d,[\tn]}=\{\uu=(u_1,\cdots, u_d) \in \NN^d \mid 1 \le u_1<\cdots <u_d \le \tn \}$.
This way, the embedding \eqref{pl-embed} can be presented in the familiar way:
\begin{equation}\label{pl-embed-0}
\Gr^{d, E} \hookrightarrow \PP(\wedge^d E)=\{(p_\uu)_{\uu \in \II_{d,\tn}} \in \GGm 
\backslash (\wde \- \{0\} )\},
\end{equation}
where $(p_\uu)$ are the $\pl$ coordinates. 

 For any element $\ui \in \rS_\bcd$, $\wedge^{\ui} E_{\bigcdot}$ canonically 
decomposes into the direct sum of some summands  of \eqref{ususal-pl-decom}.
We then set  
\begin{equation}\label{IIui-intr}
\II_\ui=\{\uu=(u_1,\cdots ,u_d) \mid 
 \hbox{$F_{u_1}\wedge \cdots \wedge F_{u_d}$ is a direct summand of $\wedge^{\ui} E_{\bigcdot}$}\}.
 \end{equation}
 Thus, $\II_{d,\tn}=\bigsqcup_{\ui \in \rS_\bcd} \II_\ui$.
Hence, the coordinate $\bp_\ui$ of \eqref{pl-embed} can be expressed as
\begin{equation}\label{bpui=}
\bp_{\ui}=(p_\uu)_{\uu \in \II_\ui}.
\end{equation}

Let $\JJ_\bcd$  be the set of all weakly increasing strings of length $d$ over the alphabet $[n]$
having at most $r_\alpha$ occurrences of the letter $\alpha \in [n]$. By 14.A of \cite{St}, 
we have  bijections 
 \begin{equation}\label{fri-intr}  \fri:  \JJ_\bcd \lra \rS_\bcd, \;\;\;\;
 \fj:  \rS_\bcd \lra \JJ_\bcd
\end{equation}
which are inverse to each other.
Following Sturmfels \cite{St}, 
we call a fixed unordered pair $$\buw=\bkoe \in \JJ^{(2)}_\bcd :=(\JJ_\bcd \times \JJ_\bcd)/\ZZ_2, \;\; \hbox{with}$$
$$\hbox{$\bko=(k_1 k_3 \cdots k_{2d-1})$ and $\bke=(k_2k_4 \cdots k_{2d})$}$$
  a \emph{\it sorted pair} if $$k_1 \le k_2 \le k_3 \le k_4 \le \cdots \le k_{2d-1} \le k_{2d}.$$
 In such a case, we also say that the string $k_1  k_2 \cdots  k_{2d-1}  k_{2d}$ is sorted.
 
 We let $\JJsrt$ be the subset of all sorted pairs $\bkoe$ of $\JJ^{(2)}_\bcd$.

  For any pair $(\ba,\bb) \in \JJ^{(2)}_\bcd $ with $\ba=(a_1 a_2\cdots a_d)$ and $\bb=(b_1b_2\cdots b_d)$,
we define $$\ba \vee\bb = a_1b_1a_2b_2\cdots a_db_d.$$
 We let $``$\sort$"$ denote the operator that takes a pair or a string and sorts it into
 a sorted pair or a sorted string.

Now, for any $\ba \in \JJ_\bcd$, we let $\PP_\ba$ be the projective space 
such that it comes equipped with the homogeneous coordinates \begin{equation}\label{yba-intr}
  y_\ba=[y_\uu]_{\uu  \in \II_{\fri (\ba)}}  \end{equation}
  where $\fri: \JJ_\bcd \to \rS_\bcd$ is the bijection in \eqref{fri-intr} 
   and  $ \II_{\fri (\ba)}$ is as defined in \eqref{IIui-intr}.

Next, for any  $\buw=\bkoe \in \JJsrt$, we define
\begin{equation}\label{La-buw}
\La_\buw=\{(\uu, \uv) \in \II_{d,n}^{(2)}
\mid \uu \in \II_{\fri (\ba)}, \uv \in \II_{\fri (\bb)},
 (\ba,\bb) \in \JJ^{(2)}_\bcd,  {\rm sort}(\ba \vee \bb)=\bko \vee \bke \}.
\end{equation}
Corresponding to $\Lambda_{\buw}$, we let $\PP_\buw$
be the projective space 
such that it comes equipped with the homogeneous coordinates \begin{equation}\label{ybw-intr}
 y_\buw=[x_{(\uu, \uv)}]_{(\uu,\uv) \in \Lambda_{\buw}}.
\end{equation}

 This sets up all the notations in Theorem-Definition \ref{thm1}.
 To explain its structure,  we  introduce the natural rational map
\begin{eqnarray}\label{theta-intro}
 \xymatrix{
\Theta^\bcd: \; \PP(\wedge^d E) \ar @{-->}[r] & 
 \prod_{\ba \in \JJ_\bcd} \PP_\ba \times \prod_{\buw \in \JJsrt} \PP_\buw
 }
\end{eqnarray}
$$[p_\uu]_{\uu \in \II_{d,[\tn]}} \lra
 \prod_{\ba \in \JJ_\bcd} [p_\uu]_{\uu \in \II_{\fri(\ba)}} \times
 \prod_{\buw \in \JJsrt}  [p_{\uu}p_{\uv}]_{(\uu,\uv) \in \La_\buw}. $$
It restricts to $\Gr^{d,E}$ to give a rational map 
\begin{eqnarray}\label{theta-gr-intro}
 \xymatrix{
\Theta^\bcd_\Gr: \; \Gr^{d, E} \ar @{-->}[r] &
 \prod_{\ba \in \JJ_\bcd} \PP_\ba \times \prod_{\buw \in \JJsrt} \PP_\buw.
 }
\end{eqnarray}
The  maps $\Theta^\bcd$ and  $\Theta^\bcd_\Gr$
are $\TTb$-equivariant where  $\TTb$
acts trivially on the target. 
We let $$\UU(\wedge^d E) \subset \PP(\wedge^d E)$$ be the Zariski open subset 
consisting of those points whose $\pl$ coordinates $p_\uu$ with $\uu \in \II_{d,[\tn]}$
are all nonzero. The rational map $\Theta^\bcd$ obviously
restricts to a morphism on $\UU(\wedge^d E)$.


 

 
\begin{prop-defn}\label{prop1} 
The rational map $\Theta^\bcd$ decends
to a morphism on the geometric GIT quotient $\UU(\wedge^d E)/\TTb$ to provide an embedding
\begin{equation} \label{q-in-prod}
\UU(\wedge^d E)/\TTb \lra  \prod_{\ba \in \JJ_\bcd} \PP_\ba \times \prod_{\buw \in \JJsrt} \PP_\buw.
\end{equation}
We let $\HH^\wdebp$ be the closure of $\UU(\wedge^d E)/\TTb$ in 
$ \prod_{\ba \in \JJ_\bcd} \PP_\ba \times \prod_{\buw \in \JJsrt} \PP_\buw$
and call it the \emph{$\fG$-quotient}
of $\PP \wdep$ by $\TTb$.
\end{prop-defn}

Now  we let $$\UU^{d, E}=\UU\wdep \cap \Gr^{d, E} .$$
Then the rational map  $\Theta^\bcd_\Gr$
restricts to a morphism on  $\UU^{d, E}$ and descends
to a morphism on its geometric GIT quotient $\UU^{d, E}/\TTb$ to provide the embedding 
\begin{equation} \label{q-in-prod-gr}
\xymatrix{
\UU^{d, E}/\TTb \ar[r] &  \prod_{\ba \in \JJ_\bcd} \PP_\ba \times \prod_{\buw \in \JJsrt} \PP_\buw}
\end{equation}
 stated in Theorem-Definition \ref{thm1}.

\subsection{The $\fG$-family $\FF^{d, E_\bcd}$ by the poly-diagonal torus}

For any closed point $\fy$ in $\HH^\wdebp $,
we express it as 
\begin{equation}\label{express-fy}
\fy= \prod_{\ba \in \JJ_\bcd} y_\ba \times \prod_{\buw \in \JJsrt} y_\buw
\;\; \in  \;\;
\prod_{\ba \in \JJ_\bcd} \PP_\ba \times \prod_{\buw \in \JJsrt} \PP_\buw,
\end{equation} where
$y_\ba=[y_\uu]_{\uu  \in \II_{\fri (\ba)}} \in \PP_\ba$ 
as in \eqref{yba-intr} and
$ y_\buw=[x_{(\uu, \uv)}]_{(\uu,\uv) \in \Lambda_{\buw}} \in \PP_\buw$
as in \eqref{ybw-intr}.

\begin{defn}\label{defn:Ify} 
For every closed point $\fy$ of $\HH^\wdebp$ expressed as in \eqref{express-fy},
we let $I_\fy$ be the homogeneous ideal of $\kk[p_{\uu}]_{\uu \in \II_{d,[\tn]}}$ generated by 
all  binomials and (possibly) monomials of the following forms
    \begin{equation}\label{GrobnerBasis-fy}
\fB_\fy=\left\{ 
\begin{array}{ccccccc}
y_{\uv} z_\uu - y_\uu z_{\uv} \\ 
x_{(\uu',\uv')} z_{\uu} z_{\uv}- x_{(\uu,\uv)} \;   z_{\uu'} z_{\uv'} \\
\end{array}
  \; \Bigg| \;
\begin{array}{ccccc}
 \uu, \uv\in \II_{\fri(\ba)}  \\
 (\uu, \uv), (\uu',\uv') \in \La_\buw
\end{array}
 \right \},
\end{equation}
for all $\ba \in \JJ_\bcd$ and $\buw=\bkoe \in \JJsrt$.

We let $S_\fy \subset \PP(\wde)$ be the subscheme defined by $I_\fy$. Observe that
 $S_\fy \subset \Gr^{d,E} (\subset \PP(\wde))$
when  $\fy \in \HH^{d,E_\bcd} (\subset\HH^\wdebp)$.
 \end{defn}
 Here, a linear relation of $\fB_\fy$ is a monomial if  exactly one of $y_\uu$ and $y_\uv$ is zero;
 it is identically zero if both are zero. Likewise, a quadratic relation of $\fB_\fy$
  is a monomial if  exactly one of  $x_{(\uu,\uv)}$ and $x_{(\uu',\uv')}$ is zero; 
 it is identically zero if both are zero.

We introduce
$$\FF^\wdebp=\{ (\bp,  \fy) \mid \bp \in S_\fy, \fy \in  \HH^\wdebp \} 
 \subset \PP(\wedge^d E) \times   \prod_{\ba \in \JJ_\bcd} \PP_\ba \times \prod_{\buw \in \JJsrt} \PP_\buw.$$
The first projection $\fq:\FF^\wdebp \lra \HH^\wdebp$
is called the \emph{$\fG$-family} of 
$\PP(\wedge^d E)$ by $\TTb$.
The two projections from $\FF^\wdebp$ give
\begin{equation}\label{pp-cc}
\xymatrix{
\FF^\wdebp \ar[r]^{\ff} \ar[d]^{\fq} &   \PP(\wedge^d E)  \\ 
\HH^\wdebp.   } 
 \end{equation}
Suppose $\fy \in   \UU(\wde) /\TTb  \subset \HH^\wdebp$. Then, it is determined by a unique orbit
$\TTb \cdot \bp$ for some $\bp \in \UU(\wde)$.  Then, we have
$ I_\fy= I_\bp$
where $I_\bp$ is the homogenenous toric ideal in Theorem \ref{bi-eq-bp}.  
Hence,
$\fq$ is a family of  toric ideal $I_\bp$ in Theorem \ref{bi-eq-bp} and their degenerations.  We let $\fq^{-1}(\UU(\wde) /\TTb) \lra \UU(\wde) /\TTb$
be the restriction of $\fq$ over the open subset $\UU(\wde) /\TTb$.
By Corollary \ref{=hpoly}, this is flat family of projective toric varities.
We let $\FF^\wdebp_*$ denote the closure of $\fq^{-1}(\UU(\wde) /\TTb)$
in $\FF^\wdebp$ and call it the {\it main component} of $\FF^\wdebp$.

For any $\ba \in \JJ_\bcd$, we introduce
the  homogeneous ideal  
\begin{equation}\label{ideal-Ja-intr}
J_\ba^\wdebp :=\langle  p_\uu  \mid \uu \in \II_{\fri(\ba)}  \rangle \subset \kk[p_\uu]_{\uu \in \II_{d,[\tn]}}.
\end{equation}

For any $\buw=\bkoe \in \JJsrt$,  
we introduce the  homogeneous ideal  
\begin{equation}\label{ideal-Jw-intr}
J_\buw^\wdebp := \langle  p_{\uu} p_{\uv} \mid {(\uu,\uv) \in \La_\buw}  \rangle
\subset \kk[p_\uu]_{\uu \in \II_{d,[\tn]}},
\end{equation}
where $\La_\buw$ is as defined in \eqref{La-buw}.

We then let $J^\wdebp$ be the product of all the above ideals 
\begin{equation} \label{J-wdep-intro}
J^\wdebp =\prod_{\ba \in \JJ_\bcd} J_\ba^\wdebp \prod_{\buw \in \JJsrt} J_\buw^\wdebp
\subset \kk[p_\uu]_{\uu \in \II_{d,[\tn]}}.
\end{equation}

\begin{prop} \label{Z'asBlowup-intro}  
The degenerate locus of the rational map $\Theta^\bcd$ 
 is precisely the closed subscheme defined by the ideal $J^\wdebp$.
The the main component $\FF^\wdebp_*$ of the
$\fG$-family $\FF^\wdebp$ is  isomorphic to 
the blowup of $\PP\wdebp$ along the ideal $J^\wdebp$. 
Furthermore, 
$\FF^\wdebp_*$ is also isomorphic to the scheme resulted from successively blowing up $\PP\wdebp$ 
along the pullbacks of the ideals $\{J_\ba^\wdebp,  J_\buw^\wdebp \mid \ba \in \JJ_\bcd, \buw \in \JJsrt\}$,
 in any given order. 
 \end{prop}

We introduce
$$  \FF^{d,E_\bcd}=\{ (\bp,  \fy) \mid \bp \in S_\fy, \fy \in  \HH^{d, E_\bcd} \} 
 \subset \Gr^{d, E} \times   \prod_{\ba \in \JJ_\bcd} \PP_\ba \times \prod_{\buw \in \JJsrt} \PP_\buw.$$
The first projection $\fq_\Gr:\FF^{d,E_\bcd}
 \lra \HH^{d, E_\bcd}$
is called the \emph{$\fG$-family} of 
$\Gr^{d, E}$ by $\TTb$.
By construction, we have $\HH^{d, E_\bcd} \subset \HH^\wdebp$.
Then, observe that $\fq_\Gr: \FF^{d,E_\bcd} \to \HH^{d, E_\bcd}$ is the restriction 
of $\fq$ to $\HH^{d, E_\bcd}$. Thus, the diagram \eqref{pp-cc} induces
\begin{equation}\label{pp-cc-gr}
\xymatrix{
\FF^{d,E_\bcd} \ar[r]^{\ff_\Gr} \ar[d]^{\fq_\Gr} &   \Gr^{d,E}  \\ 
\HH^{d, E_\bcd}.   } 
 \end{equation}
We let $\fq_\Gr^{-1}(\UU^{d,E} /\TTb) \lra \UU^{d,E} /\TTb$
be the restriction of $\fq_\Gr$ over the open subset $\UU^{d,E} /\TTb$.
Again by Corollary \ref{=hpoly}, this is flat family of projective toric varities.
Let $\FF^{d,E_\bcd}_*$ denote the closure of $\fq_\Gr^{-1}(\UU^{d,E} /\TTb)$
in $\FF^{d,E_\bcd}$ and call it the main component of $\FF^{d,E_\bcd}$.

 Similar to the case of $\FF^\wdebp$,   for any $\ba \in \JJ_\bcd$, we define
\begin{equation} \label{ideal-Ja-gr-intr} 
J_\buw^{d, E_\bcd} =\langle  p_\uu  + I_\whwp \mid \uu \in \II_{\fri(\ba)}  \rangle
\subset \kk[p_\uu]_{\uu \in \II_{d,[\tn]}}/   I_\whwp;
\end{equation}
for any $\buw \in \JJsrt$,  we define
\begin{equation} \label{ideal-Jw-gr-intr} 
J_\buw^{d, E_\bcd}=\langle  p_{\uu} p_{\uv} + I_\whwp \mid {(\uu,\uv) \in \La_\buw}  \rangle
\subset \kk[p_\uu]_{\uu \in \II_{d,[\tn]}}/   I_\whwp,
\end{equation}
where $ I_\whwp$ denotes a homogeneous ideal of $\Gr^{d,E}$
in $\kk[p_\uu]_{\uu \in \II_{d,[\tn]}}$.
Then, the product of all the above ideals is denoted 
\begin{equation} \label{J-wdep-gr-intr} 
J^{d, E_\bcd}=\prod_{\ba \in \JJ_\bcd} J_\ba^{d, E_\bcd} \prod_{\buw \in \JJsrt} J_\buw^{d, E_\bcd}
\subset \kk[p_\uu]_{\uu \in \II_{d,[\tn]}}/   I_\whwp .
\end{equation}

\begin{thm} \label{ZasBlowup-intro}  
The degenerate locus of the rational map $\Theta^\bcd_\Gr$ 
 is precisely the closed subscheme defined by the ideal $J^{d, E_\bcd}$.
The main component  $\FF^{d, E_\bcd}_*$ of the $\fG$-family $\FF^{d, E_\bcd}$ is  isomorphic to 
the blowup of $\Gr^{d, E}$ along the ideal $J^{d, E_\bcd}$. 
 Furthermore, 
$\FF^{d, E_\bcd}$ is also naturally isomorphic to the scheme resulted from 
successively blowing up $\Gr^{d,E_\bcd}$ along the pullbacks of the ideals 
$\{J_\ba^{d, E_\bcd},  J_\buw^{d, E_\bcd} \mid \ba \in \JJ_\bcd, \buw \in \JJsrt\}$,
 in any given order.
\end{thm}

\begin{rem}
Suppose one can show 
$\FF^{d, E_\bcd}_* \to \HH^{d, E_\bcd}$ is flat, then 
there will be an injective birational morphism from 
$\HH^{d, E_\bcd}$ onto the Hilbert quotient of $\Gr^{d,E}$ by $\TTb$.
We were informed that
when the torus $\TTb$ is 1-dimensional, this has been proved
in Theorem 1.8 of \cite{FW}.
\end{rem}

\subsection{The $\HH^{d, E}$ and $\FF^{d, E}$ by the maximal torus} 
\label{max torus}
\smallskip

We now focus on the important special case of
the $\fG$-quotient $\HH^{d,E}$ and  $\fG$-family 
$\FF^{d,E}$ by the maximal torus $\TT$.
In this case, we have the decomposition $E=\oplus_{i \in [n]} E_i$ with
$\dim E_i=1$ (or $\rk \; E_i=1$)  for all $i \in [n]$, hence, $\tn=n$.
We follow all the previous notations, but {\it will often  drop  the sub- or sup-index  $\bcd$.}

In this case, the corresponding $\rS_\bcd$ is precisely
the set of   lattice points in  the polytope
$$\Delta^{d,[n]} =\{(x_1, \cdots, x_n) \in {\mathbb R}^n \mid \sum_{i=1}^n x_i = d,
\;  0 \le x_i \le 1, \; \hbox{for all $1 \le i \le n$} \}$$
which  is called the hypersimplex of type $(d,[n])$;
the set $\JJ_\bcd$ is simply
$$\II_{d,[n]}=\{(u_1, \cdots, u_d) \in \NN^d \mid 1 \le u_1 < \cdots < u_d \le n\}.$$

First, note that for any $\ba \in \JJ_\bcd=\II_{d,[n]}$, $\PP_\ba$ is a single point. Hence, we can discard all these spaces from consideration.
Next,   $\JJsrt$ is the subset of sorted pairs in $\II_{d,[n]}^{(2)}$. 
For any $\buw = (\uwo, \uwe) \in \JJsrt \subset \II_{d,[n]}^{(2)}$,  $\PP_\buw$ is not a single point, i.e.,
$\La_\buw$ is not a singleton if and only if $\uwo \vee \uwe$ contains 
 at least  $d+2$ mutually distinct integers. 

{\it Thus, we let $\La^\sort_{d,[n]}$ be the set of all sorted pairs  $ \buw=(\uwo,\uwe) \in \II_{d,[n]}^{(2)}$
such that $\uwo\vee\uwe$ contains at least  $d+2$ mutually distinct integers. }

Hence,  Theorem \ref{thm1} and Proposition \ref{prop1} in this case may be restated as

\begin{cor}\label{cor1-e}  There exist two natural closed subschemes of
 $\prod_{\buw \in \La^\sort_{d,[n]}} \PP_\buw $
 $$  \xymatrix{
\HH^{d, E}   
 \ar @{^{(}->}[r] & \HH^\wdep 
  \ar @{^{(}->}[r] &
  \prod_{\buw \in \La^\sort_{d,[n]}} \PP_\buw }.$$
\end{cor}

For any $\ba \in \JJ_\bcd=\II_{d, [n]}$, $J^\wdep_\ba$ is a principal ideal, so is $J^{d,E}_\ba$. Hence,  we  discard all these ideals from the statements of
Proposition \ref{Z'asBlowup-intro}  and Theorem \ref{ZasBlowup-intro}.
As for  $\buw \in \Ladnsort$, we  have
\begin{equation}\label{ideal-Jw-intro} \nonumber
J_\buw^\wdep=\langle  p_\uu p_\uv  \mid {(\uu,\uv) \in \Lambda_\buw}  \rangle 
\subset \kk[p_\uu]_{\uu \in \II_{d,[n]}},
\end{equation}
\begin{equation}\label{ideal-Jw-gr-intro} \nonumber
J_\buw^{d,E}=\langle  p_\uu p_\uv  + I_\whwp \mid {(\uu,\uv) \in \Lambda_\buw}  \rangle 
\subset \kk[p_\uu]_{\uu \in \II_{d,[n]}}/I_\whwp.
\end{equation}
Further, by discarding all the principal ideals $J^\wdep_\ba, J^{d,E}_\ba, \ba \in \JJ_\bcd=\II_{d, [n]}$, 
we  can let
\begin{equation} \label{J-intro} \nonumber
J^\wdep =\prod_{\buw \in \La^\sort_{d,[n]}} J_\buw^\wdep \subset \kk[p_\uu]_{\uu \in \II_{d,[n]}}, \;\;\;
J^{d,E} =\prod_{\buw \in \La^\sort_{d,[n]}} J_\buw^\wdep \subset \kk[p_\uu]_{\uu \in \II_{d,[n]}}/I_\whwp.
\end{equation}

The rational maps $\Theta^\bcd$ and $\Theta^\bcd_\Gr$ in this case take the following forms.
\begin{equation}\label{theta-intr-e} 
 \xymatrix{
\Theta: \; \PP(\wedge^d E) \ar @{-->}[r] & 
 \prod_{\buw \in \Ladnsort} \PP_\buw, \;\;\;\;
\Theta_\Gr: \; \Gr^{d, E} \ar @{-->}[r] &
\prod_{\buw \in \Ladnsort} \PP_\buw. } 
\end{equation}

Then, Proposition \ref{Z'asBlowup-intro}  and Theorem \ref{ZasBlowup-intro} can be restated as follows.

\begin{cor} \label{cor2-e}  The degenerate locus of the rational map $\Theta$ 
 is precisely the closed subscheme defined by the ideal $J^\wdep$.
The main component $\FF^\wdep_*$ of the $\fG$-family $\FF^\wdep$ is  isomorphic to the blowup of $\PP\wdep$ along the ideal $J^\wdep$. In this case, we have the  closed embedding
 \begin{eqnarray}\label{embedF-intr-e}
 \xymatrix{
 \FF^\wdep\ar @{^{(}->}[r]  & \PP(\wedge^d E) \times  \prod_{\buw \in \Ladnsort} \PP_\buw.}
 \end{eqnarray}
Furthermore, 
$\FF^\wdep_*$ is also isomorphic to the scheme resulted from successively blowing up $\PP\wdep$ 
along the pullbacks of the ideals $\{J_\buw^\wdep \mid \buw \in \Ladnsort\}$,
 in any given order. 

Similarly, the degenerate locus of the rational map
$\Theta_\Gr$ 
 is precisely the closed subscheme defined by the ideal $J^{d, E}$.
The main component $\FF^{d, E}_*$ of $\fG$-family $\FF^{d, E}$ is  isomorphic to the blowup of $\Gr^{d,E}$ along the ideal $J^{d,E}$.  In this case, 
 we have the following closed embedding
 \begin{eqnarray}\label{embedF-gr-intr-e}
 \xymatrix{
 \FF^{d,E} \ar @{^{(}->}[r]  & \Gr^{d,E} \times  \prod_{\buw \in \Ladnsort} \PP_\buw.}
 \end{eqnarray}
Furthermore, $\FF^{d, E}_*$ is also isomorphic to the scheme resulted from 
successively blowing up $\Gr^{d,E}$ along the pullbacks of the ideals 
$\{J_\buw^{d, E} \mid  \buw \in \Ladnsort\}$,
 in any given order.  
\end{cor}

\subsection{Equations defining $\HH^{d, E}$ and $\FF^{d, E}_*$}  

We now describe the defining relations for the two
embeddings \eqref{embedF-intr-e} and \eqref{embedF-gr-intr-e}.
More technical details regarding these topics are in \S \ref{sec:eqForZ}.

Set $R_0=\kk[p_\ua]_{\ua \in \II_{d,[n]}}$. We let
$$R=R_0[x_{(\uu, \uv)}]_{(\uu,\uv) \in \La_{d, [n]}}$$
where $\La_{d,[n]}=\bigsqcup_{\buw \in \La^\sort_{d,[n]}} \La_\buw$.
Then, we define the homomorphism
\begin{equation}\label{vi-intro}
\vi: \; R \lra R_0,  \;\;\; \vi|_{R_0}=\id_{R_0} 
, \;\;\; x_{(\uu,\uv)} \to p_{\uu} p_{\uv}
\end{equation} 
for all  
$(\uu, \uv) \in \La_{\buw},$ and $\buw \in \Ladnsort$.    
 This corresponds to the closed embedding \eqref{embedF-intr-e}.

We then let $\vi_{\Gr}: \; R \lra R_0/I_\whwp$ be the composition
\begin{equation}\label{vi-Gr-intro}
\vi_{\Gr}: \; R \lra R_0 \lra R_0/ I_\whwp. 
\end{equation} 
where, again, $ I_\whwp$ is the homogeneous ideal of $\Gr^{d, E}$ in $\PP(\wedge^d E)$.
This corresponds to the closed embedding \eqref{embedF-gr-intr-e}.  

A polynomial of $R$ is called multi-homogeneous if
it is homogenous in each set of variables
$[x_{(\uu, \uv)}]_{(\uu,\uv) \in \La_\buw}$ for every $\buw \in \Ladnsort$.
We let $\ker^\mh \vi$, and, $\ker^\mh \vi_{\Gr}$
denote the sets of all  multi-homogeneous polynomials
 in $\ker  \vi$ and $\ker \vi_{\Gr}$, respectively.

We  determine that $\ker^\mh \vi$ contains two special kinds of binomial relations: 

$\di$
{\it The set $\cB^\wp$ of $\wp$-binomials}. 
These relations are precisely of the forms
$$p_{\uu'}p_{\uv'}x_{(\uu,\uv)} - p_\uu p_\uv x_{(\uu',\uv')}$$
where $(\uu,\uv) \ne (\uu', \uv') \in \La_{\buw}$ for some
$\buw \in \Ladnsort$. See \eqref{tildeBk} and Definition \ref{wp-bino}.

$\di$ 
{\it The set $\cB^\frb$ of $\frb$-binomials.}
A binomial $\bm-\bm'$ is a $\frb$-binomial, roughly,
 if it does not
admit a decomposition $\bm_1\bm_2 - \bm_1'\bm_2'$ such that
$\bm_1- \bm_1'$ and $\bm_2 - \bm_2'$ belong to $\ker \vi$,
and it can not be further reduced by using the above 
$\wp$-binomial
relations. See Definition \ref{defn:frb}.
Corollary \ref{cor:linear and free} states two important properties
of $\frb$-binomials: $\vr$-linearity and $\vi$-square-freeness. (Such properties of parallel binomials
in \cite{Hu2025} play important roles therein.)

\begin{prop}\label{eqs-F'-intro} 
$\FF^\wdep_*$ as a closed  subscheme  of 
$\PP(\wedge^d E) \times \prod_{\buw \in \La^\sort_{d,[n]}} 
 \PP_\buw$ is defined by the
binomials in $\cB^\wp$ and $\cB^\frb$.
\end{prop}

To suppress the issue of multivariate $\pl$ relations 
over $\ZZ$ or a field of some positive characteristic (see \cite{Abeasis}), in Theorem \ref{eqs-F-intro} 
and Theorem \ref{eqs-H-intro}, 
 we assume that the characteristic of the base field $\kk$ equals zero so that
$I_\whwp=I_\wp$, where $I_\wp$ is the ideal generated by all $\pl$ relations in $\sF$.

 As in \eqref{pl-eq-abb}, we write any $\pl$ relation as
$$F=\sum_{s \in S_F} \sgn(s) p_{\uu_s} p_{\uv_s}$$
where  $\sgn(s)$ is the $\pm$ sign associated with the term $p_{\uu_s} p_{\uv_s}$. Then, we let
\begin{equation}\label{linear-pl-intr}
L_F=\sum_{s \in S} {\rm sign(s)} x_{(\uu_s,\uv_s)}
\end{equation}
as formally introduced in Definition \ref{defn:linear-pl}.
This is a linear relation in the projective space $\PP_\buw$ 
 where $\buw={\rm sort} \; (\uu_s \vee \uv_s)$
for any (hence for all) $s \in S_F$, and we call it
the linearized $\pl$ relation induced by $F$.
{\it We let $L_\sF$ be the set of all linearized $\pl$ relations
where $\sF$ denotes the set of all $\pl$ relation.}

\begin{thm}\label{eqs-F-intro} 
Assume that the base field has characteristic zero.
Then,  $\FF^{d, E}_*$ as a closed  subscheme  of 
$\PP(\wedge^d E) \times \prod_{\buw \in \La^\sort_{d,[n]}}  \PP_\buw$
 is  defined by the binomial relations in $\cB^\wp$ and $\cB^\frb$ together with 
 all the linearized $\pl$ relations in $L_\sF$ 
 and all the  $\hpl$ relations in $\sF^{\vp\vr}$ 
(introduced in Definition \ref{defn:h-pl} and
 characterized in Proposition \ref{h-pl}).
\end{thm}

See Examples \ref{exam-h pl} and \ref{exam-vr pl}
for  concrete cases of  $\hpl$ relations.

To describe the equations of the $\fG$-quotients $ \HH^\wdep $ 
 in $\prod_{\buw \in \La^\sort_{d,[n]}}  \PP_\buw$, 
 we let
$$R_\vr=\kk[x_{(\uu, \uv)}]_{(\uu,\uv) \in \La_{d,[n]}} \;\;
\hbox{and} \;\;\;
\vi_\vr: R_\vr \to R_0$$ be the restriction of $\vi$ of \eqref{vi-intro} to $R_\vr$.
This corresponds to the embedding 
$$\HH^\wdep \lra  \prod_{\buw \in \La^\sort_{d,[n]}} \PP_\buw$$
as in Proposition \ref{prop1}.

We let $\cB^\vr$ be the set of 
the binomials in $\cB^\frb \cap R_\vr$.
A binomial in $\cB^\vr$ is called a $\vr$-binomial.
See Definition \ref{defn:cBvr} and Example \ref{exam:Bq}.

By deriving the multi-homogenous kernel $\ker^\mh_\vr$, we obtain

\begin{prop}\label{eqs-H'-intro} 
 The $\fG$-quotient
$ \HH^\wdep $ as a closed subscheme of $\prod_{\buw \in \La^\sort_{d,[n]}} \PP_\buw$
is defined by the $\vr$-binomials of $\cB^\vr$.
\end{prop}

Likewise, we let $$\vi_{\Gr,\vr}: \; R_\vr \lra R_0/I_\whwp$$ be the restriction of 
$\vi_\Gr$ of \eqref{vi-Gr-intro} to $R_\vr$. This corresponds to the embedding 
$$\HH^{d, E}\lra  \prod_{\buw \in \La^\sort_{d,[n]}} \PP_\buw$$
as stated in Theorem \ref{thm1}. 
By determining the multi-homogeneous kernel $\ker^\mh_{\Gr,\vr}$, we obtain

\begin{thm}\label{eqs-H-intro} 
Assume that the base field has characteristic zero.
Then, the $\fG$-quotient
$\HH^{d, E}$ as a closed subscheme of $\prod_{\buw \in \La^\sort_{d,[n]}} \PP_\buw$
is defined by the $\vr$-binomials of $\cB^\vr$,
the linearized $\pl$ relaltions of $L_\sF$, and
all the $\vrpl$ relations 
(introduced in Definition \ref{vr pl complex} and 
 characterized in Proposition \ref{h-pl}).
\end{thm}

$\vrpl$ relations are special $\hpl$ relations. See again Example \ref{exam-vr pl}.

\begin{rem}
We conjecture that 
$$\hbox{\it  the binomial relations of $\HH^{2,E}$
in  $\prod_{\buw \in \La_{2,n}^\sort} \PP_\buw$ 
are generated by cubic ones.}$$
Here, $\HH^{2,E}$ should be isomorphic to the Hilbert quotient,
hence isomorphic to $\overline{M}_{0,n}$ according to 
Kapranov (\cite{Kap93b}), at least when  the base field is of characteristic zero. We were informed that the fact that 
$\HH^{2,E}$ is isomorphic  to $\overline{M}_{0,n}$ 
should follow by combining
[Remarques,  Page 11, \cite{La03}] 
and [Theorem 1.4, \cite{Fang2}].
\end{rem}

\subsection{Morphisms among $\fG$-quotients}


Suppose $E=\oplus_{I \in [n]} E_i$ with $\dim E_i =1$ for all $i \in [n]$.
For any proper subset $H \subset [n]$, we let  $E_H=\bigoplus_{h \in H} E_h$.

\begin{prop}\label{pr-I-intro} 
Fix any $I \subset [n]$ with $|I| =r< d$ and $J \subset [n]$ with $|J|>d$.  
Then, there exist two canonical surjective
morphisms 
$$ \xymatrix{
\HH^{d - r, E_{[n]\- I}}   &  \ar[l]_{\;\;\;\;\;\;\;{\rm int}_{[n]\- I}  } \; \HH^{d, E} \; \ar[r]^{{\rm proj}_J \;\;\; }
 &  \HH^{d, E_J}  }
 $$
naturally induced from  the intersection with $E_{[n]\- I}$
and the projection $E \lra E_J$. 
\end{prop}

This theorem is an analogous of [Theorem 1.6.6, \cite{Kap93}].
The  two morphisms when restricted to the open subset $\UU^{d, E}/\TT$
have been  used by Beilinson, MacPherson and Schechtman to propose a theory of motivic cohomology (\cite{BMS}).

When $E=\kk^{2d}$, we let 
$\HH^{d,2d}$ denote the corresponding $\fG$-quotient $\HH^{d,E}$. 

\begin{prop}\label{d2d-intro} For any vector space $E$ with $\dim E \ge 2d$, 
there exists a natural closed embedding
\begin{eqnarray} \label{in-2/2d} 
\xymatrix{
\HH^{d,  E} \ar @{^{(}->}[r]  & (\HH^{d,2d})^{n \choose 2d}.
}
\end{eqnarray}
\end{prop}
This renders the $\fG$-quotients $\HH^{d,2d}$ 
 of some special importance.

\medskip

Let $\fL^{d, E}$ be the set  of all multi-linear coordinate subspaces 
of   $\prod_{\buw \in \La^\sort_{d,[n]}} \PP_\buw$.
Interestingly, it induces a stratification of $\HH^{d, E}$ by locally closed strata:
\begin{equation}\label{LLstr-intr} 
\HH^{d, E}=\bigsqcup_{\cL \in \fL^{d, E}} \HH^{d, E}_\cL .
\end{equation}
See \eqref{l-stra} for details.

\begin{rem}\label{more to come}
The $\fG$-quotients/families of $\Gr^{d,E}$ by $\TT_\bcd$
in this paper is constructed by using a particular class of
Gr\"obner bases of the closures of
$\TT_\bcd$-orbits through general points. Other fixed 
Gr\"obner bases can also be used to introduce $\fG$-quotients/families.
It should interesting to compare these quotients/families.

Further, a key idea of this article, namely, using Gr\"obner bases
to construct quotients/families, extends to general 
linear  actions. Moreover, such constructions can be applied to study
the moduli spaces of projective hypersurfaces
\cite{Hu2026}.


In addition,  we can correspondingly  investigate the 
$\PGL_d$-equivariant geometry on $(\PP^{d-1})^n$.
For instance, we can ask whether {\it there exists a meaningful generalization of $``$cross-ratios$"$ for points
 on the projective plane, space, etc.} Our answer is:
points of $\HH^{d,E}$ are cross-ratios,
the details of which and more will appear elsewhere. 
\end{rem}

\medskip

It would be interesting to establish interesting connections between the results of  this article
and those of several earlier works \cite{Alex02, Alex15}, \cite{AB04a, AB04b},
\cite{Corey},\cite{Faltings97, Faltings01}, 
 \cite{GM10}, \cite{GiWu}, \cite{HKT}, \cite{KT06, KT09}, and
\cite{La99, La03}, \cite{SchT}.   In particular, Question \ref{relation with La03}
asks about relations with Lafforgue's compactifications of quotients of matroid strata. After completing this article, we noticed \cite{Fang2} 
where the authors investigate
the Faltings-Lafforgue compactifications.
 We were informed that investigations of the case of 1-dimensional torus actions can be found in \cite{Fang1, FW}.

\medskip

\centerline{A List of Fixed Notations Used Throughout}


\smallskip

\noindent
$[n]$: the set $\{1, 2, \cdots, n\}$ for any finite positive integer $n$.

\noindent
$\II_{d, [n]}$: the set of all sequences of $d$ distinct integers between 1 and $n$.



\noindent
$\Gr^{d, E}$: the Grassmannian of $d$-subspaces in 
a vector space $E$. 

\noindent
$\PP(\wedge^d E)$: the projective space with $\pl$ coordinates 
$p_\uu, \uu \in \II_{d,[n]}$. 

\noindent
$E_\bcd$: the vector space $E$ together with a decomposition $E=\oplus_{i=1}^n E_i$
with $\dim E_i=r_i$.

\noindent 
$\TT_\bcd$: the center $(\GG_m)^n$ of $\Aut (E_\bcd)$ 
modulo the diagonal $\GG_m$.

\noindent
$\HH^{d, E_\bcd}$: A $\fG$-quotient of $\Gr^{d, E}$ by the subtorus $\TT_\bcd=(\GGm)^n/\GGm$.

\noindent
$\FF^{d, E_\bcd}$: the total scheme of the $\fG$-family over $\HH^{d, E_\bcd}$.

\noindent
$\FF^{d, E_\bcd}_*$: the main component of 
$\FF^{d, E_\bcd}$.

\noindent
$ \HH^\wdebp$:  A $\fG$-quotient of    the Pl\"ucker projective space $\PP(\wedge^d E)$
by $\TTb$. 

\noindent
$ \FF^\wdebp$: the total scheme of the $\fG$-family over $ \HH^\wdebp$.

\noindent
$ \FF^\wdebp_*$: the main component of $ \FF^\wdebp$.

\noindent
$\rS_\bcd$: $\{ \ui =(i_1,\cdots, i_n) \in \NN^n \mid i_1 + \cdots + i_n= d, \;
0\le i_\alpha \le r_\alpha, \alpha \in [n]\}.$

\noindent
$\JJ_\bcd$: weakly increasing strings of length $d$ over   $[n]$
having at most $r_i$ occurrences of  $i \in [n]$.

\noindent
$\JJsrt$: the subset of all sorted  pairs in $\JJ_\bcd^{(2)}$. 

\noindent
$\La^\sort_{d,[n]}$: the set of all sorted pairs in $\II_{d,[n]}^{(2)}$ 
containing at least  $d+2$ 
distinct integers.






\noindent
$\kk$: an arbitrarily fixed base field.
Throughout, all schemes are defined over $\kk$.


\section{A family of general toric ideals and  Gr\"obner bases}\label{toric-ideals}

 \subsection{Combinatorial preliminaries}
Corresponding to the set 
$$\rS_\bcd=\{ \ui =(i_1,\cdots, i_n) \in \NN^n \mid i_1 + \cdots + i_n= d, \;
0\le i_\alpha \le r_\alpha, \alpha \in [n]\},$$
there is a homomorphism
\begin{equation}\label{lift-semi-hom}
\rho_\bcd: \kk[\bx] 
 \lra \kk [\bt^\pm], \;\;\;  x_\ui \to  \bt^{\ui},
\end{equation}
where $\bx=(x_\ui)_{\ui \in  \rS_\bcd}$ and $\bt=(t_1,\cdots, t_n)$.
The kernel of $\rho_\bcd$,  denoted $I_\bcd$, is a homogeneous toric ideal
of $\kk[\bx]$.

Following 14.A of \cite{St}, there is a natural bijection between the elements of $\rS_\bcd$
and weakly increasing strings of length $d$ over the alphabet $[n]=\{1, 2, \cdots, n\}$
having at most $r_\alpha$ occurrences of the letter $\alpha$.
Under this bijection, the vector $(i_1,\cdots, i_n) \in \rS_\bcd$ is mapped to the weakly
increasing string
$$\ba=a_1a_2\cdots a_d= \underbrace{1\cdots 1}_{i_1 \rm \; times} 
 \underbrace{2\cdots 2}_{i_2  \rm \; times} 
\cdots \underbrace{n \cdots n}_{i_n \rm \; times} $$
We let $\JJ_\bcd$ 
 be the set of all the above  weakly increasing strings. 
Thus, we have bijections
\begin{equation}\label{bj0} 
 \fj:  \rS_\bcd \lra \JJ_\bcd,  \;\; \fri: \JJ_\bcd \lra \rS_\bcd
\end{equation}
that are inverse to each other. 


\begin{defn} 
For any sequence of elements
$$\ba=a_1a_2 \cdots a_d,  \bb=b_1b_2 \cdots b_d,\cdots,
\bc=c_1c_2 \cdots c_d  \in \JJ_\bcd,  $$
we set
 $$\ba\vee \bb \vee \cdots \vee \bc:=a_1b_1 \cdots c_1 a_2b_2 \cdots c_2 a_3 \cdots a_db_d \cdots c_d.$$
The sequence $(\ba, \bb, \cdots, \bc)$
 is said to be sorted if $\ba\vee \bb \vee \cdots \vee \bc$ is sorted, that is, 
$$a_1\le b_1 \le \cdots \le c_1 \le a_2 \le b_2 \le \cdots 
\le c_2 \le a_3 \le \cdots \le a_d \le b_d \le \cdots \le c_d.$$
We will use `` {\rm sort}" to denote the operator that takes any string over $[n]$ or
a sequence of strings of $\JJ_\bcd$ as above and sorts it into 
weakly increasing order or into a sorted one.

A monomial in $\kk[x_{\ba}]_{\ba \in \JJ_\bcd}$ is said to be sorted if it can be 
\underline{expressed} as 
$x_\ba x_\bb \cdots x_\bc$ such that the sequence $(\ba, \bb, \cdots, \bc)$ 
or $\ba\vee \bb \vee \cdots \vee \bc$ is  sorted.
\end{defn}

\begin{rem}
Observe here that $x_{113}x_{123}$ is sorted, 
but the same monomial expressed as $x_{123}x_{113}$, corresponding to
$((123),(113))$, would not be sorted.
\end{rem}

We set
$$\JJsrt \subset \JJ^{(2)}_\bcd$$ 
be the subset consisting of
pairs $\bkoe$ such that $\bko \vee \bke$ is sorted.

\begin{prop}\label{bi-eq-0} {\rm (Theorem 14.2, \cite{St})}
There exists a term order $\prec_\bx$ in $\kk[x_{\ba}]_{\ba \in \JJ_\bcd}$ such that 
the sorted monomials are precisely
the $\prec_\bx$-standard monomials modulo $I_\bp$. 
The corresponding  reduced Gr\"obner basis of $I_\bullet$ equals
\begin{equation}\label{GrobnerBasis}
\fG_\bcd=\{  x_{\ba} x_{\bb} -  x_\bko x_\bke \mid
{\rm sort} (\ba \vee \bb) = \bko \vee \bke\}
\end{equation}
for all sorted pairs $\bkoe \in \JJsrt$.
\end{prop}
Here and in the sequel, we refer the reader  to \cite{St} 
for Gr\"obner basis. A few basic definitions and terminologies can also be found in 
the appendix \S \ref{appendix}.

\subsection{Gr\"obner bases for the closures of general torus orbits}
In what follows, we also use $[z_\uu]_{\uu \in \II_{d, [\tn]}}$ to denote the $\pl$ coordinates of $\PP\wdep$,
especially when $[p_\uu]_{\uu \in \II_{d, [\tn]}}$ is used to stand for a particular fixed point $\bp$ of $\PP\wdep$.

There is a natural homomorphism
\begin{equation}\label{uu-ba} \ze: \;\;
\kk[z_\uu]_{\uu \in \II_{d, [\tn]}} \lra \kk[x_{\ba}]_{\ba \in \JJ_\bcd}, \;\;\; z_\uu \to x_\ba \end{equation}
where $\ba \in \JJ_\bcd$ is the unique element such that $\uu \in \II_{\fri(\ba)}$ 
(see \eqref{IIui-intr})
and $\fri$ is the bijection $\JJ_\bcd \to \rS_\bcd$ of \eqref{bj0}.
Note that $\ze^{-1}(x_\ba)=\{ z_\uu \mid \uu \in \II_{\fri(\ba)} \}$.
For any $\ba \in \JJ_\bcd$, we choose and fix any term order $\prec_\ba$ on 
$\kk[z_\uu]_{z_\uu \in \ze^{-1}(x_\ba)}$.  By Lemma \ref{ext-order},
the term order $\prec_\bx$ of Proposition \ref{bi-eq-0} and  the term orders $\prec_\ba, \ba \in \JJ_\bcd$,
induce a term order $\prec_\bz$ on $\kk[z_\uu]_{\uu \in \II_{d, [\tn]}} $.
Then, we let $\uu_\ba$ be such that 
\begin{equation}\label{uuba}
 z_{\uu_\ba}=\min
\; \{ z_\uu \mid \uu \in \II_{\fri(\ba)}\} \; \; \hbox{under $\prec_\ba$, \;\; for all $\ba \in \JJ_\bcd$, }
\end{equation}

Given any $\bp \in \UU(\wedge^d E)$, similar to \eqref{uu-ba}, there is a natural homomorphism
\begin{equation}\label{bp:uu-ba} \ze_\bp: \;\;
\kk[z_\uu]_{\uu \in \II_{d, [\tn]}} \lra \kk[x_{\ba}]_{\ba \in \JJ_\bcd} \end{equation}
$$z_\uu \to p_\uu x_\ba$$
where $\ba \in \JJ_\bcd$ is the unique element such that $\uu \in \II_{\fri(\ba)}$
Note that
$$\hbox{ $\ze_\bp^{-1}(x_\ba)= \{ z_\uu/p_\uu \mid \uu \in \II_{\fri(\ba)} \}$, whilst
$\ze ^{-1}(x_\ba)= \{ z_\uu \mid \uu \in \II_{\fri(\ba)} \}$. }$$
Thus, we can use the previously chosen 
 term order $\prec_\ba$ on 
$\kk[z_\uu]_{z_\uu \in \ze^{-1}(x_\ba)}$, apply Lemma \ref{ext-order},
 and obtain the same term order $\prec_\bz$
on $\kk[z_\uu]_{\uu \in \II_{d, [\tn]}}$ as in the paragraph of \eqref{uu-ba}.

Composing the quotient map $ \kk[x_{\ba}]_{\ba \in \JJ_\bcd} 
\to \kk[x_{\ba}]_{\ba \in \JJ_\bcd} /I_\bcd$
with the homomorphism $\ze_\bp$ of \eqref{bp:uu-ba}, we obtain the surjection
 \begin{equation}\label{uu-ba-I} 
\bar\ze_\bp:\;\;
\kk[z_\uu]_{\uu \in \II_{d, [\tn]}} \lra \kk[x_{\ba}]_{\ba \in \JJ_\bcd} 
\lra \kk[x_{\ba}]_{\ba \in \JJ_\bcd} /I_\bcd. 
\end{equation}

We let $I_\bp=\ker \bar\ze_\bp$. Then, one sees that
the closure $\overline{\TTb \cdot {\bf p}}$, as a subscheme of 
the $\pl$ projective space $\PP(\wedge^d E)$, is defined by the homogeneous ideal $I_\bp$.


\begin{thm}\label{bi-eq-bp} 
 Fix any point $\bp=[p_\uu]_{\uu \in \II_{d,[\bf n]}} \in \UU(\wedge^d E)$.
There exists a term order $\prec_\bz$ in $\kk[z_{\uu}]_{\uu \in \II_{d,[\tn]}}$ which is independent of 
the point $\bp$ such that  
a corresponding reduced Gr\"obner basis of $I_\bp$ under the order $\prec_\bz$ equals to
 \begin{equation}\label{GrobnerBasis-bp}
\fG_\bp=\left\{ 
\begin{array}{ccccccc}
p_{\uu_\ba} z_\uu - p_\uu z_{\uu_\ba} \\ 
 p_{\uu_\bko} p_{\uu_\bke} z_{\uu_\ba} z_{\uu_\bb}- p_{\uu_\ba} p_{\uu_\bb} \;   z_{\uu_\bko} z_{\uu_\bke} 
\\
\end{array}
  \; \Bigg| \;
\begin{array}{ccccc}
 \uu \in \II_{\fri(\ba)}  \\  
 {\rm sort} (\ba \vee \bb) = \bko \vee \bke\}
\end{array}
 \right \},
\end{equation}
for all $\ba\in \JJ_\bcd$ and $\buw=\bkoe \in \JJsrt$.
Upon dividing every $p_{\uu_\ba} z_\uu - p_\uu z_{\uu_\ba} $ by $p_{\uu_\ba} $ and every 
$p_{\uu_\bko} p_{\uu_\bke} z_{\uu_\ba} z_{\uu_\bb}- p_{\uu_\ba} p_{\uu_\bb} \;   z_{\uu_\bko} z_{\uu_\bke} 
$ by $p_{\uu_\bko} p_{\uu_\bke} $,
we obtain the unique reduced Gr\"obner basis of $I_\bp$ under the order $\prec_\bz$ such that
all the leading terms of the relations of $\fG_\bp$  are monic.
\end{thm}

As a consequence, a monomial $\bm$ of $\kk[z_{\uu}]_{\uu \in \II_{d,[\tn]}}$
is $\prec_\bz$-non-standard if and only if one of the following two holds:
\begin{itemize}
\item $z_\uu \mid \bm$ for some $\uu \in \II_{\fri(\ba)}$ with $\uu \ne \uu_\ba$;
\item  $z_{\uu_\ba} z_{\uu_\bb} \mid \bm$ 
for $z_{\uu_\ba} z_{\uu_\bb}$ in some binomial 
$p_{\uu_\bko} p_{\uu_\bke} z_{\uu_\ba} z_{\uu_\bb}- p_{\uu_\ba} p_{\uu_\bb} \;   z_{\uu_\bko} z_{\uu_\bke}$
of $\fG_\bp$.
\end{itemize}

\begin{proof} 
By Proposition \ref{bi-eq-0}, $I_\bcd$ is generated by 
$$ \{  x_{\ba} x_{\bb} -  x_\bko x_\bke \mid
{\rm sort} (\ba \vee \bb) = \bko \vee \bke\}
$$ for all sorted pairs $\bkoe \in \JJsrt$.
Substituting $x_\ba$ by $z_\uu/p_\uu$ for all $\ba \in \JJ_\bcd$ and all
$\uu  \in \II_{\fri (\ba)} \subset \II_{d,[\tn]}$, we obtain that $\ker \ze_\bp$ is generated by
\begin{equation}\label{lin-g}
 \{p_{\uv} z_\uu - p_\uu z_{\uv} \mid \;
\forall \; \uu, \uv \in \II_{\fri(\ba)}, \;\; \forall \; \ba \in \JJ_\bcd\}
\end{equation} together with
\begin{equation}\label{quadr-g}
 \{p_{\uh} p_{\uk} z_{\uu} z_{\uv}- p_\uu p_\uv  \  z_{\uh} z_{\uk} 
\mid (\uu, \uv), (\uh,\uk) \in \La_\buw, \;\; \forall \; \buw \in \JJsrt
\}.
\end{equation}

The relations of \eqref{lin-g} contain the special ones
\begin{equation}\label{lin-g-0}
p_{\uu_\ba} z_\uu - p_\uu z_{\uu_\ba}, \;\;  \uu \in \II_{\fri(\ba)}.
\end{equation} 
We denote this set of relations by $\fG_\bp^{\rm lin}$.

The relations of \eqref{quadr-g} contain the special ones
\begin{equation}\label{quadr-g-0}
p_{\uu_\bko} p_{\uu_\bke} z_{\uu_\ba} z_{\uu_\bb}- p_{\uu_\ba} p_{\uu_\bb} \;   z_{\uu_\bko} z_{\uu_\bke} ,
\;\; {\rm sort} (\ba \vee \bb) = \bko \vee \bke.
\end{equation}

We denote this set of relations by $\fG_\bp^{\rm qdr}$.

First, one calculates and find that for any relation of \eqref{lin-g}, we have
$$p_{\uu_\ba}(p_{\uv} z_\uu - p_\uu z_{\uv})= p_\uv (p_{\uu_\ba} z_\uu - p_\uu z_{\uu_\ba}) - 
p_\uu (p_{\uu_\ba} z_\uv - p_\uv z_{\uu_\ba}). $$

Next, one calculates and find that for any relation of \eqref{quadr-g}, we have 
$$p_{\uu_\bko} p_{\uu_\bke}  (p_{\uh} p_{\uk} z_{\uu} z_{\uv}- p_\uu p_\uv  \  z_{\uh} z_{\uk}) =$$
$$ p_{\uh} p_{\uk} (p_{\uu_\bko} p_{\uu_\bke}  z_{\uu} z_{\uv}
-p_\uu p_\uv z_{\uu_\bko} z_{\uu_\bke} )
- p_\uu p_\uv  (p_{\uu_\bko} p_{\uu_\bke} z_{\uh} z_{\uk}
-p_{\uh} p_{\uk}z_{\uu_\bko} z_{\uu_\bke} ).$$

We may suppose $\uu  \in \II_{\fri (\ba)}, \uv  \in \II_{\fri (\bb)}$ for some $\ba,  \bb \in \JJ_\bcd$.
Then,  using the relations $$p_{\uu_\ba}  z_{\uu} - p_{\uu} \;   z_{\uu_\ba}, \;\;
p_{\uu_\bb}  z_{\uv} - p_{\uv} \;   z_{\uu_\bb},$$
we obtain
$$  p_{\uu_\ba} p_{\uu_\bb} (p_{\uu_\bko} p_{\uu_\bke} z_{\uu} z_{\uv}
-p_\uu p_\uv z_{\uu_\bko} z_{\uu_\bke}) \equiv$$
$$   p_{\uu} p_{\uv} (p_{\uu_\bko} p_{\uu_\bke} z_{\uu_\ba} z_{\uu_\bb}
- p_{\uu_\ba} p_{\uu_\bb}    z_{\uu_\bko} z_{\uu_\bke}) ,\;
 \hbox{modulo relations of \eqref{lin-g-0}}.$$
Similarly, 
$$  p_{\uu_\bc} p_{\uu_\be} (p_{\uu_\bko} p_{\uu_\bke} z_{\uh} z_{\uk}
-p_\uh p_\uk z_{\uu_\bko} z_{\uu_\bke}) \equiv$$
$$   p_{\uh} p_{\uk} (p_{\uu_\bko} p_{\uu_\bke} z_{\uu_\bc} z_{\uu_\be}
- p_{\uu_\bc} p_{\uu_\be}    z_{\uu_\bko} z_{\uu_\bke}) ,\;
 \hbox{modulo relations of \eqref{lin-g-0}},$$
where $\uh  \in \II_{\fri (\bc)}, \uk \in \II_{\fri (\be)}$ for some $\bc,  \be \in \JJ_\bcd$.

This implies that $I_\bp=\ker \bar\ze_\bp$ is generated by relations of \eqref{lin-g-0}
and \eqref{quadr-g-0}, i.e., by the relations in the set $\fG_\bp$
as stated in the theorem.

It remains to prove that the set $\fG_\bp=\fG_\bp^{\rm lin} \sqcup \fG_\bp^{\rm qdr}$ 
is a reduced Gr\"obner basis under 
the term order $\prec_\bz$ described in the paragraph of \eqref{bp:uu-ba} which is the same
term order used in  the paragraph that follows \eqref{uu-ba}.

First,  let $f$ be any of the binomials of $\fG_\bp^{\rm lin}$ as in \eqref{lin-g-0}, say,
$$p_{\uu_\ba} z_\uu - p_\uu z_{\uu_\ba}, \;\;  \hbox{for some} \;\; \uu \in \II_{\fri(\ba)}.$$
let  $g$ be any element of $\fG_\bp$ such that $g \ne f$. Then, one inspects that 
$in_{\prec_\bz} (f) \; (=p_{\uu_\ba} z_\uu)$ and $in_{\prec_\bz} (g)$  are co-prime.
By applying Lemma \ref{s-pair}, we have that a complete reduction of the
 $S$-polynomial $S(f,g)$ of $f$ and $g$  by $\{f, g\}$ produces zero.

Next,  we can focus on the binomial relations of $\fG_\bp^{\rm qdr}$ as in \eqref{quadr-g-0}.

We let 
$$\kk[z_{\uu_\ba}]_{\ba \in \JJ_\bcd} \subset \kk[z_{\uu}]_{\uu \in \II_{d,[\tn]}}$$
be the subalgebra generated by $z_{\uu_\ba}, \ba \in \JJ_\bcd$.
Then, $\ze_\bp$ of \eqref{bp:uu-ba} restricted to $\kk[z_{\uu_\ba}]_{\ba \in \JJ_\bcd}$
induces the following identification of two algebras
$$\iota_\bp: \kk[z_{\uu_\ba}]_{\ba \in \JJ_\bcd} \lra \kk[x_\ba]_{\ba \in \JJ_\bcd}$$
$$z_{\uu_\ba} \to p_{\uu_\ba}x_\ba.$$

By the construction of the term order $\prec_\bz$ (using Lemma \ref{ext-order} and $\prec_\bx$),
$\iota_\bp$ preserves the term orders with respect to $\prec_\bz$ and $\prec_\bx$, that is,
$$\hbox{$ \bm \prec_\bz \bm'$ if and only $\iota_\bp (\bm) \prec_\bx \iota_\bp (\bm')$}$$
for any two monomials of $ \kk[z_{\uu_\ba}]_{\ba \in \JJ_\bcd}$.

Furthermore, $\iota_\bp$ sends a relation
$$p_{\uu_\bko} p_{\uu_\bke} z_{\uu_\ba} z_{\uu_\bb}- p_{\uu_\ba} p_{\uu_\bb} \;   z_{\uu_\bko} z_{\uu_\bke} ,
\;\; {\rm sort} (\ba \vee \bb) = \bko \vee \bke$$
of $\fG_\bp^{\rm qdr}$ in \eqref{quadr-g-0}
to the relation 
$$  x_{\ba} x_{\bb} -  x_\bko x_\bke, \;\;\; 
{\rm sort} (\ba \vee \bb) = \bko \vee \bke$$
of $\fG_\bcd$ in \eqref{GrobnerBasis}, up to a scalar multiple. 
Equivalently, $\iota_\bp$ pulls back the later relation to the former, up to a scalar multiple. 

Now we let $f$ and $g$ be any two distinct elements of the binomials of \eqref{quadr-g-0}.
We have that $\iota_\bp(f)$ and $\iota_\bp (g)$ are two relations of \eqref{GrobnerBasis}
up to scalar multiples. 
Because \eqref{GrobnerBasis} is a Gr\"obner basis, we have
that some complete reduction of the
 $S$-polynomial $S(\iota_\bp(f), \iota_\bp (g))$ of 
 $\iota_\bp(f)$ and $\iota_\bp (g) $ by the set $\fG_\bcd$ of \eqref{GrobnerBasis} produces zero.
 Now, the complete reduction process of $S(\iota_\bp(f), \iota_\bp (g))$
 by the set $\fG_\bcd$ of \eqref{GrobnerBasis} inside the algebra $\kk[x_\ba]_{\ba \in \JJ_\bcd}$
 can be pulled back to the algebra $\kk[z_{\uu_\ba}]_{\ba \in \JJ_\bcd}$ to yield 
a complete reduction process of $S(f, g)$
 by the set $\fG_\bp^{\rm qdr}$ of \eqref{quadr-g-0}, and this again should produce zero.

Hence, by the Buchberger's algorithm (criterion), we obtain that 
 $\fG_\bp$ is a  Gr\"obner basis under the term order $\prec_\bz$. 
 
Finally, for any $g, g' \in \fG_\bp$, by a direct inspection, one finds that no term of $g'$ is divisible by
$in_{\prec_\bz} (g)$. Hence, we  conclude that $\fG_\bp$ is indeed a 
reduced Gr\"obner basis of $I_\bp$.
\end{proof}

 Theorem \ref{bi-eq-bp} immediately implies

\begin{cor}\label{=hpoly}
The closure $\overline{\TTb \cdot {\bf p}}$, as a closed subscheme of 
the $\pl$ projective space $\PP(\wedge^d E)$, is defined by the relations \eqref{GrobnerBasis-bp}.
Moreover, the Hilbert polynomial of  the subscheme $\overline{\TTb \cdot {\bf p}}$
is independent of $\bp \in \UU(\wedge^d E)$. 
\end{cor}

\section{The $\fG$-quotient  $\HH^{d, E_\bcd}$ and
$\fG$-family  $\FF^{d, E_\bcd}$ by the poly-diagonal subtorus} 
\label{q-revisited}

\subsection{Proofs of Theorem \ref{thm1} and Proposition \ref{prop1}
}

First, for any $\ba \in \JJ_\bcd$, we have the vector space 
$V_\ba$
with a basis $\{e_\uu \mid \uu  \in \II_{\fri (\ba)}\}$ over the field $\kk$ and 
the projective space
\begin{equation}\label{Pba}
\PP_\ba :=\PP(V_\ba)
\end{equation} 
such that it comes equipped with the homogeneous coordinates
  \begin{equation}\label{yba}y_\ba=[y_\uu]_{\uu  \in \II_{\fri (\ba)}}.
  \end{equation}

Next, or any  $\buw=\bkoe \in \JJsrt$, recall from \eqref{La-buw},
we define
$$
\La_\buw=\{(\uu, \uv) \in \II_{d,n}^{(2)}
\mid \uu \in \II_{\fri (\ba)}, \uv \in \II_{\fri (\bb)},
 (\ba,\bb) \in \JJ^{(2)}_\bcd,  {\rm sort}(\ba \vee \bb)=\bko \vee \bke \}.$$
Corresponding to $\Lambda_{\buw}$, we introduce the vector space 
$V_\buw$  
with a basis
$ \{e_{(\uu,\uv)} \mid (\uu, \uv) \in \La_\buw\}$ over the field $\kk$ and
the projective space
\begin{equation}\label{Pw}\PP_\buw :=\PP(V_\buw)
\end{equation} 
such that it comes equipped with the homogeneous coordinates
 \begin{equation}\label{ybw}
 y_\buw=[x_{(\uu, \uv)}]_{(\uu,\uv) \in \Lambda_{\buw}}
 \end{equation}

We then introduce  the following natural rational map\footnote{cf. 
Remark \ref{more to come}.}
\begin{equation}\label{theta'1}
\xymatrix{
\Theta^\bcd:  \PP(\wedge^d E)   \ar @{-->}[r] &   
 \prod_{\ba \in \JJ_\bcd} \PP_\ba \times \prod_{\buw \in \JJsrt} \PP_\buw ,
}
\end{equation}
$$[p_\uu]_{\uu \in \II_{d,[n]}} \lra
 \prod_{\ba \in \JJ_\bcd} [p_\uu]_{\uu \in \II_{\fri(\ba)}} \times
\prod_{\buw \in \JJsrt} [p_{\uu}p_{\uv}]_{(\uu,\uv) \in \La_\buw}.$$
The rational map $\Theta$ is 
$\TTb$-equivariant, where $\TTb$ 
 acts trivially on the target.
This rational map restricts to a morphism 
$\UU(\wedge^d E) \rightarrow
 \prod_{\ba \in \JJ_\bcd} \PP_\ba \times \prod_{\buw \in \JJsrt} \PP_\buw$.  

 \begin{lemma}\label{U'/TinHilbert} The geometric GIT quotient $\UU(\wedge^d E)/\TTb $ exists and
 the morphism $$\Theta^\bcd: \UU(\wedge^d E) \lra
  \prod_{\ba \in \JJ_\bcd} \PP_\ba \times \prod_{\buw \in \JJsrt} \PP_\buw$$
descends to an embedding
$$\beta^\bcd: \UU(\wedge^d E)/\TTb \lra  
 \prod_{\ba \in \JJ_\bcd} \PP_\ba \times \prod_{\buw \in \JJsrt} \PP_\buw.$$
\end{lemma}
\begin{proof} It is a standard fact that the geometric GIT quotient $\UU(\wedge^d E)/\TTb$ exists.
The morphism $\beta^\bcd$ exists because the GIT quotient $\UU(\wedge^d E)/\TTb $ is categorical.
The fact that it is an embedding is a direct consequence of Theorem \ref{bi-eq-bp}.
\end{proof}

\subsubsection{Proof of Proposition-Definition \ref{prop1}.}

\begin{proof} Lemma \ref{U'/TinHilbert} is a restatement of the assertion of
Proposition-Definition \ref{prop1}.
\end{proof}

The rational map $\Theta_\bcd$ of (\ref{theta'1}) restricts to $\Gr^{d, E}$ to give rise to
\begin{equation}\label{theta-Gr}
\xymatrix{
\Theta^\bcd_\Gr: \Gr^{d, E} \ar @{-->}[r] & 
    \prod_{\ba \in \JJ_\bcd} \PP_\ba \times \prod_{\buw \in \JJsrt} \PP_\buw,
}
\end{equation} 
 $$[p_\uu]_{\uu \in \II_{d,[n]}} \lra
 \prod_{\ba \in \JJ_\bcd} [p_\uu]_{\uu \in \II_{\fri(\ba)}} \times
 \prod_{\buw \in \JJsrt} [ p_{\uu}p_{\uv}]_{(\uu,\uv) \in \La_\buw}$$
where $\bp=[ p_\uu]_{\uu \in \II_{d,[n]}}$ is the Pl\"ucker coordinate of a point of $\Gr^{d, E}$.

Because the inclusion $\UU^{d, E} \subset \UU(\wedge^d E)$ descends
to  $\UU^{d, E}/\TTb \subset \UU(\wedge^d E)/\TTb$, we obtain

 \begin{cor}\label{U/TinHilbert} The geometric GIT quotient $\UU^{d, E}/\TTb$ exists and
 the morphism $$\Theta^\bcd_\Gr: \UU^{d, E} \lra
  \prod_{\ba \in \JJ_\bcd} \PP_\ba \times \prod_{\buw \in \JJsrt} \PP_\buw$$
descends to an embedding
$$\beta^\bcd_\Gr: \UU^{d, E}/\TTb \lra
  \prod_{\ba \in \JJ_\bcd} \PP_\ba \times \prod_{\buw \in \JJsrt} \PP_\buw.$$
\end{cor}

\subsubsection{Proof of Theorem-Definition \ref{thm1}.}

\begin{proof} Corollary \ref{U/TinHilbert} is a restatement of the assertion in
Theorem-Definition \ref{thm1}.
\end{proof}

\subsection{Proofs of Propostion \ref{Z'asBlowup-intro} 
and Theorem \ref{ZasBlowup-intro}}
\label{F-revisited} 

In this section, we will prove Proposition \ref{Z'asBlowup-intro}
and Theorem  \ref{ZasBlowup-intro}. We need a lemma.

\subsubsection{A lemma on blowup}
Let $J$ be a homogeneous ideal of $\kk[x_1, \cdots, x_n]$ and
 $X={\rm Proj} \kk[x_1, \cdots, x_n]/J$. Fix any postive integer $k$.
 Let $I_i$ be an ideal of $\kk[x_1, \cdots, x_n]/J$ 
 having a set of generator $\cG_{I_i}=\{g_{i1}, \cdots, g_{im_i}\}$
 for every $i \in [k]$. We let $\PP_{\cG_i}$ 
 be projective space with
 the homogenous coordinates $[y_{i1},\cdots, y_{im_i}]$. Then, we have
 a rational map
 \begin{equation}\label{blowup-rat}
 X \dashrightarrow \prod_{i \in [k]}\PP_{\cG_i},\;\; x \to \prod_{i \in [k]} [g_{i1}(x), \cdots, g_{im_i}(x)].
\end{equation}
We let $U=X \- Z$ where $Z$ is the closed subscheme 
of $X$ defined by $\prod_{i=1}^k I_i$.

\begin{lemma}\label{thatLemma} Let the notation be as in the above.
Then, the following two hold.

\begin{enumerate}
\item 
The blowup scheme
$\Bl_{\prod_{i=1}^k I_i} X$ of $X$ along the ideal $\prod_{i=1}^k I_i$  
is  isomorphic to the closure of the graph of the rational
map $X \dashrightarrow \prod_{i=1}^k \PP_{\cG_i}$.
\item 
$\Bl_{\prod_{i=1}^k I_i} X$ 
is also isomorphic to
the scheme resulted from successively blowing up $X$ along the pullbacks of the ideals 
$\{I_1, \cdots, I_k\}$, in any given order.
\end{enumerate}
\end{lemma}

This lemma should be known. We present a proof for completeness.

\begin{proof}
(1).
For every $i\in [k]$, we pick exactly one generator 
$g_{ih_i}$ for some $h_i \in [m_i]$
 and then take the product $\prod_{i \in [k]} g_{ih_i}$ of these generators. 
 Putting all these products together, 
we obtain a set of generators, $\cG=\{\prod_{i \in [k]} g_{ih_i}\}$,  for the ideal $\prod_{i=1}^k I_i$.
Let $\PP_\cG$ be the projective space with homogeneous coordinates $[y_g]_{g \in \cG}$ 
and $$X \dashrightarrow \PP_\cG, \;\; x \to [g(x)]_{g \in \cG}$$ be the induced rational map. Then
the blowup  $\Bl_{\prod_{i=1}^k I_i} X$ 
is isomorphic to the closure of the embedding 
$$U \lra U \times \PP_{\cG}.$$
Consider  the  following commutative diagram
  \begin{equation}\label{H}
\xymatrix{
 U \ar @{^{(}->}[r]  \ar[d]^{=} & U \times \prod_{i \in [k]} \PP_{\cG_i} \ar[d]^{\Id_U \times \fs} \\
  U \ar @{^{(}->}[r] & U \times \PP_{\cG} ,}
 \end{equation}
 where $\fs: \PP_{\cG_i}  \to \PP_{\cG}$ is the Segre embedding.
It  implies that  the closure of the embedding in the top row is naturally isomorphic 
 to the closure of the embedding in the bottom row.
 This implies the statement of (1).
 
 (2). Fix any order of the set $\{I_1, \cdots, I_k\}$, say, $I_1, I_2, \cdots, I_k$, as listed.
We prove (1) by induction on $k$. The statement clearly holds when $k=1$. We suppose that
the statement holds for $k-1$. That is, the scheme 
$\Bl_{(I_1, \cdots, I_{k-1})} X$ resulted from successively blowing up $X$ along the pullbacks of the ideals $I_1, \cdots, I_{k-1}$, in that order,  is isomorphic to the closure of the graph of the rational
map $X \dashrightarrow \prod_{i=1}^{k-1} \PP_{\cG_i}$.
Let $\pi_{k-1}: \Bl_{(I_1, \cdots, I_{k-1}, I_k)} X  \to X$ denote the blowup morphism
and write $I_k=\langle g_{k1}, \cdots, g_{km_k}  \rangle$.  Then, 
$\Bl_{(I_1, \cdots, I_{k-1}, I_k)} X$, the blowup of $\Bl_{(I_1, \cdots, I_{k-1})} X$ along the pullback of $I_k$
is  isomorphic to 
the closure  in $\Bl_{(I_1, \cdots, I_{k-1})} X \times \PP_{\cG_k}$
of the image of the embedding 
$$\iota_{k-1} (U) \rightarrow \iota_{k-1} (U) \times  \PP_{\cG_k},$$
$$\iota_{k-1} (u) \to \iota_{k-1} (u) \times [g_{k1}(u), \cdots, g_{km_k} (u)], \;\;\;
\forall \; u \in U,$$
which is clearly equal to the closure in 
$X \times \prod_{i=1}^{k} \PP_{\cG_i}$
of  the image of the embedding
$$U \lra (U  \times  \prod_{i=1}^{k-1} \PP_{I_i}) \times  \PP_{\cG_i}= U  \times  \prod_{i=1}^{k} \PP_{\cG_i}.$$
Since the product  $\prod_{i=1}^{k} \PP_{\cG_i}$, up to isomorphisms, is independent of the order of the set $\{I_1, \cdots, I_k\}$,
 by (1),  the statement (2) follows.
\end{proof}


\subsubsection{Proof of Proposition \ref{Z'asBlowup-intro}.}
 \begin{proof}
Consider  the composed rational map 
 $$\xymatrix{
 \PP(\wedge^d E) \ar @{-->}[r]^{\Theta^\bcd \;\;\;\;\;\; \;\;\;\;\;\;\;\;\;\;\;\;\;\;} &
 \prod_{\ba \in \JJ_\bcd} \PP_\ba \times \prod_{\buw \in \JJsrt} \PP_\buw
  \ar @{^{(}->}[r]^{\;\;\;\;\;\;\;\;\;\;\;\;\;\;\;\; \fs} & \PP_{J^\wdebp}
 }$$
 where $\fs$ is the Segre embedding.
By the proof of Lemma \ref{thatLemma} (1), we see
that the degenerate locus of the rational map $\Theta^\bcd$ 
is precisely the subscheme defined by the ideal $J^\wdebp$.

To prove the second statement, just observe that from definition,
$\FF^\wdebp_*$ 
is equal to the closure of the rational map $\Theta^\bcd$
and the blowup $\Bl_{J^\wdebp}\PP(\wedge^d E)$
of $\PP(\wedge^d E)$ along the ideal $J^\wdebp$ 
is isomorphic to the closure of the rational map $\Theta^\bcd$
by  Lemma \ref{thatLemma}.
 \end{proof}

\subsubsection{Proof of Theorem \ref{ZasBlowup-intro}.}
\begin{proof}  Consider  the composed rational map 
 $$\xymatrix{
 \Gr^{d,E} \ar @{-->}[r]^{\Theta^\bcd_\Gr \;\;\;\;\;\; \;\;\;\;\;\;\;\;\;\;\;\;\;\;} &
 \prod_{\ba \in \JJ_\bcd} \PP_\ba \times \prod_{\buw \in \JJsrt} \PP_\buw
  \ar @{^{(}->}[r]^{\;\;\;\;\;\;\;\;\;\;\;\;\;\;\;\; \fs} & \PP_{J^\wdebp}}.$$
Using $\Theta^\bcd_\Gr$ in place of $\Theta^\bcd$,
the proof is totally parallel to that of Proposition \ref{Z'asBlowup-intro}.
We omit further details.
\end{proof}


\subsection{Special cases and examples } 

Let $E=\oplus_{i \in [n]} F_i$ with $\dim F_i=1$. We let $E_\bcd$ be the decomposition
$$E= E_1 \oplus E_2, \;\; E_1= \oplus_{i \in [n-1]} F_i, \;\; E_2=F_n.$$
Then, $\dim \TTb=1$. One calculates and finds
$$\rS_\bcd=\{(i_1, i_2)\in \NN^2 \mid i_1+ i_2=d,\; 0 \le i_1 \le n-1, 0 \le i_2 \le 1\}=\{(d, 0), (d-1, 1)\},$$
$$\JJ_\bcd=\{\ba=(11\cdots11), \bb=(11\cdots 1 2)\}.$$
Hence, $\PP_\buw$ is a single point for every $\buw \in \JJsrt$; we can discard them from consideration.
We have 
$$\II_{\fri(\ba)}=\II_{(d,0)}=\{ \uu=(u_1\cdots u_d) \in \II_{d,n} \mid u_i \in [n-1], i \in [d]\},$$
$$\II_{\fri(\bb)}=\II_{(d-1,1)}=\{ \uu=(u_1\cdots u_{d-1}n) \in \II_{d,n} \mid u_i \in [n-1], i \in [d-1]\}.$$
Recall (and observe) that $\II_{d,[n]}=\II_{(d,0)}  \sqcup \II_{(d-1,1)}.$
The Grassmannian $\Gr^{d,E}$ contains
$$\Gr^{d, E}_{\subset E_1}=\{K \in \Gr^{d,E} \mid K \subset E_1\},$$
$$\Gr^{d, E}_{\supset F_n}=\{K \in \Gr^{d,E} \mid K \supset F_n\}.$$
Then, $\Gr^{d, E}_{\subset E_1}$ is defined by $J^{d,E_\bcd}_\bb$, while
$\Gr^{d, E}_{\supset F_n}$ is defined by $J^{d,E_\bcd}_\ba$.
Thus, we can have the following natural inclusions and identifications:
\begin{equation}\label{include Gr}
\Gr^{d, E}_{\subset E_1}\subset \PP(\wedge^d E_1)=\PP_\ba, \;\; 
\Gr^{d, E}_{\supset F_n} \subset \PP(\wedge^{d-1} (E/F_n))=\PP_\bb.\end{equation}
\begin{prop}\label{n-1+1} 
 Let the assumption be as above. Then, we have
\begin{enumerate}
\item  $\HH^{d, E_\bcd}$ 
is naturally isomorphic to a closed subscheme of  
 $\Gr^{d, E}_{\subset E_1} \times \Gr^{d, E}_{\supset F_n}$ 
satisfying 
 all the $\pl$ relations that are linear in both $[y_\uu]_{\uu  \in \II_{(d,0)}}$
 and $[y_\uv]_{\uv  \in \II_{(d-1,1)}}$.

(Indeed, $\HH^{d, E_\bcd}$ should be precisely defined by all such relations.) 
\item  $\FF^{d, E_\bcd} \lra \HH^{d, E_\bcd}$ is naturally isomorphic to
   $\PP^1 \times \HH^{d, E_\bcd} \lra  \HH^{d, E_\bcd}$.
\item $\FF^{d, E_\bcd}=\FF^{d, E_\bcd}_*$ is
 isomorphic to the blowup of $\Gr^{d,E}$ 
along $\Gr^{d, E}_{\subset E_1} \sqcup \Gr^{d, E}_{\supset F_n}$.
\end{enumerate}
\end{prop}
\begin{proof} 
(1).
 We let $\tilde\HH^{d, E_\bcd}$ be the closed subscheme of $\PP_\ba \times \PP_\bb$
 defined by all $\pl$ relations in
 $(y_\uu)_{\uu  \in \II_{(d,0)} \sqcup \II_{(d-1,1)}}$, or equivalently, 
 the closed subscheme of  $\Gr^{d,E}_{\subset E_1} \times \Gr^{d,E}_{\supset F_n}
\subset \PP_\ba \times \PP_\bb$
 satisfying $\pl$ relations that are linear in both $[y_\uu]_{\uu  \in \II_{(d,0)}}$
 and $[y_\uv]_{\uv  \in \II_{(d-1,1)}}$. 
By the construction of the embedding $\HH^{d, E_\bcd} \to \PP_\ba \times \PP_\bb$
(Theorem \ref{thm1}), we have 
that the image of $\HH^{d, E_\bcd}$ satisfies all $\pl$ relations in
 $(y_\uu)_{\uu  \in \II_{(d,0)} \sqcup \II_{(d-1,1)}}$, that is, $\HH^{d, E_\bcd} \subset \tilde\HH^{d, E_\bcd}$.  This proves (1).
 
(2). Now, for any fixed point 
$ \fy=([y_\uu]_{\uu \in \II{(d, 0)}}, [y_\uv]_{\uv \in \II_{(d-1,1)}})  \in \HH^{d, E_\bcd}$,
 the set $$\{[\la (y_\uu), \mu (y_\uv)]_{\uu \in \II_{(d,0)},\uv \in \II_{(d-1,1)}}
\mid [\la, \mu] \in \PP^1\}  \subset \PP\wdep$$
is a subset of $\Gr^{d,E}$
and one sees that it
is a $\TTb$-orbit closure in $\Gr^{d,E}$. 
Hence, by construction, as a closed subset of $\Gr^{d,E} \times \HH^{d, E_\bcd}
 \subset  \Gr^{d,E} \times \PP_\ba \times \PP_\bb,$ we have 
$$\FF^{d, E_\bcd}=
\{([\la (y_\uu), \mu (y_\uv)]_{\uu \in \II_{(d,0)},\uv \in \II_{(d-1,1)}}, \fy)
\mid [\la, \mu] \in \PP^1, \fy  \in \HH^{d, E_\bcd}\}$$
where $\fy=([y_\uu]_{\uu \in \II{(d, 0)}}, [y_\uv]_{\uv \in \II_{(d-1,1)}}) \in 
\HH^{d, E_\bcd} \subset \PP_\ba \times \PP_\bb$.
This clearly exhibits $\FF^{d, E_\bcd}$ as the product $\PP^1 \times \HH^{d, E_\bcd}$.

(3). Clearly, $\FF^{d, E_\bcd}$ is irreducible. Thus, (3)
follows directly from Theorem \ref{ZasBlowup-intro} and the   discussions
proceeding \eqref{include Gr}.
\end{proof}

\begin{exam}\label{FF-HH-24} Let $E=F_1\oplus F_2\oplus F_3 \oplus F_4$ such that $\dim F_i=1, i \in [4]$.
We let $E_\bcd= E_1 \oplus E_2$ with $E_1=F_1\oplus F_2\oplus F_3$ and $E_2=F_4$.
Then, $\rS_\bcd=\{(2,0), (1,1)\}$. We have
$$\PP_{(2,0)}=\{[y_{12},y_{13}, y_{23}]\},\;\; \PP_{(1,1)}=\{[y_{14},y_{24}, y_{34}]\}.$$
The $\fG$-quotient $$\HH^{2, E_\bcd} \subset \PP_{(2,0)} \times \PP_{(1,1)}$$ 
is the threefold defined by the bi-linear ($\pl$) relation
$$y_{12} y_{34}- y_{13}y_{24}+y_{23}y_{14}.$$
\end{exam}

Indeed, for any $i \in [n]$, we can use the decomposition $E_{\bcd_i}$
$$E= E_{[n]\-i} \oplus F_i, \;\; \hbox{with}\;\; E_{[n]\-i}=\oplus_{h \in [n] \- i} F_h$$
in place of $E_1= \oplus_{i \in [n-1]} F_i$ and $E_2=F_n$ to obtain  statements similar to
all the above.
We let $\TT_{\bcd_i}$ be the corresponding 1-dimensional subtorus of $E_{\bcd_i}$.
Similar to \eqref{include Gr}, we have
$$\hbox{ $\Gr^{d, E}_{\subset E_{[n]\- i}} \subset \PP_{\ba_i}$ and 
$\Gr^{d, E}_{\supset F_i} \subset \PP_{\bb_i}$} $$
for some elements $\ba_i$ and $\bb_i$  in the corresponding $\JJ_\bcd$,  for all $i \in [n]$.
Then, there exists a natural rational map
$$\Gr^{d,E} \dashrightarrow \prod_{i=1}^n 
(\PP_{\ba_i} \times \PP_{\bb_i}).$$
We can let $\widehat\HH^{d,E}$ be the closure of this rational map.

\begin{rem}\label{iterated hilbert} 
It would be interesting to know whether $\widehat\HH^{d,E}$
is isomorphic to the iterated $\fG$-quotient of $\Gr^{d,E}$ by 
$\TT_{\bcd_1}, \TT_{\bcd_2}, \ldots, \TT_{\bcd_n}$,
and also how the
successive blowup of $\Gr^{d,E}$ along $\Gr^{d, E}_{\subset E_{[n]\- 1}} \sqcup \Gr^{d, E}_{\supset F_1},
\ldots, \Gr^{d, E}_{\subset E_{[n]\- n}} \sqcup \Gr^{d, E}_{\supset F_n}$ fits into the picture.
\end{rem}

\begin{rem}\label{log-maps} More generally, 
when  $\dim \TTb=1$, that is, when $E=E_1 \oplus E_2$
for some $E_1$ and $E_2$,
we wonder whether 
the $\fG$-quotient $\HH^{d, E_\bcd}$ should admit an interpretation in terms of moduli of genus-zero
log-maps  into $\Gr^{d,E}$. Some interesting 
related works in this direction
have already been carried out in \cite{FW}
(see in particular, Thereom 1.6 therein).
\end{rem}

\section{$\fG$-quotient $\HH^{d,E}$ and  family 
$\FF^{d,E}$ by the maximal torus}\label{by-max}

In this section, we focus on the important special case of
the $\fG$-quotient $\HH^{d,E}$ and  $\fG$-family 
$\FF^{d,E}$ by the {\it maximal torus} $\TT$.

\subsection{The linear subspace $\LL_\buw \subset \PP_\buw$.} 

We will assume that the symbol
$p_{i_1\cdots i_d}$ is defined for any sequence of distinct integers between 1 and $n$, subject to the relation
\begin{equation}\label{signConvention}
p_{\sigma(i_1)\cdots \sigma (i_d)}=\sgn(\sigma) p_{i_1\cdots i_d}
\end{equation}
for any permutation $\sigma$ of $n$ letters, where $\sgn(\sigma)$ denotes the sign of the permutation.
 Further, let $u_1\cdots u_d$ be any sequence of $d$  integers between 1 and $n$ such that
$u_i=u_j$ for some $1 \le i \ne j \le d$, then we set
$$ p_{u_1\cdots u_d} := 0.$$
Then, associated with any 
$\uh=\{h_1, \cdots, h_{d-1}\} \in I_{d-1,n}$ and $\uk=\{k_1, \cdots, k_{d+1}\} \in I_{d+1,n}$,
 there is a Pl\"ucker relation
\begin{equation} \label{pl-eq}
F_{\uh,\uk}= \sum_{\lambda=1}^{d+1} (-1)^{\lambda-1} p_{h_1\cdots h_{d-1} k_\lambda } p_{k_1 \cdots  \overline{k}_\lambda \cdots k_{d+1}},
\end{equation}
where $\overline{k}_\lambda$ means that the integer $``k_\lambda "$ is deleted from the sequence.

To abbreviate the presentation, 
we frequently express a {\rm general} $\pl$ relation as
\begin{equation} \label{pl-eq-abb}
F=\sum_{s \in S_F} \sgn(s) p_{\uu_s} p_{\uv_s}
\end{equation}
where $S_F$ is an index set, $\uu_s, \uv_s \in \II_{d,[n]}$
for any $s \in S_F$, and $\sgn(s)$ is the $\pm$ sign associated with the term $p_{\uu_s} p_{\uv_s}$.

 We let $\sF$ be the set of all the $\pl$ relations. 

 We let $I_\wp$ be the ideal of $\kk[p_\uu]_{\uu \in \II_{d,[n]}}$ generated by all the 
$\pl$ relations $F \in \sF$.  
 We let $I_{\whwp}$ by the homogeneous ideal of  $\Gr^{d, E}$ in $\PP(\wdeb)$.
Then $$I_{\whwp} \supset I_\wp, \;\; \hbox{in general}.$$
 In characteristic zero, $I_{\whwp} = I_\wp$.
The quotient ring $\kk[p_\uu]_{\uu \in \II_{d,[n]}}/I_\whwp$ is 
the homogeneous coordinate ring of $\Gr^{d, E}$,
called the Grassmannian algebra, in literature.

Consider any $(\uh,\uk) \in \II_{d-1,n} \times \II_{d+1,n}$ such that $F_{\uh,\uk} \in \sF$
 and $F_{\uh,\uk}$ is expressed as \eqref{pl-eq}.
 Then, one sees that 
 $$\sort ((h_1\cdots h_{d-1} k_\lambda) \vee (k_1 \cdots  \overline{k}_\lambda \cdots k_{d+1}))$$
 does not depend on $\lambda$ for any $1 \le \lambda \le d+1$. We let $\buw$ denote
 this unique common sorted pair. Thus, if we let 
$ \uh \vee \uk=(h_1\cdots h_{d-1} k_1) \vee (k_2 \cdots  k_{d+1})$, then
  $$\buw = (\uwo \vee \uwe)={\rm sort} \; (\uh \vee \uk).$$

\begin{defn}\label{defn:linear-pl}
For any given $\pl$ relation \eqref{pl-eq},
 we introduce the following linear relation in the projective space $\PP_\buw$:
 \begin{equation}\label{linear-pl}
 \sum_{\lambda=1}^{d+1} (-1)^{\lambda-1} x_{((h_1\cdots h_{d-1}k_\lambda ),
(k_1 \cdots  \overline{k}_\lambda \cdots k_{d+1}))},\end{equation}
 where $\buw={\rm sort} \; (\uh \vee \uk)$.
We call  \eqref{linear-pl} the linearized $\pl$ relation induced from \eqref{pl-eq}.
We can denote it by $L_{\uh,\uk}$.
Frequently, as in \eqref{pl-eq-abb}, we just write $F_{\uh,\uk}$ as 
$F=\sum_{s \in S} {\rm sign(s)} p_{\uu_s} p_{\uv_s}.$ Then, we will express the corresponding
linearized $\pl$ relation \eqref{linear-pl} as
\begin{equation}\label{linear-pl-abb}L_F=\sum_{s \in S} {\rm sign(s)} x_{(\uu_s,\uv_s)}.
\end{equation}
 \end{defn}

\begin{defn}\label{defn:Lw} Fix any $\buw \in \La^\sort_{d,[n]}$. We let
\begin{equation}\label{linear-subspace}
\LL_\buw =\bigcap_{(\uh \vee \uk)=\buw} (L_{\uh,\uk}=0) \; \subset  \; \PP_\buw
\end{equation} 
 be the linear subspace  of $\PP_\buw$
defined by the linearized $\pl$ equations \eqref{linear-pl}
for all possible $ (\uh,\uk)$ 
such that ${\rm sort} \; (\uh \vee \uk)=\buw.$
\end{defn}

\begin{cor}\label{imageIn} 
The rational map $\Theta_\Gr$ of (\ref{theta-intr-e})  factors as 
$$\xymatrix{
\Theta_\Gr: \Gr^{d, E} \ar @{-->}[r] &   \prod_{\buw \in \Ladnsort} \LL_\buw  \ar @{^{(}->}[r]  & \prod_\uw \PP_\buw }.$$
 Consequently, $\HH^{d,E}$ is contained in $\prod_\buw \LL_\buw$
 and  $\FF^{d,E}$ is contained in $\Gr^{d,E} \times \prod_\buw  \LL_\buw$.
\end{cor}
\begin{proof}
Let  $\bp=[ p_\uu]_{\uu \in \II_{d,[n]}}$ be the Pl\"ucker coordinate of a point of $\UU^{d, E}$.
Then, we have that  \eqref{pl-eq} is satisfied at $\bp$. 
Thus, \eqref{linear-pl} is satisfied at $\Theta_\Gr (\bp)$,
because as the coordinates of $\Theta_\Gr (\bp)$, we have
$$x_{((h_1\cdots h_{d-1}k_\lambda ),
(k_1 \cdots  \overline{k}_\lambda \cdots k_{d+1}))}=p_{(h_1\cdots h_{d-1}k_\lambda )} 
p_{(k_1 \cdots  \overline{k}_\lambda \cdots k_{d+1})}$$ 
for all $(\uh,\uk)$. This implies the first statement.
The last statement follows immediately by combining the above with Corollary \ref{U/TinHilbert}.
\end{proof}

\begin{rem} 
The exceptional locus of $\Theta_\Gr$, that is, the subscheme of $\Gr^{d,E}$ 
defined by the ideal $J^{d, E}$ is  contained in
$\Gr^{d, E}\setminus \UU^{d, E}$, but not equal in general. 
One checks that this is already the case for $\Gr^{2,E}$ when $\dim E=4$.
In this case, the co-support of $J^{2,E}$ is the union of the eight so-called facet-strata:
each corresponds to a facet of the so-called hypersimplex $\Delta^{2,[4]}$ of type (2,4)
and is isomorphic to
a $(\GGm)^4/\GGm$-invariant $\PP^2$, but,
 $\Gr^{2,E} \setminus \UU^{2,E}$ is the union of the six  Schubert divisors
 with respect to all possible Schubert decompositions. 
\end{rem}

\subsection{$\ker^\mh \vi$: $\wp$- and $\frb$-binomials}

Following Corollary \ref{cor2-e}, we will investigate the
defining equations for two closed embeddings
 \begin{eqnarray}\label{embedF-intr-e2} \nonumber
 \xymatrix{
 \FF^\wdep\ar @{^{(}->}[r]  & \PP(\wedge^d E) \times  \prod_{\buw \in \Ladnsort} \PP_\buw,} \;\;
\label{embedF-gr-intr-e2} \nonumber
 \xymatrix{
 \FF^{d,E} \ar @{^{(}->}[r]  & \Gr^{d,E} \times  \prod_{\buw \in \Ladnsort} \PP_\buw.}
 \end{eqnarray}
Indeed, for any subset $\fT$ of $\La_{d,[n]}^\sort$,  we can introduce the natural rational map
\begin{equation}\label{theta-k}
 \xymatrix{
\Theta_{\fT}: 
\PP(\wedge^d E) \ar @{-->}[r]  & \prod_{\buw \in \fT} \PP_{\buw},  } 
\end{equation}
$$
 [p_\uu]_{\uu \in \II_{d,[n]}} \lra  
\prod_{\buw \in \fT}  [p_\uu p_\uv]_{(\uu,\uv) \in \La_{\buw}},
$$   
where $[p_\uu]_{\uu \in \II_{d,[n]}} $ is the homogeneous $\pl$ coordinates of a point in $
\PP(\wedge^d E)$. When restricting $\Theta_{\fT}$ to $\Gr^{d, E}$,
it gives rise to 
\begin{equation}\label{bar-theta-k}
 \xymatrix{
\Theta_{\fT, \Gr}: =\Theta_{\fT}|_{\Gr^{d, E}}: \;
 \Gr^{d, E}  \ar @{-->}[r]  & \prod_{\buw \in \fT} \PP_{\buw}. } 
\end{equation}
Then, fix any $\fT \subset \La_{d,[n]}^\sort$, we can let 
\begin{equation}\label{tA}
 \xymatrix{
 \FF^\wdep_{ \fT *}  \ar @{^{(}->}[r]  & 
\PP(\wedge^d E) \times  \prod_{\buw \in \fT} \PP_{\buw}  
 }
\end{equation}
be the closure of the graph of the rational map $\Theta_{\fT}$, and
 \begin{equation}\label{bar-tA}
 \xymatrix{
\FF^{d, E}_{\fT *} \ar @{^{(}->}[r]  & 
\Gr^{d, E} \times  \prod_{\buw \in \fT} \PP_{\buw}   \subset 
\PP(\wedge^d E)  \times  \prod_{\buw \in \fT} \PP_{\buw}  
 }
\end{equation}
be the closure of the graph of the rational map $\Theta_{\fT, \Gr}=\Theta_{\fT}|_{\Gr^{d, E}}$. 
Note that we have $$\FF^{d, E}_{\fT*} \subset \FF^\wdep_{ \fT*}.$$
The  maximal torus $\TT$ acts on 
$\PP(\wedge^d E)  \times  \prod_{\buw \in \fT} \PP_{\buw}$
by the original action on the left  $ \PP(\wedge^d E)$
and by acting trivially on the right $\prod_{\buw \in \fT} \PP_{\buw}$.
Both of $\FF^{d, E}_{\fT*} \subset \FF^\wdep_{\fT*}$ 
are clearly invariant under this action. 

{\it  In this article, we  only need the scheme in the case when $\fT=\Ladnsort$:
$$ (\FF^{d, E}_* \subset \FF^\wdep_*) =  
(\FF^{d, E}_{\Ladnsort*} \subset \FF^\wdep_{ \Ladnsort*}).$$ 
 The slighly general case, which basically does not require any extra effort,
  is treated here for the purpose of  future reference. The reader can always assume $\fT=\Ladnsort$.}

\subsubsection{Some  algebraic preparations}

\begin{defn}\label{Rft}
Fix any $\fT \subset \Ladnsort$.
We let $$R_{\fT}=\kk[p_\ua; x_{(\uu, \uv)}]_{\ua \in \II_{d,[n]} , (\uu,\uv) \in \bigcup_{\buw \in \fT} \La_{\buw}}.$$
A  polynomial $f \in R_{\fT}$ is called multi-homogeneous if it is homogenous in 
$[p_\ua]_{\ua \in \II_{d,[n]}}$ and homogenous 
in $[x_{(\uu, \uv)}]_{ (\uu,\uv) \in  \La_{\buw}}$, for all $\buw \in \fT$.
We call $x_{(\uu, \uv)}$ a $\vr$-variable of $R_\fT$. 
To distinguish, we call  a $\pl$ variable
 $p_\uu$ with $\uu \in \II_{d,[n]}$, a $\vp$-variable.

A multi-homogeneous polynomial $f \in R_{\fT}$
 is $\vr$-linear if it is linear in $[x_{(\uu, \uv)}]_{ (\uu,\uv) \in  \La_{\buw}}$, 
 whenever it contains some $\vr$-variables of $\PP_{\buw}$,
 for  any $\buw \in \fT$.
 \end{defn}

We set $R_0:=\kk[p_\uu]_{\uu \in \II_{d,[n]}}$. 
Then, $R_{\fT}= R_0[x_{(\uu,\uv)}]_{(\uu,\uv) \in \bigcup_{\buw \in \fT} \La_{\buw}}.$ Moreover, we let
$$R_{\fT,\vr}= \kk[x_{(\uu,\uv)}]_{(\uu,\uv) \in \bigcup_{\buw \in \fT} \La_{\buw}}.$$

Furthermore, corresponding to 
the embedding \eqref{tA}, 
there exists a  homomorphism
\begin{equation}\label{vik}
\vi_{\fT}: \; R_{\fT} \lra R_0  
\end{equation} 
$$
\vi_{\fT}|_{R_0}=\id_{R_0}, \;\;\; 
x_{(\uu,\uv)} \to p_{\uu} p_{\uv}$$ for all  
$(\uu, \uv) \in \La_{\buw}, \; \buw \in \fT$.    

We then let $\vi_{\fT,\Gr}: \; R_{\fT} \lra R_0/ I_\whwp$ be the composition
\begin{equation}\label{bar-vik}
\vi_{\fT,\Gr}: \; R_{\fT} \lra R_0 \lra R_0/{ I_\whwp}. 
\end{equation} 
This corresponds to the embedding \eqref{bar-tA}:   
$\FF^{d, E}_{\fT} \lra 
 \PP(\wedge^d E) \times  \prod_{\buw \in \fT} \PP_{\buw}$. 

As mentioned earlier, we are mainly interested in the case when $\fT=\Ladnsort$. Hence, 
correspondingly, we set 
$$R:=R_{\Ladnsort},\;\; \vi:=\vi_{\Ladnsort}, \;\; \vi_\Gr:=\vi_{\Ladnsort, \Gr}.$$

We let $\ker^\mh \vi_{\fT}$ (resp. $\ker^\mh \vi_{\fT,\Gr}$)
denote the set of all  multi-homogeneous polynomials
 in $\ker  \vi_{\fT}$ (resp. $\ker \vi_{\fT,\Gr}$).


\begin{lemma}\label{defined-by-ker} 
$\FF^\wdep_{ \fT*}$, as a closed subscheme of 
$\PP(\wedge^d E) \times  \prod_{\buw \in \fT} \PP_{\buw}$,
 is defined by $\ker^\mh \vi_{\fT}$.

Similarly, 
$\FF^{d, E}_{\fT*}$, as a closed subscheme of 
$\PP(\wedge^d E) \times  \prod_{\buw \in \fT} \PP_{\buw}$,
 is defined by $\ker^\mh \vi_{\fT,\Gr}$. 
\end{lemma}
\begin{proof} This is immediate.
\end{proof}

We need to investigate $\ker^\mh \vi_{\fT}$ as well as $\ker^\mh \vi_{\fT, \Gr}$. We study $\ker^\mh \vi_{\fT}$  first.

\subsubsection{$\ker^\mh \vi_{\fT}$: $\wp$-binomials} 

Consider any $f \in \ker^\mh \vi_{\fT}$.
 We express it as the sum of its  monomials
$$f= \sum \bm_i.$$
We have $\vi_{\fT} (f)=\sum \vi_{\fT} (\bm_i)=0$ in $R_0$. Thus, the set of the monomials $\{\bm_i\}$ 
can be grouped into minimal groups to form partial sums
of $f$ so that {\it the 
$\vi_{\fT}$-images of elements of each group are 
 identical} and the image of the partial sum of each minimal group equals 0 in $R_0$.
When ch.$\kk=0$, 
this means each minimal group consists of a pair 
$\{\bm_i, \bm_j\}$ with the identical image under $\vi_{\fT}$
and its partial sum
is the difference $\bm_i -\bm_j$.
When ch.$\kk=p>0$ for some prime number $p$, this means each minimal group 
 consists of either  (1): a pair $\{\bm_i, \bm_j\}$ with the identical image under $\vi_{\fT}$
 and $\bm_i -\bm_j$ is a partial sum of $f$;
or (2):   exactly $p$ elements $\{\bm_{i_1}, \cdots,\bm_{i_p}\}$
with the identical image under $\vi_{\fT}$
and $\bm_{i_1}+ \cdots + \bm_{i_p}$ is a partial sum of $f$.
But, the relation $\bm_{i_1}+ \cdots + \bm_{i_p}$, expressed as
$$\bm_{i_1}+ \cdots + \bm_{i_p}-p \; \bm_{i_1}=
\sum_{i=2}^p  (\bm_{i_a}- \bm_{i_1}),$$ is always generated by 
the relations $\bm_{i_a} -\bm_{i_1}$, for all $2 \le a \le p$.
Thus, regardless of the characteristic of the field $\kk$, it suffices to consider binomials
$\bm -\bm' \in \ker^\mh \vi_{\fT}$.

\begin{lemma} \label{trivialB}
Fix any $\buw \in \fT\subset \Ladnsort$. We have
\begin{equation}\label{tildeBk}
p_{\uu'}p_{\uv'}x_{(\uu,\uv)} - p_\uu p_\uv x_{(\uu',\uv')} \in \ker^\mh \vi_\fT,
\end{equation}
where $x_{(\uu,\uv)}, x_{(\uu',\uv')}$ are any two distinct $\vr$-variables of $\PP_{\buw}$.
\end{lemma}
\begin{proof} This is trivial.
\end{proof}

\begin{defn}\label{wp-bino}
Any nonzero binomial  as in \eqref{tildeBk} 
 is called a $\wp$-binomial of $R_\fT$.
We let $\cB^\wp_\fT$ denote the set of all $\wp$-binomials of $R_\fT$.
\end{defn}

\subsubsection{Descendants and partial descendants of 
multi-homogeneous polynomials}

\begin{defn} Fix any subset $\fT \subset \Ladnsort$. Fix any $\buw \in \Ladnsort$. 

Let $f=\sum_i x_{(\uu_i, \uv_i)}\bn_i \in R_\fT \; (\subset R)$ be a multi-homogeneous 
expression/polynomial,
written as above,  which is allowed to be zero,
such that  $\{x_{(\uu_i, \uv_i)}\}$ are a subset of homogeneous coordinates of $\PP_\buw$
for some $\buw \in \fT\subset \Ladnsort$. Set
$$\barf = \sum_i p_{\uu_i}p_{ \uv_i}\bn_i \in R_\fT.$$
Then, we call $f$ a parent of $\barf$ and $\barf$ a descendant of $f$.

Further, if $f$ is a parent of $g$, and $g$ is parent of $h$, then we also say
$f$ is a parent of $h$ and $h$ is a descendant of $f$.
 
Moreover, if $f \in  R_\fT$ does not admit any parent other than itself,
then we say $f$ is root polynomial.
In particular, $f$ is a root parent of $g$ if it is a root polynomial
and a parent of $g$.
\end{defn}

One sees that if $f$ belongs to $\ker^\mh \vi_\fT$ if and only if  any of its descendant does.

\begin{defn} Fix any subset $\fT \subset \Ladnsort$. Fix any $\buw \in \Ladnsort$. 

Let $f=\sum_{i \in I} x_{(\uu_i, \uv_i)}\bn_i \in R_\fT \; (\subset R)$ be a multi-homogeneous 
expression/polynomial,
written as above,  which is allowed to be zero,
such that  $\{x_{(\uu_i, \uv_i)}\}$ are a subset of homogeneous coordinates of $\PP_\buw$
for some $\buw \in \fT\subset \Ladnsort$. 
Let $I= I_1 \sqcup I_2$ be a disjoint union.
Set
$$\barf = \sum_{i \in I_1} p_{\uu_i}p_{ \uv_i}\bn_i
+ \sum_{i \in I_2} x_{(\uu_i, \uv_i)}\bn_i \in R_\fT
 \in R_\fT.$$
Note that $\barf$ is non-homogeneous if $I_1, I_2 \ne \emptyset$.
Then, we call $f$ a parent of $\barf$ and $\barf$ 
a \emph{partial} descendant of $f$.

Further, if $f$ is a parent of $g$, and $g$ is parent of $h$, then we also say
$f$ is a parent of $h$ and $h$ is a partial descendant of $f$.

Moreover, if $f \in  R_\fT$ does not admit any parent other than itself,
then we say $f$ is root polynomial.
In particular, $f$ is a root parent of $g$ if it is a root polynomial
and a parent of $g$.
\end{defn}

One sees that if $f$ belongs to $\ker \vi_\fT$ if and only if  any of its 
partial descendant does.

\begin{exam} 
$$ x_{(\ua,\ub)} x_{(\uu,\uv)} - p_{\ua'}p_{\ub'} x_{(\uu',\uv')}$$
is a partial descendant of 
$x_{(\ua,\ub)} x_{(\uu,\uv)} - x_{(\ua',\ub')} x_{(\uu,\uv)}$.
\end{exam}

\begin{exam}\label{exam-expression}
The following homogeneous {\bf expression}, written as 
\begin{equation}\label{expression}
x_{(\uu', \uv')}x_{(\uu,\uv)} - x_{(\uu, \uv)} x_{(\uu',\uv')} 
\end{equation}
(which is zero as a binomial),
descends to the $\wp$-binomial
$$p_{\uu'}p_{\uv'}x_{(\uu,\uv)} - p_\uu p_\uv x_{(\uu',\uv')} .$$
It is not hard to see that a  $\wp$-binomial does not admit 
any non-zero  root parent.

The following homogeneous binomial of $\ker^\mh \vi$
\begin{eqnarray}\nonumber
\;\; \; x_{(12a,13a')}x_{(13a,2bc)}x_{(12a',3\bar b \bar c)}x_{(12 \bar a,3bc)} x_{(13 \bar a,2 \bar b \bar c)} \;
\nonumber \\
- x_{(13a,12a')} x_{(12a,3bc)}x_{(13a',2\bar b \bar c)}x_{(13 \bar a,2bc)} x_{(12 \bar a, 3 \bar b \bar c)}. \nonumber
\end{eqnarray}  
descends to 
\begin{eqnarray}\nonumber
\;\; \; 
p_{12a}p_{13a'}x_{(13a,2bc)}x_{(12a',3\bar b \bar c)}x_{(12 \bar a,3bc)} x_{(13 \bar a,2 \bar b \bar c)} \;
\nonumber \\
- p_{13a}p_{12a'} x_{(12a,3bc)}x_{(13a',2\bar b \bar c)}x_{(13 \bar a,2bc)} x_{(12 \bar a, 3 \bar b \bar c)}.  \nonumber
\end{eqnarray}  
Also, the first binomial as a  root parent is uniquely determined by the second.
\end{exam}

\subsubsection{$\fb$-reducibility of binomials of $\ker^\mh \vi_\fT$}

\begin{defn}\label{fb-irr} Let 
$f=\bm-\bm' \in \ker^\mh \vi_\fT$. 
 We say $f$ is  ${\fb}$-reducible if
 there exists a decomposition
\begin{equation}\label{fb reducible decom}
\bm=\bm_1\bm_2, \;\; \bm'=\bm_1'\bm_2'
\end{equation}
such that all the monomials in the expression are not constant and 
$$\bm_1-\bm_1', \;\; \bm_2 -\bm_2' \in \ker \vi$$
and every of $\bm_1-\bm_1'$ and  $\bm_2 -\bm_2'$
a descendant or a partial descendant of a multi-homogenous binomial.
Here, when one of the two  
above binomials in the display is zero, say, $\bm_2 -\bm_2'=0$, then we have
 $f=\bm_2(\bm_1-\bm_1')$, as a special case.

We call \eqref{fb reducible decom} a $\fb$-reducible decomposition of
$f$.

We say $f$ is $\fb$-irreducible if it is not $\fb$-reducible.
\end{defn}

Here, in \eqref{fb reducible decom}, we emphasize that $\bm_1-\bm_1'$ 
and $\bm_2 -\bm_2'$ are required to
belong to $\ker \vi$ but are not required to be multi-homogeneous.
This can occur in the following situation. Suppose we have
two multi-homogenous binomials
$$ X\bn_1 -  X'\bn_1', \;\; Y \bn_2 -  Y'\bn_2' \in \ker^\mh \vi_\fT$$
such that $X, X', Y, Y'$ are homogeneous coordinates of
$\PP_\buw$ for some $\buw \in \Ladn$.
Then, the multi-homogenous binomial
$$f= (X\bn_1) (Y\bn_2) -  (X'\bn_1') (Y'\bn_2') \in \ker^\mh \vi_\fT $$
can have a multi-homogeneous descendant,
$$\barf= (\vi(X)\bn_1) (Y\bn_2) - 
 (X'\bn_1') (\vi(Y')\bn_2') \in \ker^\mh \vi_\fT $$
such that  each of $\vi(X)\bn_1 - X'\bn_1'$
and $Y\bn_2- \vi(Y')\bn_2'$ still belongs to $\ker \vi_\fT $
but is not multi-homogeneous\footnote{This point is rather hiden in searching
for an entirely correct proof of Corollary \ref{cor:linear and free}.}.

\begin{lemma}\label{auto homo} Let $f$ 
be as in Definition \ref{fb-irr}. Assume that $f$ 
is a root binomial.
Then, in the 
$\fb$-reducible decomposition 
\eqref{fb reducible decom} of $f$, 
\begin{equation}\label{fb reducible decom 2}
f= \bm_1\bm_2- \bm_1'\bm_2',
\end{equation}
we automatically have
$$\bm_1-\bm_1', \;\; \bm_2 -\bm_2' \in \ker^\mh \vi_\fT.$$
\end{lemma}
\begin{proof}
Suppose not. W.l.o.g., assume 
$\bm_1-\bm_1' \in \ker \vi_\fT \- \ker^\mh \vi_\fT.$
Then, w.l.o.g.,  we can write
$$\bm_1 -\bm_1'= \vi (X) \bn_1 -X' \bn_1'$$
as a partial descendant of  $X \bn_1 -X' \bn_1'$,
where $X$ and $X'$ are homogeneous coordinates of $\PP_\buw$ for
some $\buw \in \Ladn$.
Meanwhile, by definition,
$\bm_2 -\bm_2'$ is also the partial descendant 
$$\bm_2 -\bm_2'=Y \bn_2 -\vi(Y') \bn_2'$$ of a multi-homogeneous binomial
$Y \bn_1 -Y' \bn_1'$,
where $Y$ and $Y'$ are homogeneous coordinates of $\PP_{\buw'}$ for
some $\buw' \in \Ladn$. Because $f$ is multi-homogeneous, we see that 
we must have
$$\buw' =\buw.$$
 Therefore, 
$$f= (\vi (X) \bn_1) (Y \bn_2) - (X' \bn_1') (\vi (Y') \bn_2')$$
for some  $\bn$, implying that $f$ is the  descendant of
$$\hat f= (X \bn_1) (Y \bn_2) - (X' \bn_1') (Y' \bn_2'),$$
a contradiction to that $f$ is a root binomial.

This proves the lemma.
\end{proof}

There are many other binomials of $\ker^\mh \vi$ that do not belong to $\cB^\wp$.

\begin{exam}\label{exam:Bq} Consider $\Gr^{3,E}$. 
Then, the following binomials belong to $\ker^\mh \vi_\fT \- \cB^\wp$.

Fix $a,b,c \in [n]$, all being distinct:
\begin{eqnarray} x_{(12b,13c)}  x_{(23b,12c)} x_{(13b,23c)}  \; \nonumber \\
 - x_{(13b,12c)} x_{(12b,23c)} x_{(23b,13c)}.  \label{rk0-0-3'} 
 \end{eqnarray}
 \begin{eqnarray}
x_{(12a,13b)}x_{(13a,12c)}x_{(12b,13c)} \; \nonumber \\
-x_{(13a,12b)}x_{(12a,13c)}x_{(13b,12c)}. \label{rk0-0-3} 
\end{eqnarray}
\begin{eqnarray}
x_{(12a,13b)}x_{(13a,12c)}x_{(12b,23c)} x_{(23b,13c)} \;\nonumber \\
-x_{(13a,12b)}x_{(12a,13c)}x_{(23b,12c)} x_{(13b,23c)}. \label{rk0-0-4}
\end{eqnarray}
\begin{eqnarray}
x_{(123,3bc)}x_{(13a,2bc)}x_{(12b,23c)} x_{(12a,13b)} \;\nonumber \\
-x_{(13b,23c)}x_{(12a,3bc)}x_{(123,2bc)} x_{(13a,12b)}. \label{rk0-1-4}
\end{eqnarray}
Fix $a,b,c, \bar a, \bar b, \bar c \in [n]$, all being distinct:
\begin{eqnarray}
\;\; \;x_{(12a,3bc)}x_{(13a,2\bar b \bar c)}x_{(13 \bar a,2bc)} x_{(12 \bar a,3 \bar b \bar c)}\; \nonumber \\
-x_{(13a,2bc)}x_{(12a,3\bar b \bar c)}x_{(12 \bar a,3bc)} x_{(13 \bar a,2 \bar b \bar c)}. \label{rk1-1}
\end{eqnarray}

Fix $a,b,c, a', \bar a, \bar b, \bar c \in [n]$, all being distinct:
\begin{eqnarray}
\;\; \; x_{(12a,13a')}x_{(13a,2bc)}x_{(12a',3\bar b \bar c)}x_{(12 \bar a,3bc)} x_{(13 \bar a,2 \bar b \bar c)} \;
\nonumber \\
- x_{(13a,12a')} x_{(12a,3bc)}x_{(13a',2\bar b \bar c)}x_{(13 \bar a,2bc)} x_{(12 \bar a, 3 \bar b \bar c)}.\label{rk0-1}
\end{eqnarray} 
These binomials are arranged so that one sees visibly the matching for multi-homogeneity.
\end{exam}

\subsubsection{$\vr$-linearity and $\vi$-square-freeness}


For any $\bm -\bm' \in \ker^\mh \vi_{\fT}$, we define 
$ \deg_{\vr} (\bm -\bm')$ to be the total degree of $\bm$ (equivalently, $\bm'$)
in $\vr$-variables of $R_{\fT}$.

Observe here that  for any nonzero binomial $\bm -\bm' \in \ker^\mh \vi_{\fT}$,
we automatically have $ \deg_\vr (\bm -\bm') > 0$, since $\vi_{\fT}$ restricts
to the identity on $R_0$.

\begin{lemma}\label{rho=1}
Consider any nonzero binomial $\bm -\bm' \in \ker^\mh \vi_{\fT}$ with $ \deg_\vr (\bm -\bm') = 1$.
Then $\bm -\bm'$ is a multiple of $\wp$-binomial.
\end{lemma}
\begin{proof}
Because $\deg_{\vr}(\bm-\bm') =1$, we can write $$\bm -\bm' = f x_{(\uu,\uv)} -g  x_{(\uu',\uv')} $$ 
for some $f, g \in R_0$,
and two $\vr$-variables  $x_{(\uu,\uv)}$ and $x_{(\uu',\uv')}$ of $\PP_{\buw}$
for some $\buw \in \fT \subset \La_{d,[n]}^\sort$.  
If $x_{(\uu,\uv)} =x_{(\uu',\uv')}$, then one sees that $f=g$ and $\bm -\bm'=0$.
Hence, we assume that  $x_{(\uu,\uv)} \ne x_{(\uu',\uv')}$. Then, we have
$$f p_{\uu} p_{\uv}=g  p_{\uu'} p_{\uv'}.$$
Because  $x_{(\uu,\uv)}$ and $ x_{(\uu',\uv')}$
are two distinct $\vr$-variables of $\PP_{\buw}$,  one checks from the definition that the two sets
$$\{ p_{\uu}, p_{\uv} \}, \; \{p_{\uu'}, p_{\uv'} \}$$ are disjoint. 
Consequently, 
$$  p_{\uu}p_{\uv} \mid g, \;\; p_{\uu'} p_{\uv'}  \mid f.$$
Write $$ g=g_1 p_{\uu}p_{\uv} , \;\;  f=f_1 p_{\uu'} p_{\uv'} .$$
Then we have 
$$ p_{\uu}p_{\uv} p_{\uu'} p_{\uv'}  (f_1-g_1)=0 \in R_0.$$
Hence, $f_1=g_1$. Then, we have
$$\bm -\bm' =h ( p_{\uu'} p_{\uv'}   x_{(\uu,\uv)} -  p_{\uu}p_{\uv}  x_{(\uu',\uv')}) $$ 
where $h:=f_1=g_1$. This implies the statement.
\end{proof}

Observe that if a nonzero binomial $\bm -\bm'$ belongs to
$\ker^\mh \vi_{\fT}$, then we automatically have
 $\deg_\vr (\bm -\bm') >0$.

\begin{defn}\label{no common wp factor}
 Consider any nonzero binomial 
$f=\bm -\bm' \in \ker^\mh \vi_{\fT}$.
We say that $f$ is $\wp$-reducible if there exists 
$p_{\uu'} p_{\uv'}   x_{(\uu,\uv)} -  p_{\uu}p_{\uv}  x_{(\uu',\uv')} \in \cB^\wp_\fT$ 
such that either
$$\hbox{ $p_{\uu'} p_{\uv'}   x_{(\uu,\uv)}  \mid \bm$
 and $x_{(\uu',\uv')} \mid \bm'$}$$
or $$\hbox{    $x_{(\uu,\uv)}  \mid \bm$ and
$p_{\uu}p_{\uv}  x_{(\uu',\uv')} \mid \bm'$.}$$

We say $f$ is $\wp$-irreducible if 
it is not $\wp$-reducible.
\end{defn}
Suppose $f$ us $\wp$-reducible as in the definition, w.l.o.g., say,
$x_{(\uu,\uv)}  \mid \bm$ and 
$  p_{\uu}p_{\uv}  x_{(\uu',\uv')} \mid \bm'$. Then, we can write
$$\bm =x_{(\uu,\uv)}   \bn \; \hbox{and} \;
\bm'= p_{\uu}p_{\uv}  x_{(\uu',\uv')} \bn'$$ for some $\bn, \bn'$.
Then,  using, 
$p_{\uu'} p_{\uv'}   x_{(\uu,\uv)} -  p_{\uu}p_{\uv}  x_{(\uu',\uv')} \in \cB^\wp$,
 we obtain
\begin{equation}\label{fully recover}
\bm- \bm'=x_{(\uu,\uv)}   \bn -p_{\uu}p_{\uv}  x_{(\uu',\uv')} \bn'
\equiv x_{(\uu,\uv)} (  \bn -p_{\uu'}p_{\uv'} \bn'), \; \mod \cB^\wp.
\end{equation}
Furthermore, knowing that the relation 
$p_{\uu'} p_{\uv'}   x_{(\uu,\uv)} -  p_{\uu}p_{\uv}  x_{(\uu',\uv')}$
is applied in the above, we can fully recover the original form of $\bm-\bm'$.

Let $\AA^l$ (resp. $\PP^l$) be the affine (resp. projective)
space of dimension $l$ for some positive integer $l$ with
coordinate variables $(x_1,\cdots, x_l)$ (resp. with
homogeneous coordinates $[x_1,\cdots, x_l]$).
A monomial $\bf m$ is {\it square-free} if 
$x^2$ does not divide $\bf m$ for every coordinate variable $x$ in the affine space.  
A polynomial is square-free if all of its monomials are
square-free.

\begin{defn}
A polynomial $f$ of $R_\fT$ is called $\vi$-square-free if
for any monomial summand $\bm$ of $f$, $\vi(\bm)$ is square-free.
\end{defn}

Recall that multi-homogeneous polynomial $f \in R_{\fT}$
 is $\vr$-linear if it is linear in $[x_{(\uu, \uv)}]_{ (\uu,\uv) \in  \La_{\buw}}$, 
 whenever it contains some $\vr$-variables of $\PP_{\buw}$,
 for  any $\buw \in \fT$.

\begin{lemma}\label{rho=2}
Consider any nonzero binomial 
$\bm -\bm' \in \ker^\mh \vi_{\fT}$ with 
$ \deg_\vr (\bm -\bm') = 2$.
Assume that $\bm$ and $\bm'$ do not have a common factor,
and is $\wp$-irreducible.
Then $$\bm -\bm'=\vi(X_1)X_2X_3-\vi(X_1')X_2'X_3'$$
where $\{X_i, X_i' \}$ are $\vr$-variables of $\PP_{\buw_i}$ 
for some $\buw_i \in \Ladn$ 
such that $\buw_1, \buw_2$ and $\buw_3$ are pairwise distinct.
 In particular, the root parent $X_1X_2X_3 - X_1'X_2'X_3'$  
of $\bm -\bm'$ is $\vr$-linear. 
Furthermore, 
$X_1X_2X_3 - X_1'X_2'X_3'$ is $\vi$-square-free.
\end{lemma}
\begin{proof}
Because $\deg_{\vr}(\bm-\bm') =2$, we can write 
$$\bm -\bm' = f x_{(\ua_1,\ub_1)}x_{(\uu_1,\uv_1)} -g x_{(\ua_2,\ub_2)} x_{(\uu_2,\uv_2)} $$ 
for some $f, g \in R_0$ and 
$\{x_{(\ua_1,\ub_1)},x_{(\ua_2,\ub_2)}\}$
($\{ x_{(\uu_1,\uv_1)},  x_{(\uu_2,\uv_2)} \}$) are $\vr$-variables of $\PP_\buw$ for some $\buw \in \Ladn$ ($\buw' \in \Ladn$).
Consider 
$$ f p_{\ua_1}p_{\ub_1}x_{(\uu_1,\uv_1)} -g p_{\ua_2} p_{\ub_2} x_{(\uu_2,\uv_2)} .$$ 
Because $\bm-\bm'$ is $\wp$-irreducible, we must have
$ x_{(\ua_1,\ub_1)} \ne x_{(\ua_2,\ub_2)}$, equivalently, 
$\{ p_{\ua_1}, p_{\ub_1} \} \cap \{ p_{\ua_2}, p_{\ub_2} \}=\emptyset$,
 for otherwise $\bm$ and $\bm'$ have a common factor,
contradicting to the assumption;
$ p_{\ua_1}p_{\ub_1} \nmid g$, for otherwise 
$ p_{\ua_1}p_{\ub_1}x_{(\ua_2,\ub_2)} \mid \bm'$,
contradicting to that $\bm-\bm'$ is not $\wp$-reducible.
This leaves only one possibility: 
$\{ p_{\ua_1}, p_{\ub_1} \} \cap \{ p_{\uu_2}, p_{\uv_2} \}$ is a singleton,
w.l.o.g., say, 
$$\hbox{$ p_{\ua_1}=p_{\uu_2}$, and then $p_{\ub_1} \mid g$.}$$
By symmetry, the similar arguments when applied to 
$p_{\ua_2} p_{\ub_2}$ will give us, 
w.l.o.g., say, 
$$\hbox{$ p_{\ua_2}=p_{\uu_1}$, and then $p_{\ub_2} \mid f$.}$$
Hence,
$$\bm -\bm' = {\bar f}p_{\ub_2} x_{(\ua_1,\ub_1)}x_{(\ua_2,\uv_1)} 
-{\bar g}p_{\ub_1} x_{(\ua_2,\ub_2)} x_{(\ua_1,\uv_2)} .$$ 
Because $\bm -\bm' \in \ker^\mh \vi_{\fT}$, we obtain
$${\bar f} p_{\uv_1}={\bar g} p_{\uv_2}.$$
Because $\bm$ and $\bm' $ admit no common factor, we obtain
$${\bar f}=p_{\uv_2}, \; {\bar g} =p_{\uv_1}.$$
That is
$$\bm -\bm' =p_{\uv_2} p_{\ub_2} x_{(\ua_1,\ub_1)}x_{(\ua_2,\uv_1)} 
-p_{\uv_1}p_{\ub_1} x_{(\ua_2,\ub_2)} x_{(\ua_1,\uv_2)} .$$ 
Because $\ua_1 \vee \ub_1 = \ua_2 \vee \ub_2$ and $\ua_2 \vee \uv_1 =\ua_1 \vee \uv_2$,
we obtain 
$$\ua_1 \vee \ub_1 \vee \ua_2 \vee \uv_1 = \ua_2 \vee \ub_2 \vee \ua_1 \vee \uv_2,$$
hence $$ \ub_1  \vee \uv_1 =  \ub_2 \vee \uv_2.$$
This implies $\bm -\bm' $ is a descendant of 
$$x_{(\uv_2, \ub_2)} x_{(\ua_1,\ub_1)}x_{(\ua_2,\uv_1)} 
-x_{(\uv_1, \ub_1)} x_{(\ua_2,\ub_2)} x_{(\ua_1,\uv_2)} .$$
We have $\buw_1 \ne \buw_2$.
Assume on the contrary
that $\buw_1 \ne \buw_2$. Then 
$\uv_2 \vee \ub_2 = \ua_1 \vee \ub_1 =\ua_2 \vee \ub_2$,
hence $\uv_2= \ua_2$. But, then $\ua_1 = \uv_1$,
implying a common factor $x_{(\uv_1, \uv_2)}$,  equal to both
$x_{(\ua_2,\uv_1)}$ and $x_{(\ua_1,\uv_2)}$.
The other cases are totally similar whose details  we omit. Hence
$\buw_1, \buw_2$ and $\buw_3$ are pairwise distinct.

Assume on the contrary  that 
$\vi(X_1X_2X_3)$ is not square-free. For example, 
Assume 
$$\{\uv_2, \ub_2 \} \cap \{\ua_1,\ub_1\} \ne \emptyset.$$
Because of $x_{(\ua_1,\uv_2)}$,  we see that $\uv_2 \ne \ua_1$.
If $\uv_2=\ub_1$, then $\uv_2 \vee \ub_2=\uv_1\vee \ub_1$
implies $\ub_2=\uv_1$. Then $x_{(\ub_1, \ub_2)}$ is a common factor 
of $\bm-\bm'$, a contradiction. Because 
$\ua_1 \vee \ub_1=\ua_2 \vee \ub_2$, we see that 
$\ub_2 \ne \ua_1,\ub_1$ because either case would imply 
a common factor $x_{(\ua_1,\ub_1)}$. Hence, we conclude
$\{\uv_2, \ub_2 \} \cap \{\ua_1,\ub_1\} = \emptyset.$
All the remaining cases follow from similar routine checks.

This establishes the lemma.
\end{proof}

\begin{exam}\label{p-reducible}
Consider the following binomial:
$$f=p_{123}p_{13a}p_{2bc}x_{(13b,23c)}x_{(12a,3bc)}-
p_{13b}p_{23c}p_{12a}x_{(123,3bc)}x_{(13a,2bc)}.$$
This binomial belongs to $\ker^\mh (\vi)$, is
$\fb$-irreducible, but $\wp$-reducible.
Using the $\wp$-binomial
$p_{13a}p_{2bc}x_{(12a,3bc)} -p_{12a}p_{3bc} x_{(13a,2bc)}$,
 $f$ can be reduced to
$$p_{12a}p_{(13a,2bc)}(p_{123}p_{3bc}x_{(13b,23c)}
-p_{13b}p_{23c}x_{(123,3bc)}), \mod \cB^\wp.$$
Or, using the $\wp$-binomial
$p_{13b}p_{23c}x_{(123,3bc)} - p_{123}p_{3bc}x_{(13b,23c)}$,
 it can also be reduced to
$$p_{123}p_{(13b,23c)}(p_{13a}p_{2bc}x_{(12a,3bc)}
-p_{12a}p_{3bc}x_{(13a,2bc)}), \mod \cB^\wp.$$
In addition, observe that the above binomial $f$ is a root binomial.
\end{exam}

Indeed, any 
$\fb$-irreducible and $\wp$-reducible 
binomial $f=\bm-\bm' \in \ker^\mh  \vi_\fT$ with
$\deg_\vr f=2$ takes a form like the above.

\begin{lemma}\label{rho=2 wp-reducible}
Suppose $\bm-\bm' \in \ker^\mh  \vi_\fT$ with
$\deg_\vr (\bm-\bm')=2$ is $\fb$-irreducible and $\wp$-reducible.
Then, it takes of the following form
$$p_{\uu'}p_{\uv'} p_{\ua'} x_{(\ua,\ub)} x_{(\uu, \uv)}-
p_{\ua}p_{\ub} p_\uu x_{(\ua',\uv)} x_{(\uu', \uv')}$$
where $x_{(\ua,\ub)}, x_{(\ua',\uv)}$ are homogeneous coordinates
of $\PP_\buw$ for some $\buw$ and 
$x_{(\uu,\uv)}, x_{(\uu',\uv')}$ are homogeneous coordinates
of $\PP_{\buw'}$ for some $\buw' \ne \buw$. In particular,
$\bm-\bm'$ is $\vr$-linear. Moreover, $\bm-\bm'$ is also
$\vi$-square-free.
\end{lemma}
\begin{proof} 
Because $\bm-\bm'$ is $\wp$-reducible, w.l.o.g., we can assume
$$\bm-\bm'= f x_{(\ua,\ub)} x_{(\uu, \uv)}-
g p_{\ua}p_{\ub}  x_{(\ua',\ub')} x_{(\uu', \uv')}$$
for some $f, g \in R_0$.
Then, we have 
$$f  x_{(\uu, \uv)}- g p_{\ua'}p_{\ub'}   x_{(\uu', \uv')} \in \ker^\mh \vi_\fT.$$
By Lemma \ref{rho=1}, the above must be equal to 
$$hp_{\uu'}p_{\uv'} x_{(\uu, \uv)}-  hp_{\uu}p_{\uv}   x_{(\uu', \uv')}$$
for some $h \in R_0$. Hence, by comparing the coefficients, 
we obtain
$$f=hp_{\uu'}p_{\uv'} ,\; g p_{\ua'}p_{\ub'}=hp_{\uu}p_{\uv}.$$
Because $h \mid f$, 
using that $\bm-\bm'$ is $\fb$-irreduicble 
(hence having no common factor),
we have  that $h$ and $g$ are co-prime,  
therefore $h \mid p_{\ua'}p_{\ub'}$. Observe that
$h$ can not equal to $p_{\ua'}p_{\ub'}$, for otherwise, $p_{\ua'}p_{\ub'} \mid f$, and we obtain 
$$\hbox{$p_{\ua'}p_{\ub'}  x_{(\ua,\ub)} \mid \bm$ and
$p_{\ua}p_{\ub}  x_{(\ua',\ub')} \mid \bm'$},$$
then this would imply that $\bm -\bm'$ is $\fb$-reducible, 
a contradiction.
W.l.o.g., we assume $h =p_{\ua'}$. Then $g p_{\ub'} = p_{\uu}p_{\uv}$.
W.l.o.g., we can assume $p_{\ub'}=p_\uv$. 
Then, we obtain $$f=p_{\ua'}p_{\uu'}p_{\uv'} ,\; g=p_{\uu},
\; \ub'=\uv.$$
This implies 
$$\bm-\bm'=p_{\uu'}p_{\uv'} p_{\ua'} x_{(\ua,\ub)} x_{(\uu, \uv)}-
p_{\ua}p_{\ub} p_\uu x_{(\ua',\uv)} x_{(\uu', \uv')}.$$

 Suppose on the contrary $\buw'=\buw$. Then
from  $x_{(\uu, \uv)}$ and $x_{(\ua',\uv)}$ being homogeneous coordinates
of the same projective space $\PP_\buw=\PP_{\buw'}$, we would obtain
$\uu=\ua'$, producing a common factor $x_{(\uu, \uv)}$ of $\bm-\bm'$,
a contradiction.
This implies that $\bm-\bm'$ is $\vr$-linear.

To show that $\bm-\bm'$ is $\vi$-square-free, 
w.l.o.g., say, $\ua=\uu$. Then, we have
$$\ua \vee \ub = \ua' \vee \uv \;\; \hbox{and} \;\;
  \ua \vee \uv=\uu' \vee \uv'.$$
Hence, 
$$\ua \vee \ub \vee  \ua \vee \uv= \ua' \vee \uv \vee \uu' \vee \uv'.  $$
Thus, $$\ua \vee \ub \vee  \ua=\ua' \vee \uu' \vee \uv'.  $$
Note that $$\ua' \subset \ua \vee \ub=\ua' \vee \uv.$$
Therefore, as sets of  integers (allowing repetition), we have
\begin{equation}\label{wrong id}
\{\ua \vee \ub \vee  \ua\} \; \- \{\ua'\}= \uu' \vee \uv'
\end{equation}
Observe also that as a set of  integers (allowing repetition), 
$\ua'$ is made of exactly half of
$\ua \vee \ub=\ua' \vee \uv$, hence as a set of 
\emph{\it distinct} integers, we see that 
$$\ua \vee \ub \vee  \ua \- \{\ua'\}$$
contains at most $d$ distinct integers. But, by the definition of $\Ladnsort$ in \S \ref{max torus},
$$\uu' \vee \uv'$$
contains at least $d+2$ distinct integers, a contradiction to the identity
\eqref{wrong id}. The remaining cases such as 
$\ua'=\uu'$, etc., are totally parallel, whose details are omitted.

This establishes the lemma.
\end{proof}


\begin{lemma}\label{vi X stays}
Consider any binomial $f \in \ker^\mh \vi_\fT$.
Assume $f= \vi(X) \bm - \vi(X') \bm'$ 
is a descendant of $X \bm - X' \bm'$ where
$X$ and $X'$ are homogeneous coordinates of $\PP_\buw$ for some $\buw \in \Ladn$. Suppose 
$$f = \bm_1 \bm_2 - \bm_1' \bm_2'$$ is a
$\fb$-reducible decomposition of $f$. Then, $\vi(X) \mid \bm_1$
or $\vi(X) \mid \bm_2$. Similarly, $\vi(X') \mid \bm_1'$
or $\vi(X') \mid \bm_2'$.
\end{lemma} 
\begin{proof}
We apply induction on $\deg_\vr(\bm -\bm' )$.
By  applying Lemmas \ref{rho=1}, \ref{rho=2}, 
and \ref{rho=2 wp-reducible}, we see that the lemma
holds when $\deg_\vr(\bm -\bm' ) \le 2$.

Assume that $\deg_\vr(\bm -\bm' )=e$
and the lemma holds when $\deg_\vr =e-1$.
Then, we can write $f$ as
$$f= \vi(X) Y \bn - \vi(X') Y' \bn'$$ such that
$Y$ and $Y'$ are homogeneous coordinates of $\PP_{\buw'}$ for some
$\buw' \in \Ladn$.
In the decomposition $f = \bm_1 \bm_2 - \bm_1' \bm_2'$,
we have
\begin{enumerate}
\item either $Y$ and $Y'$ appear in exactly one of $\bm_1  - \bm_1' $
and  $\bm_2 - \bm_2'$, w.l.o.g., say, in $\bm_1  - \bm_1' $,
\item or else, w.l.o.g., say, $Y$ appears in $\bm_1$ and $Y'$ appears in $\bm_2'$.
\end{enumerate}
In the former case, we write $\bm_1 = Y \bn_1$ and $\bm_1'=Y'\bn_1'$,
Then, we have
$$f =(Y \bn_1) \bm_2 - (Y' \bn_1') \bm_2'.$$
Note that $\vi(X) \mid \bn_1 \bm_2$ and $\vi(X') \mid \bn_1' \bm_2'.$
Consider $$\barf=(\vi(Y) \bn_1) \bm_2 - (\vi(Y') \bn_1') \bm_2'.$$ 
Because $\deg_\vr \barf =e-1$, by induction, we have
$\vi(X) \mid \bn_1 \mid \bm_1$ or $\vi(X) \mid  \bm_2$. Likewise,
$\vi(X') \mid \bn_1' \mid \bm_1$ or $\vi(X') \mid  \bm_2'$. This implies
the statement.

In the latter case, we write $\bm_1 = Y \bn_1$ and $\bm_2'=Y'\bn_2'$,
Then, we have
$$f =(Y \bn_1) \bm_2 - \bm_1' (Y' \bn_2') $$
Note that $\vi(X) \mid \bn_1 \bm_2$ and $\vi(X') \mid \bm_1' \bn_2'.$
Consider $$\barf=(\vi(Y) \bn_1) \bm_2 -  \bm_1' (\vi(Y')\bn_2').$$ 
Because $\deg_\vr \barf =e-1$, by induction, we have
$\vi(X) \mid \bn_1 \mid \bm_1$ or $\vi(X) \mid  \bm_2$. Likewise,
$\vi(X') \mid \bm_1' $ or $\vi(X') \mid \bn_2' \mid  \bm_2'$. This implies
the statement in this case.

All in all, the lemma is proved.
\end{proof}

\begin{lemma}\label{linear and free} 
Consider any binomial $f \in \ker^\mh \vi_\fT$.
Assume $f$ is either  not $\vr$-linear or not $\vi$-square-free.
Then $f$ is $\fb$-reducible.
\end{lemma}
\begin{proof}
Let $f \in \ker^\mh \vi_\fT$ be a binomial.
Assume $f$ is neither  $\vr$-linear nor $\vi$-square-free. 

We apply induction on $\deg_\vr(\bm -\bm' )$.
By  applying Lemmas \ref{rho=1}, \ref{rho=2}, 
and \ref{rho=2 wp-reducible}, we see that the lemma
holds when $\deg_\vr(\bm -\bm' ) \le 2$.

Suppose now $\deg_\vr f  =e > 2$.
We assume that the lemma holds when $\deg_\vr   =e-1$.

We can express
$f$ as
$$f= XY \bm - \hbox{(the other term)}$$ such that
either $X$ and $Y$ are homogeneous coordinates of $\PP_\buw$
for  some $\buw \in \Ladn$  when $f$ is not $\vr$-linear,
or else, $X$ and $Y$ are two variables  such that
$\vi(X) \vi(Y)$ is not square-free when $f$ is not $\vi$-square-free.
Because $\deg_\vr f  > 2$, there are $Z$ and $Z'$
such that $Z$ and $Z'$ are homogeneous coordinates of $\PP_{\buw'}$ for  some $\buw' \in \Ladn$ (which may be equal to $\buw$), and we can write 
$$f= XY Z\bn - Z' \bn'$$
for some $\bn$ and $\bn'$.
Then, we consider
$$\barf= XY \vi(Z) \bn -  \vi(Z') \bn' .$$
Clearly, $\barf$ is either  
not $\vr$-linear, or else, not $\vi$-square-free. 
Because $\deg_\vr \barf  =e  -1$, by induction, we
have a $\fb$-reducible decomposition
$$\barf=  (X \bn_1) (Y\bn_2) - (\bn_1') (\bn_2')  $$
such that $X \bn_1 -\bn_1'$ and $Y\bn_2 -  \bn_2'$, 
 are (partial) descendant of a multi-homogenous binomial, 
and belong to $\ker\vi_\fT$.
By Lemma \ref{vi X stays}, applied to the $\fb$-reducible decomposition 
$\barf = (X \bn_1) (Y\bn_2) - (\bn_1') (\bn_2') $,
we obtain that $\vi(Z)$ and $\vi(Z')$ appears in exactly one of
$ X \bn_1 -\bn_1'$ and  $Y\bn_2 - \bn_1'$, w.l.o.g., say. in 
$ X \bn_1 -\bn_1'$, or else, w.l.o.g., say, 
$\vi (Z)$ appears in $ X \bn_1$ and $\vi(Z')$ appears in $\bn_2'$.

In the former case, we write  $\bn_1 = \vi(Z) \bn_{11}$ and
$\bn_1'= \vi(Z') \bn_{11}'$. Then, we have
$$\barf = (X \vi(Z) \bn_{11}) (Y\bn_2) - (\vi(Z')\bn_{11}') (\bn_2') .$$
Hence, 
$$f = (X Z \bn_{11}) (Y\bn_2) - ( Z'\bn_{11}') (\bn_2') $$
which implies the desired statement.

In the latter case, we write  $\bn_1 = \vi(Z) \bn_{11}$ and
$\bn_2'= \vi(Z') \bn_{21}'$. Then, we have
$$\barf = (X \vi(Z) \bn_{11}) (Y\bn_2) - (\bn_1') (\vi(Z')\bn_{21}') .$$
Hence, 
$$f = (X Z \bn_{11}) (Y\bn_2) - (\bn_1') (Z'\bn_{21}') .$$
which also implies the desired statement.

All in all, the lemma is proved.
\end{proof}

\begin{cor}\label{cor:linear and free}
Suppose a binomial $f \in \ker^\mh \vi_\fT$  is $\fb$-irreducible.
Then it is both $\vr$-linear and $\vi$-square-free.
\end{cor}

\subsubsection{$\fb$-irreducible root binomials of 
$\ker^\mh \vi_\fT$}

\begin{cor}\label{reduce to rb}
Let $f=\bm-\bm'$ be any  root binomial in $\ker^\mh \vi_\fT$.
Then it is generated by $\fb$-irreducible  root binomials 
in $\ker^\mh \vi_\fT$.

Moreover, if $f$ is $\wp$-irreducible, then it is generated by 
$\fb$-irreducible  root binomials that are also $\wp$-irreducible.
\end{cor}
\begin{proof} Take any binomial
$(\bm -\bm') \in \ker^\mh \vi_\fT$
 with $\deg_\vr (\bm -\bm') \ge 2$. 
By applying Lemma \ref{auto homo}, we can express
$$\bm -\bm' =\prod_{i=1}^\ell \bm_i -  \prod_{i=1}^\ell \bm_i',$$ 
such that $\bm_i -\bm'_i \in \ker^\mh \vi_\fT$ and
are $\fb$-irreducible root binomials,   for all $i \in [\ell]$.
Then, we have
$$\bm -\bm'=\prod_{i=1}^{\ell} \bm_i - \bm_\ell'  \prod_{i=1}^{\ell-1} \bm_i
+ \bm_\ell'  \prod_{i=1}^{\ell-1} \bm_i -\prod_{i=1}^{\ell} \bm_i',
$$ 
$$=(\bm_\ell - \bm_\ell')  \prod_{i=1}^{\ell-1} \bm_i
+ \bm_\ell'  (\prod_{i=1}^{\ell-1} \bm_i -\prod_{i=1}^{\ell-1} \bm_i').$$ 
Thus, by a simple induction on the integer $\ell$, 
we conclude that $f$ is generated by 
$\bm_i -\bm_i' \in \ker^\mh \vi_\fT$, $i \in [\ell]$, 
which are $\fb$-irreducible.

Moreover, 
since $\bm -\bm' =\prod_{i=1}^\ell \bm_i - 
 \prod_{i=1}^\ell \bm_i'$,
it follows that if any $\bm_i -\bm_i'$ is $\wp$-reducible, so is $\bm-\bm'$.
That is, if $\bm -\bm'$ is $\wp$-irreducible, 
so are all $\bm_i -\bm_i'$, $i \in [\ell]$.

This establish the corollary.
\end{proof}

\begin{defn} \label{defn:frb}
A $\fb$-irreducible root binomial that is also $\wp$-irreducible
is called $\frb$-irreducible. We let 
 $\cB^\frb_\fT$ be the set of all $\frb$-irreducible binomials of $\ker^\mh \vi_\fT$.  A binomial of $\cB^\frb_\fT$ is called
a $\frb$-binomial.
\end{defn}

{\it To keep the exposition concise, 
in the remainder of this section, we will focus on the main case 
when $\fT=\Ladnsort$, hence, we will
drop the subindex $``$ $_\fT$ $"$, e.g.,
 $$R=R_{\Ladnsort}, R_\vr=R_{\Ladnsort,\vr},
\cB^{\frb}=\cB^{\frb}_{\Ladnsort}, 
\vi=\vi_{\Ladnsort}, \; \hbox{etc.}$$  We comment, however, that much
of our discussions below still hold or can be straightforwardly extended  to the general subset $\fT$.}

\subsection{Equations defining $\FF^{d,E}_*$ and proofs of Proposition \ref{eqs-F'-intro} and Theorem \ref{eqs-F-intro}}
\label{sec:eqForZ}

\begin{defn}\label{fv-sR} Fix $\um \in \II_{d,[n]}$.
For every $\buw \in \La^\sort_{d,[n]}$, choose and fix an arbitrary element 
$(\ua_\buw,\ub_\buw) \in \La_\buw$.
Then, the scheme $\PP(\wedge^d E) \times  \prod_{\buw \in \Ladnsort} \PP_{\buw}  $ 
is covered by the open affine space charts
of the form  $$(p_\um \ne 0) \times \prod_{\buw \in \La^\sort_{d,[n]}}
(x_{(\ua_\buw,\ub_\buw)} \ne 0) \subset 
\PP(\wedge^d E) \times \prod_{\buw \in \La^\sort_{d,[n]}} \PP_\buw  .$$ 
We call such an affine open chart a standard chart of
$\PP(\wedge^d E) \times \prod_{\buw \in \La^\sort_{d,[n]}} \PP_\buw .$ 
\end{defn}

Note that any standard chart of
$\PP(\wedge^d E) \times  \prod_{\buw \in \Ladnsort} \PP_{\buw}  $ in the above definition is uniquely indexed
 by an element of the set 
$\II_{d,[n]} \times \prod_{\buw \in \La_{d,[n]}^\sort} \La_\buw.$

Fix any standard chart $\fV$ of 
 $\PP(\wedge^d E) \times  \prod_{\buw \in \Ladnsort} \PP_{\buw}  $ 
 and a multi-homogeneous polynomial $f$ of
$R$, we let $f|_\fV$ denote the restriction of $f$ to $\fV$. Suppose $\mathcal S$ is a collection of
multi-homogeneous polynomials of $R$,  we let ${\mathcal S}|_\fV$ be the set of the restrictions of the polynomials of $\mathcal S$ to $\fV$.

\subsubsection{Proof of Proposition \ref{eqs-F'-intro}.}

\begin{proof} By the first statement of
Lemma \ref{defined-by-ker} 
(the case when $\fT=\Ladnsort$),
the scheme $\FF^\wdep_*$, as a closed subscheme of   
$\PP(\wedge^d E) \times  \prod_{\buw \in \Ladnsort} \PP_{\buw} $,
 is defined by $\ker^\mh \vi_\fT$.

\noindent
{\bf Claim 1.} Suppose
a binomial $f =\bm-\bm' \in \ker^\mh \vi_\fT$
is not a root binomial. Then $f$
can be reduced to a root binomial in $\ker^\mh \vi_\fT$.

By Definition \ref{fv-sR}, 
 $\PP(\wedge^d E) \times  \prod_{\buw \in \Ladnsort} \PP_{\buw} 
 $ is covered by standard affine charts.

Fix any standard chart $\fV$ of  $\PP(\wedge^d E) \times  \prod_{\buw \in \Ladnsort} \PP_{\buw}  $ as in Definition \ref{fv-sR}. 
We want to reduce   $f|_\fV$ 
to a root binomial in $\ker^\mh \vi_\fT$, restricted to the chart $\fV$.

We prove it by induction on $\deg_\vp (f)$, 
where $\deg_\vp f$ denote the degree of $f$ considered as
a polynomial in $\vp$-variables only

When  $\deg_\vp (f) =0$, the statement holds trivially.

Assume that statement holds for $\deg_\vp< e$ for some $e>0$.

Consider $\deg_\vp (f) = e$.

By definition, we can write 
$$f= p_{\uu_s} p_{\uv_s} \bn_s - p_{\uu_t} p_{\uv_t} \bn_t$$
with  $(\uu_s,\uv_s)$ and $(\uu_t,\uv_t)$ belonging to $\La_\buw$ for some 
$\buw \in \Ladnsort$, and $\bn_s, \bn_t \in R$.

Now take any $(\ua_\buw, \ub_\buw) \in \La_\buw$.
Consider the following two relations in $\cB^\wp$
$$p_{\uu_s} p_{\uv_s} x_{(\ua_\buw, \ub_\buw)}- p_{\ua_\buw}p_{\ub_\buw} x_{(\uu_s, \uv_s)}, \;\;
p_{\uu_t} p_{\uv_t} x_{(\ua_\buw, \ub_\buw)}- p_{\ua_\buw}p_{\ub_\buw} x_{(\uu_t, \uv_t)}.$$
This implies that
$$x_{(\ua_\buw, \ub_\buw)}f
\equiv p_{\ua_\buw}p_{\ub_\buw}(x_{(\uu_s, \uv_s)} \bn_s - x_{(\uu_t, \uv_t)} \bn_t),
\; \mod  \cB^\wp.$$
Note that on the chart $\fV$,  $x_{(\ua_\buw, \ub_\buw)}$ is invertible 
(see Definition \ref{fv-sR}).
Also, observe that $(x_{(\uu_s, \uv_s)} \bn_s - x_{(\uu_t, \uv_t)} \bn_t) \in \ker^\mh_{\fT}$.
Since  $\deg_\vp (x_{(\uu_s, \uv_s)} \bn_s - x_{(\uu_t, \uv_t)} \bn_t) < e$, 
the statement then follows from the inductive assumption.

{\bf Claim 2.} Any $\wp$-reducible 
root binomial $f =\bm-\bm' \in \ker^\mh \vi_\fT$
 can be reduced to $\wp$-irreducible root binomials,
modulo $\cB^\wp$.

Since $f$ is $\wp$-reducible, by using the notation immediately above
\eqref{fully recover}, w.l.o.g., we obtain
\begin{equation}\label{fully recover 2}
f=\bm- \bm'=x_{(\uu,\uv)}   \bn -p_{\uu}p_{\uv}  x_{(\uu',\uv')} \bn'
\equiv x_{(\uu,\uv)} (  \bn -p_{\uu'}p_{\uv'} \bn'), \; \mod \cB^\wp.
\end{equation}
Because, $f$ is a root binomial, one sees that 
$g=\bn -p_{\uu'}p_{\uv'} \bn'$ must also be a root binomial.
Now if $g$  is $\wp$-irreducible, then we are
done. Otherwise, we  repeat and apply the same procedure to
$g$. As the $\deg_\vr g < \deg_\vr f$, the process must terminates.
Thus, the claim holds.

Then, by applying Corollary \ref{reduce to rb},
  we obtain that 

{\bf Claim 3} Any $\wp$-irreducible
 root binomial $f =\bm-\bm' \in \ker^\mh \vi_\fT$
can be reduced to $\fb$-irreducible root binomials in $\ker^\mh \vi_\fT$
that are also $\wp$-irreducible, that is,
 reduced to $\frb$ binomials of $\ker^\mh \vi_\fT$.told

Combining Claims 1, 2, and 3,
we obtain that the defining equations of $\ker^\mh \vi_\fT$
can be reduced to $\cB^\wp \sqcup \cB^{\frb}$.
This proves the proposition.
 \end{proof}

Next, we study $\ker^\mh \vi_{\Gr}$, where $\vi_\Gr=\vi_{\Ladnsort, \Gr}$.

\subsubsection{On the non-toroidal part of $\ker^\mh \vi_{\Gr}$} $\ $

{\it To suppress the issue of multivariate 
$\pl$ relations, we assume 
 in the sequel
 that the base field $\kk$ has characteristic zero
so that $I_\whwp=I_\wp$.}

We let $f \in R$ be any multi-homogenous polynomial such that 
 $\vi_\Gr(f)=0$. Then, by \eqref{bar-vik},
 it holds if and only if $\vi (f) \in I_\wp$.
 Thus, we can express $f=\sum_{F \in \sF} f_F$ such that 
 $\vi (f_F)$ is a multiple of $F$ for all $F \in \sF$. It suffices to consider an
 arbitrarily fixed $F \in \sF$. Hence, we may assume that $f=f_F$ for some arbitrarily fixed $F \in \sF$.
 That is, in such a case,  $\vi_\Gr (f)=0$ if and only if
  $$\hbox{$\vi  (f) =hF$  for some $h \in R_0$.}$$
  To exclude the (trivial) case when $f=hF$, 
we  assume $\deg_\vr f>0$.
  We  write $F=\sum_{s\in S_F} \sgn (s) p_{\uu_s}p_{\uv_s}$. 
  Accordingly, we can express
  \begin{equation}\label{initial f}
  f=\sum_{s \in S_F} \sgn (s) f_s
  \end{equation} such that 
\begin{equation}\label{initial f 2} \hbox{$\vi (f_s) = h p_{\uu_s}p_{\uv_s}$, for all $s \in S_F$.}
\end{equation}

Without loss of generality, in the sequel, we focus on investigating the case when {\it all the terms of $f$ 
in the expression of \eqref{initial f} are co-prime,
 i.e., having no common factors, modulo $\ker^\mh \vi$.}

We let $L_\sF$ be the set of $L_F$ for all $F \in \sF$ (see Definition \ref{defn:linear-pl}). 
One sees  that  $\vi (L_F)=F$  for all $F \in \sF$. 

\begin{defn}\label{defn:h-pl}
Let  $f \notin L_\sF$ with $\deg_\vr f>0$ be as in \eqref{initial f} such that
 $\vi (f) =h F$ and all the terms of $f$ as
in the expression of \eqref{initial f} are co-prime,
then we say $f$ is a $\hpl$ relation.

We denote the set of all $\hpl$  relations by  $ \sF^{\vp\vr}$.
\end{defn}

From definition and by the above discussions, we have

\begin{lemma}\label{ker-hF-generate}
 Assume that the base field has characteristic zero. Then
$$
 \ker^\mh \vi_{\Gr}=\langle \ker^\mh \vi, \sF, L_\sF, \sF^{\vp\vr}  \rangle.$$
\end{lemma}

The following are $\hpl$ relations.

\begin{exam} \label{exam-h pl}
Consider $\Gr^{3,E}$. Then, the following binomials belong to $\ker^\mh \vi_\Gr$.  

Fix $a,b,c \in [n] \- \{1,2,3\}$, all being distinct:
\begin{eqnarray} 
f: p_{1ab}p_{12c}p_{13c} x_{(123, 1ac)}- p_{1ac}p_{12c}p_{13b}x_{(12a,13c)}  
 + p_{1ac}p_{12b}p_{13c} x_{(13a,12c)}.   \label{odd vr pl} 
 \end{eqnarray}
 Here, $\vi(f)=  p_{1ac}p_{12c}p_{13c} F$, where
$F= p_{123} p_{1ab}- p_{12a}p_{13b} + p_{13a}p_{12b}.$

Fix $a,b,1', b' \in [n] \- \{1,2,3\}$, all being distinct:
\begin{eqnarray} 
f: (p_{1ab}p_{1'2a}p_{1'2b'} x_{(1'ab, 123)}- p_{12a}p_{1'ab}p_{1'2b'}x_{(1'2a,13b)} )
x_{(1'3a,12b')}  \\ \nonumber 
+ p_{1'ab}p_{1'2a}p_{12b'} x_{(1'3a,12b)}x_{(13a,1'2b')}.   \label{m.t vr pl} 
 \end{eqnarray}
 Here, $\vi(f)=  p_{1'ab}p_{1'2a}p_{1'2b'}p_{1'3a}p_{12b'} F$, where
$F= p_{123} p_{1ab}- p_{12a}p_{13b} + p_{13a}p_{12b}.$

Fix $a,b,1', 2', b' \in [n] \- \{1,2,3\}$, all being distinct:
\begin{eqnarray} 
f: (p_{1ab}p_{1'2a}p_{1'2'b'} x_{(1'ab, 123)}- p_{12a}p_{1'ab}p_{1'2'b'}x_{(1'2a,13b)} )
x_{(1'3a,12'b')}  \\ \nonumber 
+ p_{1'ab}p_{1'2a}p_{12'b'} x_{(1'3a,12b)}x_{(13a,1'2'b')}.   \label{m.t vr pl-2} 
 \end{eqnarray}
 Here, $\vi(f)=  p_{1'ab}p_{1'2a}p_{1'2'b'}p_{1'3a}p_{12'b'} F$, where
$F= p_{123} p_{1ab}- p_{12a}p_{13b} + p_{13a}p_{12b}.$

Fix $a,b,c, a',b', c' \in [n] \- \{1,2,3\}$, all being distinct: 
\begin{equation}
\label{m.t vr pl-3} 
f: 
p_{abc} x_{(a'bc, 123)}x_{(12a',3b'c')}p_{13a'}p_{2b'c'}p_{23a'}
 -p_{12a} x_{(12a', 3bc)}x_{(13a',2b'c')}p_{23a'}p_{a'bc}p_{3b'c'}  
\end{equation} 
$$+
p_{13a} x_{(13a', 2bc)}x_{(12a',3b'c')}p_{23a'}p_{a'bc}p_{2b'c'}  - p_{23a} x_{(23a', 1bc)}x_{(12a',3b'c')}p_{13a'}p_{a'bc}p_{2b'c'} .$$  
 Here, $\vi(f)=  p_{12a'}p_{13a'}p_{23a'}p_{2b'c'}p_{3b'c'}p_{a'bc}F$, where
$F= p_{123} p_{abc}- p_{12a}p_{3bc} + p_{13a}p_{2bc}-p_{23a}p_{1bc}.$

The above relation is also equal to the following one, modulo $\ker^\mh\vi$:
\begin{equation}\label{m.t vr pl-3'} 
f: p_{abc} x_{(a'bc, 123)}x_{(12a',3b'c')}p_{13a'}p_{a'bc}p_{23a'} -
p_{12a} x_{(12a', 3bc)}x_{(13a',2b'c')}p_{23a'}p_{a'bc}p_{3b'c'}
\end{equation} 
$$+ p_{13a} x_{(13a', 2bc)}x_{(12a',3b'c')}p_{23a'}p_{a'bc}p_{2b'c'}
 - p_{23a} x_{(23a', 1bc)}x_{(13a',2b'c')}p_{12a'}p_{a'bc}p_{3b'c'} . $$ 
Here, $\vi(f)=  p_{12a'}p_{13a'}p_{23a'}p_{2b'c'}p_{3b'c'}p_{a'bc}F$, where
$F= p_{123} p_{abc}- p_{12a}p_{3bc} + p_{13a}p_{2bc}-p_{23a}p_{1bc}.$
\end{exam}

None of the polynomials in the above examples belong to $R_\vr$. Further,
one  inspects directly that  {\it none of  them admit any parent polynomials} other than themselves.

The following is an example of $\hpl$ relation in $R_\vr$.

\begin{exam}\label{exam-vr pl}
Fix $a,b,c, a',b', c' \in [n] \- \{1,2,3\}$, all being distinct. Let
\begin{equation}\label{the-vr}
f= 
x_{(abc,23a')} x_{(13a,2b'c')}x_{(a'bc, 123)}x_{(12a',3b'c')}x_{(23a,13a')}
 \end{equation} 
 $$-x_{(23a,a'bc)}x_{(12a,3b'c')} x_{(12a', 3bc)}x_{(13a',2b'c')}x_{(13a,23a')}  $$
$$+x_{(23a,a'bc)}x_{(13a,2b'c')} x_{(13a', 2bc)}x_{(12a',3b'c')}x_{(13a,23a')} $$ 
$$-x_{(23a,a'bc)}x_{(13a,2b'c')} x_{(23a', 1bc)}x_{(12a',3b'c')}x_{(23a,13a')} $$ 
Then, $\vi(f)=  p_{12a'}p_{13a'}p_{23a'}p_{2b'c'}p_{3b'c'}p_{a'bc}p_{13a}p_{23a}F$, where
$F= p_{123} p_{abc}- p_{12a}p_{3bc} + p_{13a}p_{2bc}-p_{23a}p_{1bc}.$
\end{exam}



We write $F=\sum_{t \in S_F} \sgn (t) p_{\uu_t}p_{\uv_t}$.
We follow the notations of \eqref{initial f} and \eqref{initial f 2}.

The following special $\hpl$ relations can be reduced to  multiples of $\pl$
relations, modulo $\ker^\mh \vi$.

\begin{lemma} \label{f-to-F} 
 Suppose $\vi (f)=F$ and
$$\hbox{ $p_{\uu_s}p_{\uv_s} \mid f_s$ for some $s \in S_F$.}$$
 Then, 
$$\hbox{ $f \equiv \bn F, \mod (\ker^\mh \vi)$, for some $\bn \in R$.}$$
\end{lemma}
\begin{proof}
We prove the  statement by applying induction on $\deg_\vr (f)$.

When $\deg_\vr (f)=0$, then by a direct inspection, we must have
$f=n F$ for some $n \in R_0$. Hence, the statement holds.

Assume $\deg_\vr (f)=e>0$ and the statement holds for $\deg_\vr (f)=e-1$.

Then, we  can write
 $$
f= X_s T_s' \sgn(s) p_{\uu_s}p_{\uv_s} 
+ \sum_{t \in S_F \- s} \sgn(t) X_t T_t $$
 for some $T_s' \in R$ and $T_t \in R$ for all $t \in S_F \-s$ such that
 $\{X_t \mid t \in S_F\}$ is a subset of the homogenous coordinates  of $\PP_\buw$
 for some $\buw \in \Ladnsort$.
 
 Consider 
 $$\barf= \vi(X_s) T_s'  \sgn(s) p_{\uu_s}p_{\uv_s} + \sum_{t \in S_F \- s} \sgn(t) \vi(X_t) T_t.$$
 By the inductive assumption, we have 
 $$\hbox{ $\barf \equiv \bar\bn F, \mod (\ker^\mh \vi)$, for some $\bar\bn \in R$.}$$
 This implies that 
 $$ \vi(X_s) T_s'  \equiv \bar\bn, \; \vi(X_t) T_t \equiv \bar\bn p_{\uu_t}p_{\uv_t},
 \mod (\ker^\mh \vi) , \; \forall \; t \in S_F \- s.$$
 Substituting $\bar\bn$ by $\vi(X_s) T_s' $ in all the second equalities above, we have
  $$\vi(X_t) T_t \equiv \vi(X_s) T_s' p_{\uu_t}p_{\uv_t}, \mod (\ker^\mh \vi) , \; \forall \; t \in S_F \- s.$$
  Consequently,
   $$X_t T_t \equiv X_s T_s' p_{\uu_t}p_{\uv_t}, \mod (\ker^\mh \vi), \; \forall \; t \in S_F \- s.$$
 Hence, we obtain
 $$f 
  \equiv X_s T_s' \sgn (s) p_{\uu_s}p_{\uv_s} + \sum_{t \in S_F \- s} X_s T_s' \sgn (t) p_{\uu_t}p_{\uv_t}
  = X_s T_s' F, \mod (\ker^\mh \vi)$$
  Thus, the lemma follows by induction.
\end{proof}

As $F=\sum_{t \in S_F} \sgn (t) p_{\uu_t}p_{\uv_t}$, we have
$L_F=\sum_{t \in S_F} \sgn (t) x_{(\uu_t,\uv_t)}$.

The following special $\hpl$ relations can be reduced to  multiples of 
linearized $\pl$
relations, modulo $\ker^\mh \vi$.

\begin{lemma} \label{f-to-LF} 
 Suppose $\vi (f)=F$ and
$$\hbox{ $x_{(\uu_s, \uv_s)} \mid f_s$ for some $s \in S_F$.}$$
 Then, $$\hbox{ $f \equiv \bn L_F, \mod (\ker^\mh \vi)$, for some $\bn \in R$.}$$
\end{lemma}
\begin{proof}
By assumption, we  can write
 $$f=    \sgn(s) x_{(\uu_s, \uv_s)} T_s+ \sum_{t \in S_F \- s} \sgn(t) X_t T_t $$
 for some $T_t \in R$ for all $t \in S_F$ such that
 $\{x_{(\uu_s, \uv_s)} , X_t \mid t \in S_F \- s\}$ 
 is a subset of the homogenous coordinates  of $\PP_\buw$
 for some $\buw \in \Ladnsort$.
 
 Consider 
 $$\barf= \sgn(s)  p_{\uu_s}p_{\uv_s}  T_s+ \sum_{t \in S_F \- s} \sgn(t) \vi(X_t) T_t. $$
 By Lemma \ref{f-to-F}, we have 
 $$\barf \equiv \bn F, \mod (\ker^\mh \vi), \; \hbox{for some $\bn \in R$}.$$
  This implies that 
 $$ T_s \equiv \bn, \; \vi(X_t) T_t \equiv p_{\uu_t}p_{\uv_t} \bn,
 \mod (\ker^\mh \vi) , \; \forall \; t \in S_F \- s.$$
 Observe that $X_t$ and $x_{(\uu_t, \uv_t)}$ are both homogenous coordinate of $\PP_\buw$.
 Consequently, 
 $$X_t T_t \equiv  x_{(\uu_t, \uv_t)} \bn,
 \mod (\ker^\mh \vi) , \; \forall \; t \in S_F \- s.$$
 Therefore, 
  $$f  =    \sgn(s) x_{(\uu_s, \uv_s)} T_s+ \sum_{t \in S_F \- s} \sgn(t) X_t T_t $$
 $$ \equiv \sgn(s) x_{(\uu_s, \uv_s)} \bn+ \sum_{t \in S_F \- s} \sgn(t) x_{(\uu_t, \uv_t)} \bn
= \bn L_F,  \mod (\ker^\mh \vi).$$
\end{proof}

\begin{prop}\label{h-pl} Let $f \in  R$ be any $\hpl$ relation
such that $\vi(f)=hF$ for some $h$.
 Then, following the notation above, we have
\begin{enumerate}
\item $ x_{(\uu_t,\uv_t)} f_s - x_{(\uu_s,\uv_s)} f_t \in \ker^\mh \vi_\fT. $
\item $x_{(\uu_s,\uv_s)} f \equiv f_s L_F$,
modulo $\ker^\mh \vi_\fT$, for all $s, t \in S_F$.
\end{enumerate}
\end{prop}
\begin{proof}
From \eqref{initial f 2}, one checks directly that 
$x_{(\uu_t,\uv_t)} f_s - x_{(\uu_s,\uv_s)} f_t \in \ker^\mh \vi$,
 for all $s, t \in S_F$.

Next, we fix any $s \in S_F$.
By  (1),  for any $t \in S_F$, we have
$$ \hbox{$x_{(\uu_t,\uv_t)} f_s - x_{(\uu_s,\uv_s)} f_t \in 
\ker^\mh \vi_\fT$}. $$
Thus, using $f=\sum_{t \in S_F} \sgn (t) f_t$
and $L_F=\sum_{t \in S_F} \sgn (t) x_{(\uu_t,\uv_t)}$,
 we obtain
$$x_{(\uu_s,\uv_s)} f = \sum_{t \in S_F} \sgn (t) x_{(\uu_s,\uv_s)} f_t
\equiv \sum_{t \in S_F} \sgn (t) x_{(\uu_t,\uv_t)} f_s = f_s L_F, 
\; \mod \ker^\mh \vi_\fT .$$
\end{proof}

\begin{rem}\label{PP is covered}
 In the above, we let $\buw \in \Ladnsort$ be such that 
$(\uu_s,\uv_s) \in \La_\buw$ for all $s \in S_F$.
Suppose the projective space $\PP_\buw$ is covered by
the charts $$\{(x_{(\uu_s,\uv_s)}\ne 0) \mid s \in S_F\}.$$
Then by Proposition \ref{h-pl} (2),  it follows that
any $\hpl$  relation, such as $f$ in that proposition, 
can be reduced to $L_F$.  The above covering condition
does hold for $\Gr^{2,E}$ (see \S \ref{2,E}, especially \eqref{Pw=PF}). However, it does not hold in general.
In \cite{Hu2025}, there is a parallel construction of $\PP_F$
using precisely $$\{x_{(\uu_s,\uv_s)}, s \in S_F\}$$
as homogenous coordinates, hence in that paper, similar $\hpl$  relations can all reduced to linearized $\pl$ relations.

One may wonder whether in the present paper we could use $\prod_{F  \in \sF} \PP_F$ to take palce of
$\prod_{\buw  \in \Ladnsort} \PP_\buw$. Yes, we can, see \S \ref{F-blowups}.
Unfortunately, the so-obtained natural map from
$\UU^{d,E}/\TT$ into $\prod_{F  \in \sF} \PP_F$ is in genreal only a contraction but not an embedding. 
\end{rem}

\subsubsection{Proof of Theorem \ref{eqs-F-intro}.}

\begin{proof} By the second part of Lemma \ref{defined-by-ker} 
(the case when $\fT=\Ladnsort$),
 the scheme $\FF^{d, E}_*$, as a closed subscheme of 
 $\PP(\wedge^d E) \times  \prod_{\buw \in \La^\sort_{d,[n]}} \PP_\buw$, 
 is defined by  $\ker^\mh \vi_{\Gr}$. Then, 
by Lemma \ref{ker-hF-generate}
 it is defined by  $\ker^\mh \vi$, $\sF$, $ L_\sF$, and $\sF^{\vp\vr}$.  
 Then, by applying 
the same proof as in that of Proposition \ref{eqs-F'-intro},
we can reduce $\ker^\mh \vi$ to $\cB^\wp \sqcup  \cB^\frb$.

 Thus, it remains to reduce $\sF$ to $L_\sF$.

 Fix any standard chart $\fV$ of  
 $\PP(\wedge^d E) \times  \prod_{\buw \in \La^\sort_{d,[n]}} \PP_\buw$ as in Definition \ref{fv-sR}. 
We claim that for any relation $F$ of $\sF$,
 $F|_\fV$ is a multiple of $(L_F)|_\fV$,
modulo $\langle  \cB^\wp \rangle$. 

Write $F =\sum_{s \in S_F} \sign (s) p_{\uu_s}p_{\uv_s}$.
We may assume that $\sort (\uu_s \vee \uv_s)=\buw$ for one hence for all $s \in S_F$
for some $\buw \in \Ladnsort$.

Considering the following relations of $\cB^\wp$
 $$p_{\uu_s} p_{\uv_s} x_{(\ua_\buw, \ub_\buw)}
 - p_{\ua_\buw}p_{\ub_\buw} x_{(\uu_s, \uv_s)}$$ for all $s \in S_F$.
By multiplying $\sign (s)$ to  the above relation and adding together all the resulted binomials, 
we obtain, 
\begin{equation}\label{F=pL} 
x_{(\ua_\buw, \ub_\buw)} F \equiv p_{\ua_\buw}p_{\ub_\buw} L_F, \; \mod \langle \cB^\wp \rangle. \end{equation}
Because on the chart $\fV$,  $x_{(\ua_\buw, \ub_\buw)}$ 
is invertible (see Definition \ref{fv-sR}), the claim follows.

This proves the theorem.  \end{proof}

Then, modulo the relations of  $L_\sF$, we can state the following.
\begin{cor} 
 Assume that the base field has characteristic zero.
Then,  $\FF^{d, E}_*$ is embedded in 
$\PP(\wedge^d E) \times \prod_{\buw \in \La^\sort_{d,[n]}}  \LL_\buw$,
 a product of projective spaces, as the closed  subscheme defined by
binomial relations in $\cB^\wp$, $\cB^\frb$, 
and $\sF^{\vp\vr}$.
\end{cor}
\begin{proof}
Note that 
$\prod_{\buw \in \La^\sort_{d,[n]}} \LL_\buw$  as a closed subscheme of 
$\prod_{\buw \in \La^\sort_{d,[n]}} \PP_\buw$ is precisely defined by
all the linearized $\pl$ relations of $L_\sF$. Hence, the statement of
the corollary follows from Theorem \ref{eqs-F-intro}.
\end{proof} 

\begin{cor}\label{F(2E)}
 Assume that the base field has characteristic zero.
Then,  $\FF^{2, E}_*$ is embedded in 
$\PP(\wedge^2 E) \times \prod_{\buw \in \La^\sort_{2,[n]}}  \PP_\buw$
 as the closed  subscheme defined by
binomial relations in $\cB^\wp$, $\cB^\frb$, 
and linearized $\pl$ relations in $L_\sF$.
\end{cor}
\begin{proof} 
By Theorem \ref{eqs-F-intro}, it suffices to reduce
$\sF^{\vp\vr}$ to $L_\sF$.

Using \S \ref{2,E}, \eqref{Pw=PF},
we see that the projective space $\PP_\buw$ is covered by
the charts $$\{(x_{(\uu_s,\uv_s)}\ne 0) \mid s \in S_F\}$$
for  any $\buw \in \La_{2,n}^{\rm sort}$,
where
$F$ is the $\pl$ relation of \eqref{F=buw},
uniquely determined by $\buw \in \La_{2,n}^{\rm sort}$.

Then by Proposition \ref{h-pl} (2) (and using the notation therein),  
for any $\hpl$  relation $f$,
we have
$$x_{(\uu_s,\uv_s)} f = f_s L_F, 
\; \mod \ker^\mh \vi .$$
It follows that $f$ can be reduced to $L_F$. 
\end{proof}

\subsection{Equations defining 
 $\HH^{d,E}$ and  proofs of Proposition \ref{eqs-H'-intro}
Theorem \ref{eqs-H-intro} }
\label{sec:eqForY}

Consider the  homomorphism of algebras
\begin{equation} \label{vivr}
\varphi_\vr: R_\vr \lra R_0, \;\;
x_{(\uu,\uv)} \to p_\uu p_\uv
\end{equation}
where recall that
 $ R_\vr= \kk[(x_{(\uu,\uv)})]_{(\uu,\uv) \in \Lambda_\buw, \buw \in \La^\sort_{d,[n]}}$
and  $R_0=\kk[p_\uu]_{\uu \in \II_{d,[n]}}$.
It induces the  homomorphism
\begin{equation} \label{vivr-Gr}
\varphi_{\Gr,\vr}: R_\vr \lra R_0 \lra R_0/I_\whwp, \;\;
x_{(\uu,\uv)} \to p_\uu p_\uv + I_\whwp.
\end{equation}
We  aim to describe 
the subset $\ker^{\rm mh} \vi_\vr$  of $\ker \vi_\vr$, respectively,
the subset $\ker^\mh \vi_{\Gr,\vr}$ of $\ker \vi_{\Gr,\vr}$, that consists of multi-homogenous polynomials.

\begin{defn}\label{defn:cBvr}
We let $\cB^\vr = \cB^\frb \cap R_\vr$.
A binomial of $\cB^\vr$ is called a $\vr$-binomial.
\end{defn}

\begin{lemma}\label{ker-vivr}
$\ker^\mh \vi_\vr$ is generated by $\cB^\vr$
\end{lemma}
\begin{proof}
Let $\bm-\bm' $ be any binomial in $\ker^\mh \vi_\vr$. We 
can decopose it as
$$\bm -\bm' =
\prod_{i=1}^\ell \bm_i -  \prod_{i=1}^\ell \bm_i',$$ 
such that $\bm_i -\bm'_i \in \ker^\mh \vi_\vr$ and
are $\fb$-irreducible root binomials,   for all $i \in [\ell]$.
Then $\bm_i -\bm'_i$ are automatically $\wp$-irreducible,
hence $\frb$-irreducible, therefore belong to $\cB^\vr$.
Then by the same line of argument as in the proof of 
Lemma \ref{reduce to rb}, $\bm -\bm'$ is generated by
$\bm_i -\bm'_i \in \cB^\vr, i \in [\ell]$. 

This establishes the lemma.
\end{proof}

\begin{defn}\label{vr pl complex}
Let $f \in  R$ be any $\hpl$.
If $f$ belongs to $R_\vr$, then we say $f$ is a $\vrpl$ relation.
We denote the set of all $\vrpl$ relations by $\sF^\vr$.
In other words, $\sF^\vr=\sF^{\vp\vr} \cap R_\vr$.
\end{defn}

The polynomial in Example \ref{exam-vr pl} is a $\vrpl$ relation.


Analogous to Lemma \ref{ker-hF-generate}, applying the similar discussions 
as in the three paragraphs immediately before Definition \ref{defn:h-pl}, 
using  $\vi_{\Gr,\vr}$ in place of  $ \vi_{\Gr}$, 
we obtain

\begin{lemma}\label{ker-vrF-generate} 
Assume that the base field has characteristic zero.
$ \ker^\mh \vi_{\Gr,\vr}=\langle \ker^\mh \vi_\vr, \sF^\vr  \rangle$.
\end{lemma}

\begin{lemma}\label{ker-vivr-Gr} 
 Assume that the base field has characteristic zero. Then,
$\ker^{\rm mh} \vi_{\Gr,\vr}$ is generated 
by $\cB^\vr$ and  $\sF^\vr$. 
\end{lemma}
\begin{proof} 
 This follows from by applying 
 Lemma \ref{ker-vrF-generate} and then Lemma \ref{ker-vivr}.
 \end{proof}

\subsubsection{Proof of Proposition \ref{eqs-H'-intro}.}

\begin{proof}
 Consider the commutative diagram 
  \begin{equation}\label{d-schemes}
 \xymatrix{
 \FF^\wdep_*\ar[d] \ar @{^{(}->}[r]  & 
  \PP(\wedge^d E) \times  \prod_{\buw \in  \Ladnsort} \PP_{\buw} \ar[d]^{\rm proj} \\
 \HH^\wdep  \ar @{^{(}->}[r]  &   \prod_{\buw \in \Ladnsort} \PP_{\buw}  .
 }
 \end{equation}
Correspondingly, we have  the diagram of algebras and homomorphisms
\begin{equation}\label{d-algebras}
 \xymatrix{
 R_\vr \ar[d] \ar[r]^{\vi_\vr} &  R_0 \ar[d]^{=}  \\
 R  \ar[r]^{\vi} &  R_0,
 }
 \end{equation}
 where the left vertical arrow is the inclusion.

Note here that by the first statement of
Lemma \ref{defined-by-ker} (the case when $\fT= \Ladnsort$),
the scheme $\FF^\wdep_*$, as a closed subscheme of  
 $\PP(\wedge^d E) \times  \prod_{\buw \in  \Ladnsort} \PP_{\buw} $,
 is defined by $\ker^\mh \vi$.
Since the image of  $\FF^\wdep$ under the projection
(the right vertical arrow) of \eqref{d-schemes}
is precisely  $\HH^\wdep$, one sees that
the multi-homogeneous kernel,  $\ker^\mh \vi_{\vr}$,   
computes the ideal of 
$\HH^\wdep$ in $\prod_{\buw \in  \Ladnsort} \PP_{\buw}$.
Therefore, by  Lemma \ref{ker-vivr},
 $\HH^\wdep$, as a closed subscheme of 
$\prod_{\buw \in \La^\sort_{d,[n]}} \PP_\buw$, is defined by $\cB^\vr$.
 \end{proof}


\subsubsection{Proof of Theorem \ref{eqs-H-intro}.}

\begin{proof} 
 Consider the commutative diagram of schemes and morphisms
  \begin{equation}\label{d-schemes2}
 \xymatrix{
 \FF^{d, E}_*\ar[d] \ar @{^{(}->}[r]  & 
  \Gr^{d, E} \times  \prod_{\buw \in  \Ladnsort} \PP_{\buw} \ar[d]^{\rm proj} 
  \subset \PP\wdep\times  \prod_{\buw \in  \Ladnsort} \PP_{\buw} \ar[d]^{\rm proj}  \\
 \HH^{d, E} \ar @{^{(}->}[r]  &   \prod_{\buw \in  \Ladnsort} \PP_{\buw}  .
 }
 \end{equation}
Correspondingly, we have  the diagram of algebras and homomorphisms
\begin{equation}\label{d-algebras2}
 \xymatrix{
 R_\vr \ar[d] \ar[r]^{\vi_{\Gr,\vr }\;\;\;\;\;} &  R_0/I_\wp \ar[d]^{=}  \\
 R  \ar[r]^{\vi_\Gr \;\;\;\;\;} &  R_0/I_\wp,
 }
 \end{equation}
 where the left vertical arrow is the inclusion.

Note here that by the second statement of
Lemma \ref{defined-by-ker} (the case when $\fT= \Ladnsort$),
the scheme $\FF^{d, E}_*$, as a closed subscheme of  
 $\PP(\wedge^d E) \times  \prod_{\buw \in  \Ladnsort} \PP_{\buw} $,
 is defined by $\ker^\mh \vi_\Gr$.
Since the image of  $\FF^{d, E}$ under the projection 
(the right vertical arrow) of \eqref{d-schemes2}
is precisely  $\HH^{d, E}$, one sees that
the multi-homogeneous kernel, $\ker^\mh \vi_{\Gr,\vr}$,  
computes the ideal of 
$\HH^{d, E}$ in  $\prod_{\buw \in  \Ladnsort} \PP_{\buw}$.
Hence, according to Lemma \ref{ker-vivr-Gr},
 $\HH^{d, E}$, as a closed subscheme of 
$\prod_{\buw \in \La^\sort_{d,[n]}} \PP_\buw$, is defined by
 $\cB^\vr$, $L_\sF$,  and $\sF^\vr$.

This proves the theorem. 
\end{proof}

 Modulo the relations of  $L_\sF$, we can state the following.
\begin{cor}\label{eqForY(d,n)}
Assume that the base field has characteristic zero.
Then, the $\fG$-quotient $\HH^{d, E}$ is embedded in
$\prod_{\buw \in \La^\sort_{d,[n]}} \LL_\buw$, a product of linear projective subspaces, 
as the closed subscheme defined by the binomials of $\cB^\vr$ and
$\sF^\vr$.
 \end{cor}
\begin{proof}
Note that 
$\prod_{\buw \in \La^\sort_{d,[n]}} \LL_\buw$  as a closed subscheme of 
$\prod_{\buw \in \La^\sort_{d,[n]}} \PP_\buw$ is precisely defined by
all the linearized $\pl$ relations of $L_\sF$. Hence, the statement of
the corollary follows from Theorem \ref{eqs-H-intro}.
\end{proof}

\begin{cor}\label{H(2E)}
 Assume that the base field has characteristic zero.
Then,  $\HH^{2, E}_*$ is embedded in 
$\prod_{\buw \in \La^\sort_{2,[n]}}  \PP_\buw$
 as the closed  subscheme defined by
binomial relations in $\cB^\vr$
and linearized $\pl$ relations in $L_\sF$.
\end{cor}
\begin{proof} 
By Theorem \ref{eqs-H-intro}, it suffices to reduce
$\sF^\vr$ to $L_\sF$. The remaining proof is parallel to
that of Corollary \ref{F(2E)}

By Proposition \ref{h-pl} (2) (and using the notation therein),  
for any $\vr$  relation $f$, we have
$$x_{(\uu_s,\uv_s)} f = f_s L_F, 
\; \mod \ker^\mh \vi .$$
Using \S \ref{2,E}, \eqref{Pw=PF},
we see that the projective space $\PP_\buw$,
where $\buw$ and the $\pl$ relation $F$ determine each other
according to \eqref{F=buw},
 is covered by
the charts $$\{(x_{(\uu_s,\uv_s)}\ne 0) \mid s \in S_F\}$$
It follows that $f$ can be reduced to $L_F$. 
\end{proof}

\subsection{Morphisms among $\HH^{d,E}$
and proofs of Propositions \ref{pr-I-intro} and \ref{d2d-intro}}\label{sect:d-E}

The results of this section may be extended to include all  $\fG$-quotients $\HH^{d,E_\bcd}$.
To simplify presentation, we only focus on the $\fG$-quotients $\HH^{d,E}$ throughout this section.

For any subset $H \subset [n]$,
we let  $E_H=\bigoplus_{i \in H} E_i.$
Further, we define
$${\TT}_H=(\GG_m)^H/\GG_m=\{ (t_i)_{i \in H} \cdot \GG_m \mid t_i \in \GG_m, \;\hbox{for all}\;\; i \in H \}.$$
We have the obvious projection
 $${\TT} \lra {\TT}_H$$
 which is a group homomorphism.

Fix any $J \subset [n]$ with $|J|>d$.  
We let  
$$\II_{d,J}=\{(u_1<\cdots< u_d) \in \II_{d,[n]} \mid u_a \in J, \;\hbox{for all}\; a \in [d]\}
\subset \II_{d,[n]}.$$
Likewise, we define
$$\La_{d,J}^\sort=\{\buw=(\uwo,\uwe) \in  \Ladnsort \mid \uwo, \uwe \in \II_{d,J}\} \subset  \Ladnsort.$$

For any point $(K \hookrightarrow E)$ of $\UU^{d, E}$, we have the diagram
$$ \xymatrix{
 K \ar[r]^{\subset \;}  \ar[d]_{\rm projection} 
 & E \ar @{>}[d]^{\rm projection} \\
K^\vee_J \;\; \ar[r]^{\subset \;} & \; E_J ,
}
$$
where $K^\vee_J$ is the image of $K$ under the projection $E \to E_J$.
This gives rise to the projection 
$$\UU^{d, E} \lra \UU^{d, E_J},\;\;\;\;  [p_\uu]_{\uu \in \II_{d,[n]}} \to [p_\uu]_{\uu \in \II_{d,J}} $$
 by forgetting all $p_\uu$ with $\uu \notin \II_{d,I}$.
 This projection is obviously equivariant under the group homomorphism
 ${\TT} \lra {\TT}_J.$ Hence, we obtain
 \begin{equation}\label{proj-I}
 \UU^{d, E} /{\TT} \lra \UU^{d, E_J}/{\TT}_J.
 \end{equation}

 Likewise, suppose $|I| =r < d$. 
We write $I=\{i_1, \cdots, i_r\}$.
$$\II_{(d, [n]); {\supset  I}}=\{(u_1<\cdots< u_d) \in \II_{d,[n]} \mid  \{ u_1, \cdots, u_d\} \supset I \}
\subset \II_{d,[n]}.$$
$$\La_{(d, [n]); {\supset  I}}^\sort=\{\buw=(\uwo,\uwe) \in  \Ladnsort \mid \uwo, \uwe \in \II_{(d, [n]); {\supset  I}}\} \subset  \Ladnsort.$$
Observe here that we have a natural bijection
\begin{equation}\label{bj} 
\La_{(d, [n]); {\supset  I}}^\sort \longleftrightarrow \La_{d-r, [n-r]}^\sort.
\end{equation}
To see the bijection, first, we identify the naturally 
ordered set $[n-r]$ with the induced totally ordered set
$[n]\- I$.
Then, the right arrow of the bijection is given by
$$(\uwo,\uwe) \La_{(d, [n]); {\supset  I}}^\sort \to (\uwo \- I ,\uwe \- I) \in \La_{d-r, [n-r]}^\sort;$$
 Using the identification, $[n-r]=[n]\- I$, the left arrow of the bijection is given by
$$(\uw_{\rm o}',\uw_{\rm e}')  \in \La_{d-r, [n-r]}^\sort  
\to (\sort(\uw_{\rm o}' \cup I) , \sort(\uw_{\rm e}' \cup I)) \in  \La_{(d, [n]); {\supset  I}}^\sort .$$
In particular, for any $(\uwo,\uwe) \La_{(d, [n]); {\supset  I}}^\sort $, we have
a natural identification
\begin{equation}\label{iden}
\PP_{(\uwo,\uwe)} = \PP_{(\uwo \- I ,\uwe \- I)}.
\end{equation}

Now,
for any point $(K \hookrightarrow E)$ of $\UU^{d, E}$, we have the diagram
$$ \xymatrix{  K \ar[r]^{\subset \;}  \ar[d]_{\rm intersection} 
 & E \ar @{>}[d]^{\rm intersection} \\
K_{[n]\-I}=K\cap E_{[n]\- I} \;\; \ar[r]^{\subset \;} & \; E_{[n]\-I} =E\cap E_{[n]\-I}. } $$
This gives rise to the surjective morphism 
$$\UU^{d, E} \lra \UU^{d-r, E_{[n] \-I}}.$$ 
 This morphism is  equivariant under the  group homomorphism
 ${\TT} \lra {\TT}_{[n]\-I}.$ Hence, we obtain
  \begin{equation}\label{int-n-I}\UU^{d, E} /{\TT} \lra \UU^{d-r, E_{[n]\-I}}/{\TT}_{n\-[I]}.
  \end{equation} 


\subsubsection{ Proof of Proposition \ref{pr-I-intro}.}

\begin{proof}
Apply \eqref{proj-I} and consider the commutative diagram
\begin{equation}\label{7.0}
\xymatrix{
\UU^{d, E} /{\TT}  \ar[r]^{\subset \;\;\;\;\;\;\;}  \ar[d] & \prod_{\buw \in \La_{d,[n]}^\sort }\PP_\buw 
\ar @{>}[d]^{\rm projection} \\
\UU^{d,E_J} /{\TT}_J\ar[r]^{\subset \;\;\;\;\;\;\;} &  \prod_{\buw \in \La_{d,J}^\sort} \PP_\buw .
}
\end{equation}
By taking the closure of $\UU^{d, E}/{\TT}$ in 
$\prod_{\buw \in \La_{d,[n]}^\sort }\PP_\buw $ and the closure of $\UU^{d,E_J} /{\TT}_J$ in 
$\prod_{\buw \in \La_{d,J}^\sort }\PP_\buw $, we have
$\HH^{d, E} \lra \prod_{\buw \in \La_{d,J}^\sort} \PP_\buw$ whose image equals to $\HH^{d, E_J}.$
Hence, we obtain 
\begin{equation}\label{pi-I}
{\rm proj}_J: \HH^{d, E} \lra \HH^{d, E_J}.
\end{equation}
This proves the right arrow of the display in Proposition \ref{pr-I-intro}.

We can apply \eqref{int-n-I} and \eqref{iden} to obtain the commutative diagram
\begin{equation}\label{7.00}
\xymatrix{
\UU^{d, E} /{\TT}  \ar[r]  \ar[d] & \prod_{\buw \in \La_{d,[n]}^\sort }\PP_\buw 
\ar @{>}[d]^{\rm projection} \\
\UU^{d-r,E_{[n]\- I}} /{\TT}_{[n]\- I}  \ar[r] & 
\prod_{\buw \in \La_{(d,[n]); \supset I}^\sort} \PP_\buw = \prod_{\buw' \in \La_{d-r,[n]\- I}^\sort} \PP_{\buw'} .
}
\end{equation}
By taking the closure of $\UU^{d, E}/{\TT}$ in 
$\prod_{\buw \in \La_{d,[n]}^\sort }\PP_\buw $ and the closure of $\UU^{d-r,E_{[n] \-[I]}} /{\TT}_{[n]\-[I]}$ in 
$\prod_{\buw \in \La_{(d,[n]); \supset I}^\sort} \PP_\buw = \prod_{\buw' \in \La_{d-r,[n]\- I}^\sort} \PP_{\buw'}$, we obtain
$$\HH^{d, E} \lra \prod_{\buw' \in \La_{d-r,[n] \- I}^\sort} \PP_{\buw'} $$ such that its image 
is equal to $\HH^{d-r, E_{[n] \- I}}.$
This gives
\begin{equation}\label{int-I}
{\rm int}_{[n]\- I}: \HH^{d, E} \lra \HH^{d-r, E_{[n]\- I}}.
\end{equation}
This proves the left arrow of the display in Proposition \ref{pr-I-intro}. 
\end{proof}



As an interesting special case of \eqref{pi-I}, we look at the following.

Consider  $\Gr^{d, E}$. 
When $ 2d> n=\dim E$, we can replace $\Gr^{d, E}$ by $\Gr^{E,n-d}$ where
$$\Gr^{E, n-d}=\{ \hbox{quotient subspaces} \;  E \twoheadrightarrow  G \mid \dim G=n-d\}.$$
Hence, we may assume that $2d \le n$. 

\begin{defn} An element $\buw = \uwoe$ of $\La^\sort_{d,[n]}$ is called distinguished if
$$w_1 < w_2 < w_3< w_4 <\cdots < w_{2d-1}< w_{2d},$$
where $\uwo=(w_1,w_3,\cdots, w_{2d-1})$ and $\uwe=(w_2,w_4, \cdots, w_{2d})$,
that is, all the elements in the sequence are pair-wise distinct.
We  let $\rladnsort$ denote the set of all distinguished elements of $\La^\sort_{d,[n]}$.
\end{defn}
One counts and finds $|\rladnsort|={n \choose 2d}$.

In particular, when $\dim E=2d$, we have that $\rladnsort$ consists of a single element:
$$\buw_{(d,2d)}:=\uwoe=((1,3,\cdots, 2d-1),(2, 4, \cdots, 2d)).$$

For any $\buw$ of $\La^\sort_{d,[n]}$, its support, denoted supp($\buw$), is the {\it set}
of all the integers occurring in $\{w_1, \cdots, w_{2d}\}$ (without repetition, i.e.,
 with redundant elements 
being removed from the set $\{w_1, \cdots, w_{2d}\}$). Clearly, given any $\buw$ of $\La^\sort_{d,[n]}$,
there exists  an distinguished $\wtbuw \in \rladnsort$ such that 
${\rm supp}(\buw) \subset {\rm supp}(\wtbuw)$.

Now,  consider and fix an arbitrary (distinguished) element
$\wtbuw = (\tilde\uw_{\rm o},\tilde\uw_{\rm e}) \in \rladnsort$ 
with $\tilde\uw_{\rm o}=(w_1,w_3,\cdots, w_{2d-1})$
and $\tilde\uw_{\rm e}=(w_2,w_4, \cdots, w_{2d})$.
We let $$I_\wtbuw=\supp\; (\wtbuw)=\{w_1, w_2, w_3, w_4, \cdots, w_{2d-1}, w_{2d}\} \subset [n].$$
Then, we have
\begin{equation}\label{iff}
{\rm supp}(\buw) \subset {\rm supp}(\wtbuw) \; \iff \; \buw \in \La_{d,I_\wtbuw}.
\end{equation}

We abbreviate $E_{I_\wtbuw}$ as $E_\wtbuw$. Then, by applying \eqref{pi-I} to $I_\wtbuw$, we obtain
$$\pi_{\wtbuw}: \HH^{d, E_\bcd} \lra \HH^{d, E_\wtbuw}$$
where $\pi_{\wtbuw}=\pi_{I_\wtbuw}.$

The $\fG$-quotient, $\HH^{d, E}$ with $\dim E=2d$,  plays some special role
amongst all $\fG$-quotients.
In the sequel, we  write 
$\HH^{d,2d}$ for $\HH^{d, E}$ when $E=\kk^{2d}$. 

\subsubsection{Proof of Proposition \ref{d2d-intro}.}

\begin{proof} 
Fix any $\wtbuw \in \rladnsort$. 

First, observe that we have
 $$\HH^{d,E_\wtbuw} \cong \HH^{d,2d}.$$

Next, we have the projection
\begin{equation}\label{proj-to-wtbuw}
\varpi_{\wtbuw}:
\prod_{\buw \in \La^\sort_{d,[n]}} \PP_\buw \lra 
\prod_{{\rm supp}(\buw) \subset {\rm supp}(\wtbuw)} \PP_\buw = \prod_{\buw \in
\La_{d,I_\wtbuw}} \PP_\buw
\end{equation}
where the equality in the display is due to \eqref{iff}.
This gives rise to
\begin{equation}\label{d-2d}
\prod_{\buw \in \rladnsort}\varpi_{\wtbuw}: \; \prod_{\buw \in \La^\sort_{d,[n]}} \PP_\buw \lra
\prod_{\wtbuw \in \rladnsort} (\prod_{{\rm supp}(\buw) \subset {\rm supp}(\wtbuw)} \PP_\buw).
\end{equation}
Clearly, the morphism $\prod_{\buw \in \rladnsort}\varpi_{\wtbuw}$ is an 
 embedding.
Apply \eqref{7.0} and \eqref{pi-I} to $I=I_\wtbuw$, we see that
$\varpi_{\wtbuw}$ sends $\HH^{d, E}$ onto $\HH^{d, E_\wtbuw}$.
Hence, the morphism $\prod_{\buw \in \rladnsort}\varpi_{\wtbuw}$ sends
$\HH^{d, E}$ into $ \prod_{\wtbuw \in \rladnsort} \HH^{d,E_\wtbuw}.$
Because $\HH^{d,E_\wtbuw} \cong \HH^{d,2d}$,
the above implies the embedding 
$$\HH^{d, E}  \longrightarrow \prod_{\wtbuw \in \rladnsort} \HH^{d,E_\wtbuw}
  \cong (\HH^{d,2d})^{n \choose 2d}, $$
as stated in the proposition.
\end{proof}

\subsection{A stratification of $\HH^{d,E}$, $\HH^{2,E}$ and cubic binomials }

\subsubsection{A stratification of $\HH^{d,E}$ by multi-linear subspaces}
For any fixed $\buw \in  \Ladnsort$,
recall that $\LL_\buw$ is the linear subspace of $\PP_\buw$ defined by all the
 linearized $\pl$ relations of $\PP_\buw$. A linear subspace $L_\buw$
of $\LL_\buw$ is called a $\LL_\buw$-coordinate  subspace
if it is the intersection of $\LL_\buw$ with a coordinate  subspace of $\PP_\buw$.
Here, by a coordinate  subspace of $\PP_\buw$ with the homogeneous coordinates
$[x_{(\uu,\uv)}]_{(\uu, \uv) \in \La_\buw}$,
we mean an intersection of  the coordinate hyperplanes 
$$\{(x_{(\uu,\uv)}=0) \mid {(\uu, \uv) \in \La_\buw}\}.$$

Consider the multi-linear subspace  $\cL$ of $\prod_{\buw \in \La^\sort_{d,[n]}} \LL_\buw$ of the following form
$$\cL = \prod_{\buw \in \La^\sort_{d,[n]}} L_\buw,$$
where $L_\buw$  is a $\LL_\buw$-coordinate subspace of $\LL_\buw$,
called an $\LL$-coordinate subspace of
$\prod_{\buw \in \La^\sort_{d,[n]}} \LL_\buw$. 
Observe that an intersection of $\LL$-coordinate subspaces, if nonempty, is again an $\LL$-coordinate subspace.
We let 
$$\fL^\wdep=\{ \hbox{$\cL$ $\mid$ $\cL$ is a $\LL$-coordinate subspace of
$\prod_{\buw \in \La^\sort_{d,[n]}} \LL_\buw$} \}.$$
The set $\fL^\wdep$ is partially ordered by inclusion.
Take any $\cL \in \fL^\wdep$ and set
 $$\HH^\wdep_\cL = \{ \ba \in \HH^\wdep\mid \ba \in \cL \- \cL', \; \forall \; 
 \cL' \in \fL^\wdep \; \hbox{with}\; \cL' \subsetneqq \cL\}.$$
 Then, we obtain a stratification
 \begin{equation} \label{l-stra-0}
 \HH^\wdep=\bigsqcup_{\cL \in \fL^\wdep} \HH^\wdep_\cL .
 \end{equation}
We call \eqref{l-stra-0} the $\LL$-stratification of $\HH^\wdep$.

The stratification \eqref{l-stra-0} of $\HH^\wdep$ induces 
a stratification of $\HH^{d, E}$:
$$ \HH^{d, E}=\bigsqcup_{\cL \in \fL^\wdep} \HH^{d, E}\cap \HH^\wdep_\cL .$$
Put it differently, we let 
$$\fL^{d, E}=\{ \hbox{$\cL$ $\mid$ $\cL$ is a $\LL$-coordinate subspace of
$\prod_{\buw \in \La^\sort_{d,[n]}} \LL_\buw$} \mid \cL \cap \HH^{d, E}\ne \emptyset\}.$$
A member of $\fL^{d, E}$ is referred to as a relevant $\LL$-coordinate subspace.
The set $\fL^{d, E}$ is partially ordered by inclusion.
It gives rise to a stratification of $\HH^{d, E}$ as follows. Take any $\cL \in \fL^{d, E}$.
 $$\HH^{d, E}_\cL = \{ \ba \in \HH^{d, E}\mid \ba \in \cL \- \cL', \; \forall \; 
 \cL' \in \fL^{d, E} \; \hbox{with}\; \cL' \subsetneqq \cL\}.$$
 Then, we obtain a stratification
 \begin{equation} \label{l-stra}
 \HH^{d, E}=\bigsqcup_{\cL \in \fL^{d, E}} \HH^{d, E}_\cL ,
 \end{equation}
called the $\LL$-stratification. 
 The largest multi-linear subspace in $\fL^{d, E}$
is $\LL:=\prod_{\buw \in \La^\sort_{d,[n]}} \LL_\buw$ itself. 
It gives the open stratum $\HH^{d, E}_\LL=\UU^{d, E}/{\TT} .$

\begin{exam} 
The hypersimplex $\Delta^{2,[4]}$ is the convex hull of six lattice points
$$(1100), (0011); (1010),(0101); (1001), (0110) \in \ZZ^4 \subset {\mathbb R}^4.$$
There are exactly three matroid pavings given by
{\footnotesize
$$ 
P\{(1100); (1010),(0101); (1001), (0110) \}  \cup 
P\{(0011); (1010),(0101); (1001), (0110) \},$$
$$P\{(1100), (0011); (1010); (1001), (0110)\} \cup
P\{(1100), (0011); (0101); (1001), (0110) \},$$
$$P\{(1100), (0011); (1010),(0101); (1001)\} \cup P\{(1100), (0011); (1010),(0101); (0110)\},$$   }
where $P\{ \bullet\}$ stands for the convex hull generated by lattice points of $\bullet$.

\medskip


 
\bigskip

\begin{center}

\begin{tikzpicture}[thick,scale=3.1]
\coordinate (A1) at (0,0);
\coordinate (A2) at (0.6,0.2);
\coordinate (A3) at (1,0);
\coordinate (A4) at (0.4,-0.2);
\coordinate (B1) at (0.5,0.5);
\coordinate (B2) at (0.5,-0.5);

\begin{scope}[very thick,dashed,,opacity=0.6]
\draw (A1) -- (A2) -- (A3);
\draw (B1) -- (A2) -- (B2);
\end{scope}
\draw[fill=green!90!black,opacity=0.6] 
(A1) -- (A4) -- (B1);
\draw[fill=orange, opacity=0.6] 
(A1) -- (A4) -- (B2);
\draw[fill=green!90!black,opacity=0.6] 
(A3) -- (A4) -- (B1);
\draw[fill=orange, opacity=0.6] 
(A3) -- (A4) -- (B2);
\draw (B1) -- (A1) -- (B2) -- (A3) --cycle;
\draw [very thick, green!60!black] (B1) -- (A1);
\draw [very thick, green!60!black] (B1) -- (A3);
\draw [very thick, green!60!black ] (B1) -- (A4);
\draw [very thick, orange!60!black] (B2) - - (A1);
\draw [very thick, orange!60!black] (B2) - - (A3);
\draw [very thick, orange!60!black] (B2) - - (A4);

\coordinate (a1) at (1.3,0);
\coordinate (a2) at (1.9,0.2);
\coordinate (a3) at (2.3,0);
\coordinate (a4) at (1.7,-0.2);
\coordinate (b1) at (1.8,0.5);
\coordinate (b2) at (1.8,-0.5);

\begin{scope}[very thick,dashed,,opacity=0.6]
\draw (a1) -- (a2) -- (a3);
\draw (b1) -- (a2) -- (b2);
\end{scope}
\draw[fill=green!80!black,opacity=0.6] 
(a1) -- (a4) -- (b1); 
\draw[fill=green!80!black,opacity=0.6] 
(a1) -- (a4) -- (b2);
\draw[fill=orange!80!black, opacity=0.6] 
(a3) -- (a4) -- (b1);  \draw [thick, orange!90!] (a4) - - (a3);
\draw[fill=orange!80!black, opacity=0.6] 
 (a3) -- (a4) -- (b2);
\draw [very thick, green!60!black](b1) -- (a1) ;
\draw [very thick, orange!60!black](b1) - - (a3);
\draw [very thick, green!60!black]  (b2) -- (a1);
\draw [very thick, orange!60!black ](b2) -- (a3);

\draw [very thick, green!60!black] (a4) - - (a1);
\draw [very thick, orange!60!black] (a4) - - (a3);

\coordinate (c1) at (2.6,0);
\coordinate (c2) at (3.2,0.2);
\coordinate (c3) at (3.6,0);
\coordinate (c4) at (3,-0.2);
\coordinate (d1) at (3.1,0.5);
\coordinate (d2) at (3.1,-0.5);

\draw[fill=green!95!black,opacity=0.6] 
(c1) -- (c4) -- (d1);
\draw[fill=green!95!black,opacity=0.6] 
(c1) -- (c4) -- (d2);
\draw[fill=green!95!black,opacity=0.6] 
(c3) -- (c4) -- (d1);
\draw[fill=green!95!black,opacity=0.6] 
(c3) -- (c4) -- (d2);
\begin{scope}[very thick, red, dashed,,opacity=0.6]
\draw [very thick, red] (c1) -- (c2) -- (c3);
\draw [very thick, red] (d1) -- (c2) -- (d2);
\end{scope}
\begin{scope}[very thick, opacity=3]
\draw [thick, red!80!black](d1) -- (c1) -- (d2) -- (c3) --cycle;
\end{scope}
\draw [very thick, green!60!black] (c4) - - (c1); 
\draw [very thick, green!60!black] (c4) - - (c3);
 \draw [very thick, green!60!black] (c4) - - (d1); 
\draw [very thick, green!60!black] (c4) - - (d2); 
\end{tikzpicture}

\end{center}

\bigskip

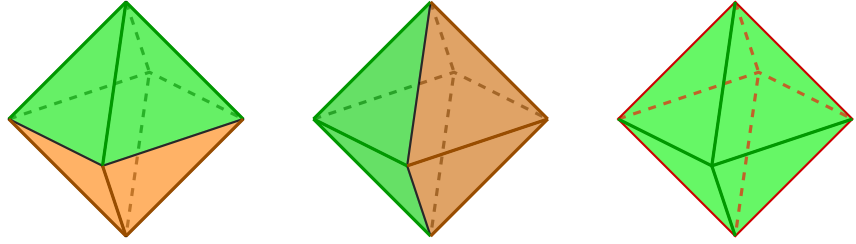
\centerline{Three matroid pavings of $\Delta^{2,[4]}$}
(For the third paving, the matroid subpolytope behind has orange color.)

\bigskip

One calculates and finds that $\HH^{2,4}$ embeds into
$\PP^2_{[x,y,z]}$ as the line $$x-y+z=0,$$
where $\hbox{$x=x_{(12,34)}$, $y=x_{(13,24)}$, and $z=x_{(14,23)}$}.$

\bigskip

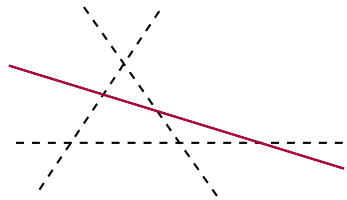
\begin{center}
\begin{tikzpicture}[thick,scale=3.1]
\coordinate (A1) at (-0.1,-0.1);
\coordinate (A2) at (0.43,0.69);
\coordinate (B1) at (-0.2, 0.1);
\coordinate (B2) at (1.2, 0.1);
\coordinate (C1) at (0.09,0.68);
\coordinate (C2) at (0.66,-0.13);
\coordinate (h1) at (-0.23,0.43);
\coordinate (h2) at (1.2, -0.01);

\draw (h1) - - (h2); 

\node[draw,align=left] at (2.3, 0.2) {\footnotesize \it  the three dashed lines \\
 \footnotesize \it are coordinate lines;\\ 
\footnotesize \it the solid line is $H^{2,4}$.};

\begin{scope}[thick,dashed,,opacity=2]
\draw (A1) -- (A2);
\draw (B1) - - (B2);
\draw (C1) - - (C2); 
\end{scope}

\draw[thick, purple](h1) - - (h2);
\end{tikzpicture}
\end{center}

\bigskip

The three intersection points $\ba$ of $\HH^{2,4}$ with the three coordinate lines of $\PP^2_{[x,y,z]}$,
$$\hbox{$[0,1,1],[1,0,-1],$ and $[1,1,0]$,}$$
are the only non-open $\LL$-strata of $\HH^{2,4}$, and, 
via the transformation $$\ba \lra \ff_\Gr(\fq_\Gr^{-1}(\ba)),$$
they precisely give rise to the above three matroid pavings of the hypersimplex $\Delta^{2,[4]}$, 
in the corresponding order, as listed.
\end{exam}

\subsubsection{Relations with matroid strata}

We let $\mathring{\Gr}^{d, E}=\{K \in \Gr^{d, E} \mid \dim \TT \cdot K =\dim \TT\}.$
Assume that $(d, n) > (2,4)$, lexicographically, and  $n - d>1$. Then, it should be useful to observe
the following: 
\begin{equation}\label{0isotropy-in-vJ} 
{\mathring\Gr}^{d, E} \cap V(J^{d, E}) \ne \emptyset.
\end{equation}

To prove  \eqref{0isotropy-in-vJ}, we can construct a point
in ${\mathring\Gr}^{d, E} \cap V(J^{d, E})$. 
We can assume $2d \le n$. Thus, $d < n-2$.
Hence, there exists a proper subset $I \subset [n] \- [4]$ such that $|I|=d-2$. When $d=2$, we let
$I=\emptyset$.

Then, we obtain the element 
$$\buw=((13I),(24I)) \in \Ladnsort.$$
The corresponding projective space $\PP_\buw$ is isomorphic to $\PP^2$ with the homogeneous
coordinates 
$$[x_{(12I),(34I)}, x_{(13I),(24I)}, x_{(14I),(23I)}].$$

We let $\bp \in \Gr^{d, E}$ be  any point such that
$$p_{12I}=p_{13I}=p_{14I}=0$$
but the remaining $\pl$ coordinates are nonzero.
Then, $$\bp \in  V(J_\buw^{d, E}) \subset V(J^{d, E}).$$
As $I \cup [4] \subsetneqq [n]$, there exists $k \in [n] \- I \cup [4]$. 
Suppose $\bt \in (\GG_m)^n$ fix the point $\bp$.
Then, we have
$$t_a t_k t_I =t_b t_k t_I, \; \forall \; a, b \notin I \cup \{k\}$$
because $p_{akI}p_{bkI} \ne 0$.
Hence $t_a=t_b$ for all $a, b \notin I \cup \{k\}$. Similarly,  because
$$t_2 t_k t_I=t_2 t_3 t_I, $$ we obtain $t_k=t_3$. 
Likewise, for any $i \in I$, because 
$$t_2 t_3  t_i t_{I \- i}=t_2 t_3 t_k t_{I\-i}, $$ we obtain $t_i=t_k$. 
Hence, $\bt \in \GG_m$.

\begin{rem}
Let $\ud$ any  matroid on $[n]$ [p1, \cite{La03}]. By [Proposition, p4, \cite{La03}], we have 
\begin{equation}\label{mtd-in-vJ} 
\Gr^{d, E}_\ud \cap V(J^{d, E_\bcd}) \ne \emptyset \iff \Gr^{d, E}_\ud \subset V(J^{d, E_\bcd}).
\end{equation} Equivalently, 
\begin{equation} \label{mtd-notin-vJ} \Gr^{d, E}_\ud \nsubseteq V(J^{d, E_\bcd})
\iff  \Gr^{d, E}_\ud \cap V(J^{d, E_\bcd}) =\emptyset.
\end{equation}
Also equivalently, $V(J^{d, E_\bcd})$ is a (disjoint) union of matroid strata.
\end{rem}


\begin{question}\label{relation with La03} 
For any matroid $\ud$ on the set $[n]$, 
Lafforgue  constructed a compactification 
of $\Gr^{d, E}_\ud/\TTb$ in {\rm Theorem,  pp. 7-8, \cite{La03}}.  
Suppose $$\hbox{$\Gr^{d, E}_\ud \nsubseteq V(J^{d, E_\bcd})$, \;
that is, \; $\Gr^{d, E}_\ud \cap V(J^{d, E_\bcd}) =\emptyset$,}$$ 
and $\Delta^{d, E_\bcd}_\ud$ (=$S_\RR$, p4, \cite{La03})
 has dimension $n-1$. 
Then, the rational map $\Theta^\bcd_\Gr$ of \eqref{theta-gr-intro} or \eqref{theta-Gr}
restricts to a morphism on $\Gr^{d, E}_\ud$
$$\Theta^\bcd_\Gr|_{\Gr^{d, E}_\ud}: \Gr^{d, E}_\ud \lra \HH^{d,E_\bcd},$$
and it descends to an embedding on the quotient
$$\bar\Theta^\bcd_\Gr|_{\Gr^{d, E}_\ud}: \Gr^{d, E}_\ud/\TTb \lra \HH^{d,E_\bcd}.$$
We let $\overline\Gr^{d, E_\bcd}_\ud/\TTb$ 
be the closure of the image of $\bar\Theta^\bcd_\Gr|_{\Gr^{d, E}_\ud}$.
This is also a compactification of $\Gr^{d, E}_\ud/\TTb$. How  does this relate to 
 Lafforgue's compactification?

When $d=2$, we were told that
$\HH^{2,E}$ coincides with the one
constructed in \cite{Fang2}, and in
[Theorem 1.4, \cite{Fang2}], it is
proved that $\HH^{2,E}$
 is isomorphic to the Lafforgue’s space over $\ZZ$.
\end{question}

\medskip

\subsubsection{$\HH^{2,E}$ and cubic binomials}\label{2,E}

The special case of  $\Gr^{2,E}$ might worth some additional comments.
We make some observations and raise a concrete question. 

One first observes and checks that 
\begin{equation}\label{Pw=PF}
\La^\sort_{2,[n]}=\{ \buw=((ik), (jl)) \mid 1 \le i <j <k <l \le n \}, 
\end{equation}
$$\Lambda_\buw =
\{(ij,kl), (ik,jl), (il,kj)\}, \;\; \forall \;\; \buw=((ik), (jl)).$$
Thus,  $\PP_\buw$ comes equipped with the homogeneous coordinates
$[x_{(hi,jk)},x_{(hj,ik)},x_{(hk,ij)}]$.  In particular, every $\PP_{\buw}$ is isomorphic to $\PP^2$, and,
$$\hbox{$\LL_\buw$ is isomorphic to $\PP^1$,} $$ embedded in  
$\PP_\buw$, as the line
 defined by $$x_{(ij,kl)}- x_{(ik, jl)} + x_{(il,kj)}=0,$$
which is the linearized $\pl$ relation induced from the $\pl$ relation 
\begin{equation}\label{F=buw}
p_{ij} p_{kl}- p_{ik}p_{ jl} + p_{il}p_{kj},
\end{equation}
which, in turn, is uniquely determined by $\buw=((ik), (jl))$.
 
Hence, by Theorem \ref{thm1}, we have
$$\HH^{2, E} \subset \prod_{\buw \in \La^\sort_{2,n}} \LL_\buw \;\;
(\cong (\PP^1)^{n \choose 4}) \; \subset \prod_{1\le i<j<k<l \le n} \PP_{((ik),(jl))}
\;\; (\cong (\PP^2)^{n \choose 4}).$$

 As a particular case, take  $n=4$, 
one  finds that $\HH^{2,4}$ is embeded into  $\PP^2_{[x,y,z]}$ as the line $x-y+z=0$,
where $x=x_{(12,34)}$, $y=x_{(13,24)}$, and $z=x_{(14,23)}$.

Now, we consider equations of the $\fG$-quotient $\HH^{2,E}$
in $\prod_{\buw \in \La^\sort_{2,n}} \LL_\buw$.

One checks directly that the following (multi-homogeneous) cubic binomials, 
\begin{equation}\label{cubic-b}
x_{(hi,jk)}x_{(hk,jl)}x_{(hl,ij)}-x_{(hk,ij)}x_{(hl,jk)}x_{(hi,jl)}, \; 1\le h <i < j <k<l \le n
\end{equation} 
belong to the toric part of the ideal for $\HH^{2,E}$
in $\prod_{1\le a<b<c<d \le n} \PP_{((ac),(bd))},$ besides the linear 
and $\vr$-$\pl$ relations.

\begin{conj}\label{conj:2E}
The toric $\fG$-quotient $\HH^{(\wedge^2 E)}$, as a closed subscheme of 
$$\prod_{1\le a<b<c<d \le n} \PP_{((ac),(bd))} \cong (\PP^2)^{n \choose 4},$$
is defined by the cubic binomial relations 
$$x_{(hi,jk)}x_{(hk,jl)}x_{(hl,ij)}-x_{(hi,jl)}x_{(hk,ij)}x_{(hl,jk)}$$
for all $1\le h <i < j <k<l \le n$.
 \end{conj} 

Then, by Corollary \ref{H(2E)},
assuming Conjecture \ref{conj:2E},
 it would follow
that the $\fG$-quotient $\HH^{2,E}$, as  a closed subscheme of
$\prod_{1\le a<b<c<d\le n} \LL_{((ac),(bd))}
 \cong (\PP^1)^{n \choose 4},$ 
is defined by the cubic binomial relations in Conjecture \ref{conj:2E},
at least when the base field has characteristic zero.

As an evidence to Conjecture \ref{conj:2E}, consider the following quintic binomial relation 
$$g: x_{(12,34)}x_{(14,35)}x_{(15,26)}x_{(16,23)}x_{(13,25)}
-x_{(14,23)}x_{(15,34)}x_{(16,25)}x_{(13,26)}x_{(12,35)}.$$
One checks that this is an indeed a relation for $\HH^{2,E}$ when $\dim E \ge 6$.
But, we also have the following cubic binomial relations 
$f: x_{(12,34)}x_{(14,35)}x_{(15,23)}-x_{(14,23)}x_{(15,34)}x_{(12,35)}.$ 
Hence, modulo $f$, we obtain
$$g\equiv  x_{(12,34)}x_{(14,35)}(x_{(15,26)}x_{(16,23)}x_{(13,25)}-x_{(16,25)}x_{(13,26)}x_{(15,23)}),
\mod (f).$$
Observe that $x_{(15,26)}x_{(16,23)}x_{(13,25)}-x_{(16,25)}x_{(13,26)}x_{(15,23)}$
is cubic binomial  of \eqref{cubic-b}.

To close this section, we claim that  {\it given any of the cubic  binomials as in \eqref{cubic-b}, 
there exists at least one term such that it contains two nonzero factors. }
To verify this,  consider \eqref{cubic-b}.
Suppose first that $x_{(hi,jk)}x_{(hk,jl)} \ne 0$, then the claim holds.
Suppose now that $x_{(hi,jk)}=x_{(hk,jl)}=0$, then due to the linearized $\pl$ equations
$$x_{(hi,jk)}-x_{(hj,ik)}+x_{(hk,ij)}=0 \;\; \hbox{and} \;\; x_{(hj,kl)}-x_{(hk,jl)}+x_{(hl,jk)}=0,$$
we must have $x_{(hk,ij)}x_{(hl,jk)} \ne 0$. This proves the claim.

Assuming that Conjecture \ref{conj:2E} holds,
then, the above observation implies that 
$\HH^{2,E}$ is regular over any field by directly applying the Jacobian criterion.
 This would directly prove the smoothness of $\HH^{2,E}$ over any perfect field without 
 borrowing  the isomorphism 
$\HH^{2,E}  \cong \overline{M}_{0,n}$ (\cite{Kap93b}).

\section{Appendix. Inducing and extending term orders}\label{appendix}

Consider the polynomial ring $\kk[x_1,\cdots,x_n]$.
The monomials in $\kk[\bx]$, written as $\bx^\ba=x_1^{a_1}x_2^{a_2}\cdots x_n^{a_n}$,
can be identified with $\ba=(a_1,\cdots, a_n) \in \NN^n$ where $\NN$ is the set of all nonnegative
integers. A term order $\prec$ on the set of monomials of  $\kk[\bx]$ 
is (equivalent to) a total order on $\NN^n$ such that
the zero vector is the unique minimal element, and $\ba \prec {\bf b}$ implies
$\ba + {\bf c} \prec {\bf b} + {\bf c}$ for all $\ba, {\bf b}, {\bf c} \in \NN^n$.

Given a term order $\prec$ on the polynomial ring $\kk[\bx]=\kk[x_1,\cdots,x_n]$,
every non-zero polynomial $f \in \kk[\bx]$ has a unique initial monomial, denoted $in_\prec (f)$.
 If $I$ is an ideal of the polynomial ring $\kk[\bx]$,
 then its initial ideal is the monomial ideal 
 $$in_\prec (I) :=\langle in_\prec (f) \mid f \in I \rangle.$$

A finite set $\fG \subset I$ is a Gr\"obner basis of the ideal $I$ with respect to the term order $\prec$ if
the initial ideal $in_\prec (I)$ is generated by $\{in_\prec (g) \mid g \in \fG\}$.

\begin{defn}\label{reduced-grob}
A Gr\"obner basis $\fG$ of the ideal $I$ under a given term order $\prec$ is minimal 
 if all leading monomials of its elements are irreducible by the other elements of the basis;
it is reduced if for any two distinct elements $g, g' \in \fG$, 
no term of $g'$ is divisible by $in_\prec (g)$.
\end{defn} 

The reduced Gr\"obner basis is unique if in addition, 
 we require that {\it the coefficient of $in_\prec (g)$ equals 1 for all $g \in \fG$}.


A monomial is a $\prec$-standard monomial modulo  the ideal $I$ if it does not lie
in the initial ideal $in_\prec (I)$. Equivalently,  a monomial is  $\prec$-standard
 modulo $I$ if it is not a multiple of any leading monomial in the Gr\"obner basis. 
The images of all the standard monomials form
a $\kk$-vector space basis for the quotient ring $\kk[x_1,\cdots,x_n]/I$.

 \begin{lemma}\label{s-pair} Let $g_1=f_1+h_1$ and $g_2=f_2+h_2$ be binomials with $f_1, h_1, f_2, h_2$ 
 being monomials. Suppose $in_\prec (g_i)=f_i, \; i=1,2$. Assume $f_1$ and $f_2$ are co-prime.
 Then, for the $S$-polynomial $S(g_1,g_2)=f_2 g_1 - f_1 g_2$, a complete reduction 
 of $S(g_1,g_2)$ by $\{g_1, g_2\}$ produces zero.
 \end{lemma}
 \begin{proof}
First, $$S(g_1,g_2)=f_2 g_1 - f_1 g_2=f_2h_1 - f_1 h_2.$$
 Then, we have
 $$
 \xymatrix{
 S(g_1,g_2) \ar[r]^{h_2g_1 \hskip 2cm} &  S(g_1,g_2)+ h_2g_1 =f_2h_1 + h_2h_1
 \ar[r]^{-h_1g_2 \hskip -.3cm} & f_2h_1 + h_2h_1- h_1 g_2=0. }$$
 \end{proof}

Consider a polynomial ring $R_\bx=\kk[x_1,\cdots, x_n]$ with $\bx=(x_1,\cdots, x_n)$.
For any $i \in [n]$, suppose we have
 a polynomial ring $R_{\by_i}=\kk[y_{i1},\cdots, y_{ik_i}]$ for some positive integer $k_i$ (depending on $i$)
 with $\by_i=(y_{i1},\cdots, y_{ik_i})$.
Then, we have the polynomial ring 
 $$R_\by =\kk[\by_1,\cdots, \by_n]=\kk[y_{11},\cdots, y_{1k_1}, \cdots, y_{n1},\cdots, y_{nk_n}].$$
 There is a natural homomorphism 
 $$\ze: \; R_\by \lra R_\bx, \;\;\;\;
 y_{ij} \to x_i, \;\;\;\; i \in [n], \; j \in [k_i].$$

Write $\ba_i=(a_{i1},\cdots, a_{ik_i}) \in \NN^{k_i},  \ba=(\ba_1, \dots, \ba_n) \in \NN^k$
where $k=\sum_{i=1}^n k_i$. Then, any monomial of $R_\by$ takes the form
 $$\by^\ba= \by_1^{\ba_1} \cdots  \by_n^{\ba_n} .$$
Write $|\ba_i|=\sum_{j=1}^{k_i}a_{ij}$. Then, $\ze$ takes the monomial $\by^\ba$ to its image
$$\bx_\ba= x_1^{|\ba_1|} \cdots  x_n^{|\ba_n|} .$$

 Suppose we have a term order $\prec_\bx$ on a polynomial ring $R_\bx=\kk[x_1,\cdots, x_n]$.
 For any $i \in [n]$, we let $\prec_i$ be any fixed term order on 
 a polynomial ring $R_{\by_i}=\kk[y_{i1},\cdots, y_{ik_i}]$.
 
 We can introduce a total order $\prec_\by$ on the set of monomials of $R_\by$ 
 induced by the orders $\prec_\bx$ and $\prec_i$ ($i \in [n]$) as follows.
 Fix any two distinct monomial $\by^\ba$ and $\by^\bb$ of $R_\by$ for some $\ba$ and $\bb$ 
 in $\NN^k$ as above.  We say
 \begin{equation}\label{order-y}
 \by^\ba \prec_\by \by^\bb
 \end{equation}
 if one of the following two holds:
 \begin{enumerate}
 \item $\bx_\ba \prec_\bx \bx_\bb$, that is,
  $x_1^{|\ba_1|} \cdots  x_n^{|\ba_n|}  \prec_\bx x_1^{|\bb_1|} \cdots  x_n^{|\bb_n|} $;
 \item $\bx_\ba = \bx_\bb$ and 
 $\by_i^{\ba_i}=\by_i^{\bb_i}, 1 \le i \le j-1$ but $\by_j^{\ba_j} \prec_j \by_j^{\bb_j}$
 for some $j\in [n]$.
 \end{enumerate}

 \begin{lemma}\label{ext-order}
  The  order $\prec_\by$ is a term order on $R_\by$.
 \end{lemma}
 \begin{proof}
 First, one inspects directly that $\prec_\by$ is a well defined total order.
 
 Clearly, $1=\by^{\bf 0}$ is the unique minimal element among all monomials of of $R_\by$.
 
 Thus, it suffices to show that $\by^\ba \prec_\by \by^\bb$ implies 
 $\by^{\ba+\bc} \prec_\by \by^{\bb +\bc}$ for all $\ba, \bb, \bc \in \NN^k$.
 
 Suppose $\bx_\ba \prec_\bx \bx_\bb$. Then, because $\prec_\bx$ is a term order,
 we have
 $$x_1^{|\ba_1|+|\bc_1|}  \cdots  x_n^{|\ba_n| +|\bc_n|} \prec_\bx  
 x_1^{|\bb_1|+|\bc|_1} \cdots  x_n^{|\bb_n|+|\bc_n|} .$$
 As $|\ba_i| +|\bc_i|=|\ba_i +\bc_i|$ for all $i \in [n]$, we obtain
 $$\bx_{\ba+\bc} \prec \bx_{\bb +\bc}.$$
Hence, $\by^{\ba+\bc} \prec \by^{\bb +\bc}$ by definition.

 Suppose $\bx_\ba = \bx_\bb$  and 
 $\by_i^{\ba_i}=\by_i^{\bb_i}, 1 \le i \le j-1$ but $\by_j^{\ba_j} \prec_j \by_j^{\bb_j}$.
 Then, we obtain $\bx_{\ba+\bc} = \bx_{\bb+\bc}$ and 
 $\by_i^{\ba_i +\bc_i}=\by_i^{\bb_i +\bc_i}, 1 \le i \le j-1$ but 
 $\by_j^{\ba_j+\bc_j} \prec_j \by_j^{\bb_j +\bc_j}$. This also implies
 $\by^{\ba+\bc} \prec \by^{\bb +\bc}$ by definition.
 Hence, the lemma is proved.
 \end{proof}

 \smallskip\noindent
{\bf Acknowledgements.}
The author thanks Matthew Baker 
for explaining to him that the multi-exchange $\pl$ relations (\cite{Abeasis})
should define the Grassmannian over any field. He also thanks Hanlong Fang for pointing out a mistake in an earlier version
 related to the definition \eqref{La-buw}.
He thanks J\'anos Koll\'ar for email correspondences 
 inspiring this version of the article.
 The final revision of this paper was completed when the author visited  
the Simons Center for Geometry and Physics from January 19 - 31, 2025.
He is deeply grateful to the organizers of the program 
{\it Recent Developments in Higher Genus Curve Counting}
 (January 6 – February 28, 2025) at SCGP, especially Melissa Liu, for their efforts in creating such a stimulating environment. 
He would also like to extend his gratitude to the Simons Center for its generous support and for providing an excellent environment during his two-week stay.
The author declares no conflict of interest.

\end{document}